\documentclass[10pt,reqno,a4paper]{amsart}

\usepackage{amsmath,amsrefs,amssymb,tensor, eufrak, dsfont}
\usepackage{graphicx}
\usepackage{xcolor}
\usepackage{tikz-cd}
\usepackage{bigints}

\usepackage[colorlinks=true,linkcolor=red,citecolor=blue,urlcolor=teal,bookmarks=false,hypertexnames=true,pdfborder={0 0 0}]{hyperref}
\usepackage{cleveref}
\usepackage{mathrsfs}

\numberwithin{equation}{section}

\DeclareMathOperator{\Scal}{S}
\DeclareMathOperator{\Ricci}{Ric}

\DeclareMathOperator{\Weyl}{W}

\DeclareMathOperator{\Sym}{Sym}
\DeclareMathOperator{\Vol}{Vol}

\DeclareMathOperator{\bigO}{O}
\DeclareMathOperator{\smallo}{o}

\newcommand{\R}{\mathbb{R}}
\renewcommand{\S}{\mathbb{S}}
\newcommand{\N}{\mathbb{N}}
\newcommand{\Z}{\mathbb{Z}}

\renewcommand{\[}{\left[}
\renewcommand{\]}{\right]}
\renewcommand{\(}{\left(}
\renewcommand{\)}{\right)}

\newtheorem{theorem}{Theorem}[section]
\newtheorem{proposition}{Proposition}[section]
\newtheorem{lemma}{Lemma}[section]

\newtheorem{remark}{Remark}[section]
\newtheorem{corollary}{Corollary}[section]

\begin{document}

\title[Barycenter technique for GJMS]{Barycenter technique for the higher order $Q$-curvature equation}

\author{Saikat Mazumdar}
\address{Saikat Mazumdar, Department of Mathematics, Indian Institute of Technology, Bombay, Powai, Mumbai, Maharashtra 400076, India}
\email{saikat.mazumdar@iitb.ac.in, saikat@math.iitb.ac.in}

\author{Cheikh Birahim Ndiaye}
\address{Cheikh Birahim Ndiaye, Department of Mathematics, Howard University, Annex 3, Graduate School of Arts and Sciences DC 20059 Washington, USA}
\email{cheikh.ndiaye@howard.edu}

\date{March 4, 2026}

\begin{abstract}
Let $k\ge1$ be an integer, and $\(M,g\)$ be a smooth, closed Riemannian manifold of dimension $2k+1\le n\le 2k+3$, or $(M,g) $ be locally conformally flat of dimension $n\ge 2k+1$.  Applying the Bahri–Coron barycenter method, we show the existence of a conformal metric with constant  $Q$-curvature of order $2k$, or equivalently, the existence of a positive solution for the $2k$-th order $Q$-curvature equation involving the GJMS operator $P_{g}$. We only assume a natural positivity preserving condition on $P_{g}$ and do not suppose any condition on the sign of the {\emph{mass}} of $P_{g}$. In particular, we obtain existence without using a positive mass theorem. 
\end{abstract}

\maketitle

\section{Introduction}
Let $k\geq 1$ be a positive integer and $(M,g)$ be a closed smooth Riemannian manifold of dimension $n \geq 2k+1$. We consider the conformally covariant GJMS operator, denoted by $P_{g}$, introduced by Graham, Jenne, Mason and Sparling in the seminal work \cite{GJMS}. $P_{g}$ is an elliptic self-adjoint differential operator with the leading term $\Delta^k_{g}$, where $\Delta_{g}= -\text{div}_g(\nabla \cdot)$ is the Laplace--Beltrami operator. The operator $P_{g}$ is conformally covariant, that is,  defining $\widehat{g}:= u^{\frac{4}{n-2k}}g$ for $u \in C^\infty(M),\,u>0$, we have the relation
\begin{equation} \label{conf.inv.Pg} 
P_{\widehat g} (f)= u^{- \frac{n+2k}{n-2k}} P_g(u f)~\hbox{ for all } f \in C^\infty(M).
\end{equation}
The construction of $P_{g}$ is based on the Fefferman and Graham ambient metric \cite{FefGra1, FefGra2}. When $k=1$, $P_g$ is the  conformal laplacian $\Delta_g + \frac{n-2}{4(n-1)} \Scal_g,$ with $\Scal_g$ denoting the scalar curvature of $(M,g)$. For $k=2$, $P_g$ is the Paneitz-Branson operator \cite{Branson1, Paneitz}.  Also see Graham and Zworski \cite{GZ} for a related construction via scattering theory. Explicit formulas for $P_{g}$ are known only in some low-order cases, see Branson \cite{Branson1,Branson2,Branson3}, Paneitz \cite{Paneitz}, Gover and Peterson \cite{GovPete}, and Juhl \cite{JuhlGJMS}. In general, universal algebraic formulas expressing $P_{g}$ in terms of second-order building blocks were obtained by Juhl \cite{JuhlGJMS} (also see Fefferman and Graham \cite{FefGra3}). However, in special cases such as Einstein manifolds, $P_{g}$ factorizes as a product of second-order operators of the form $\Delta_{g}+\gamma_{i,n}\Scal_{g}$, with $\gamma_{i,n}=\frac{(n-2i)(n+2i-2)}{4n(n-1)}$, $1\le i\le k$ (see Gover \cite{Gov} and Fefferman-Graham \cite{FefGra2}). Recently, Case and Malchiodi \cite{CaseMal} obtained factorizations for special Einstein product manifolds.
\medskip

A scalar valued curvature quantity called the $Q$-curvature of order $2k$, denoted by $Q_{g}$, is associated with the GJMS operator $P_{g}$ and is defined via the zeroth-order terms of $P_{g}$ as $Q_{g}:=\frac{2}{n-2k}P_{g}(1)$.  When $k=1$, $Q_{g}=\Scal_{g}$ (up to a positive constant), and for $k=2$, $Q_{g}$ is the Branson $Q$-curvature (introduced in Branson and \O rsted \cite{BransonOrs} and generalized by Branson \cite{Branson3}). For more details regarding the $Q$-curvature and its significance, we refer to Juhl \cite{JuhlBook}, Case and Gover \cite{CaseGov} and the references therein.
\medskip

In this paper, we consider the question of the existence of a conformal metric with constant positive $Q_{g}$ curvature on $(M,g)$. By the conformal transformation rule \eqref{conf.inv.Pg}, this is equivalent to showing the existence of a smooth positive solution for the $2k$-th order $Q$-curvature equation:
\begin{align}\label{eq:one}
P_{g}u=\lambda\, u^{\frac{n+2k}{n-2k}}~\hbox{ in } M, \hbox{ for some constant }\lambda>0. 
\end{align}
Thus if $u\in C^{\infty}(M)$, $u>0$ satisfies eq. $\eqref{eq:one}$, then $Q_{u^{\frac{4}{n-2k}}g}\equiv\, \text{constant}>0$.  We define the corresponding Yamabe invariant as
\begin{align}\label{Yamabe invariant}
Y_{k}(M,g):=\inf_{\widehat{g}\in\[g\]}\frac{\displaystyle\int_{M}Q_{\widehat{g}}\,dv_{\widehat{g}}}{\Vol_{\,\widehat{g}}\(M\)^{\frac{n-2k}{n}}}=\inf_{\substack{u\in C^\infty\(M\),\\u>0\text{ in }M}}\frac{\displaystyle\int_{M}uP_{g}u\,dv_g}{\displaystyle\(\int_{M}u^{2^*_k}\,dv_g\)^{\frac{n-2k}{n}}}.
\end{align}
Here $\[g\]$ denotes the conformal class of $g$ and $2^{*}_{k}:=2n/\(n-2k\)$ is the critical Sobolev exponent. It follows that $Y_{k}(M,g)$ is conformally invariant. 
\medskip

To obtain positive solutions, we will assume throughout the following {\emph{positivity condition}}:
\begin{align}\label{positivity}
\text{$Y_{k}(M,g)>0$ and the Green's function $G_{g}$ of $P_{g}$ is positive in $M$.}
\end{align}
The Green's function of $P_g$ is then the unique function $G_g \in C^\infty(M \times M \backslash \{x=y\})$ satisfying $P_g G_g(x, \cdot) = \delta_x$ for all $x \in M$, and $G_g$ is positive if $G_g(x,y) >0$ for all $x \neq y$.  Assumption \eqref{positivity} implies that $P_g$ is coercive and satisfies the maximum principle, that is,  if $u\geq 0$ satisfies $P_{g}u\ge0$ in $M$ then either $u>0$ or $u\equiv 0$ in $M$. Condition \eqref{positivity} is satisfied for the standard round sphere $\S^{n}$, its quotients and closed Einstein manifolds with positive scalar curvature. Recently, Case and Malchiodi in \cite{CaseMal} have obtained examples of special Einstein product manifolds for which the  Green's function of $P_{g}$ is positive. For the case $k=2$, the assumptions $Y_{1}(M,g)>0$ and $Q_g \geq 0$ on $M$, $Q_g\not\equiv 0$, imply that $\text{Ker}(P_{g})=\{0\}$ and the Green's function $G_{g}$ is positive, as shown  in Hang and Yang \cite{HangYang1, HangYang2} and Gursky and Malchiodi \cite{GurMal}. Also see Gursky, Hang, and Lin \cite{GurHangLin}. For the general case of $k\ge1$, $Y_{k}(M,g)>0$ implies  $\text{Ker}(P_{g})=\{0\}$, but geometric assumptions on $(M,g)$ that would ensure condition \eqref{positivity} remain unknown, to the best of our knowledge. 
\medskip

For all $1\le k<n/2$, the existence of positive smooth solutions for eq. $\eqref{eq:one}$ was established in Mazumdar and V\'etois \cite{MazVetois} assuming the positivity condition \eqref{positivity}. In addition, a positive mass type result was assumed in \cite{MazVetois} to obtain existence when  $2k+1\le n\le 2k+3$ or $(M,g) $ is locally conformally flat, similar to $k=1,2$.  One may think of the mass at $\xi\in M$ as the constant term in the expansion of  $G_{g}(\xi, \cdot)$ in conformal normal coordinates around $\xi$ (see \eqref{expan.green2}). Proving a positive mass type theorem for general $k\geq 1$, in the spirit of the seminal work of Schoen and Yau \cite{SchoenYau}, remains elusive, and this was recently established for the case $k=2$ in  Avalos, Laurain and Lira \cite{ALL}, and Gong, Kim and Wei \cite{GongKimWei}. Earlier results in this direction, for the case $k=2$, were obtained in Humbert and Raulot \cite{HumRau}, and \cite{GurMal,HangYang1}.  In this paper, we show that a positive mass type result is not necessary to obtain the existence of a solution to eq. \eqref{eq:one}. Our main result is the following theorem.
\begin{theorem}\label{thm:main}
Let $k\ge1$ be an integer, and $\(M,g\)$ be a smooth, closed Riemannian manifold of dimension $2k+1\le n\le 2k+3$ or $(M,g) $ be locally conformally flat. Assume that the $2k$th order GJMS operator $P_{g}$ satisfies the positivity condition \eqref{positivity}. Then  eq. \eqref{eq:one} has a smooth positive solution for some $\lambda>0$, or equivalently,  there exists a metric conformal to $g$ with constant positive $Q$-curvature of order $2k$.
\end{theorem}
\begin{remark}
Note that we do not assume any condition on the sign of the {\emph{mass}} of $P_{g}$. The solutions or the metrics obtained by Theorem \ref{thm:main} are not necessarily minimizing. 
\end{remark}
\medskip


Approaching the problem in the spirit of the method introduced by Yamabe \cite{Yamabe}, Aubin \cite{Aubin} and Schoen \cite{Schoen} for the case of the Yamabe problem, one attempts to show the existence of minimisers of the corresponding energy functional. Now, a non-compact minimizing sequence develops a spherical bubble at a point on the manifold $M$ and thus a strict inequality $Y_{k}(M,g)<Y_{k}(\mathbb{S}^{n})$ will imply the existence of a solution of eq. \eqref{eq:one}. To show the strict inequality, one derives an energy expansion for a suitable family of test functions depending on a parameter. For $n\le2k+3$ or $(M,g)$ locally conformally flat,  the test functions are chosen to be the Sobolev extremals localized on the manifold glued to the Green's function with the parameter.  The positivity of the mass then ensures the strict inequality, thus proving existence. This program was carried out for general $k\ge1$ in \cite{MazVetois}, but the positivity of mass still remains out of reach, to the best of our knowledge. To show existence without a positive mass type result, we use the novel technique introduced and pioneered by Bahri \cite{Bahri}, Bahri and Coron \cite{BahriCoron}, and more recently, used by Mayer and Ndiaye \cite{MayNdi1} to resolve the remaining cases of the Cherrier-Escobar problem. Also see Bahri-Brezis \cite{BahriBrezis} for the case of Yamabe-type problems.  We remark that when $(M,g)$   is locally conformally flat, existence for eq. \eqref{eq:one} was first obtained  by Qing and Raske \cite{QingRas}, under the assumption that the Poincar\'e exponent of the holonomy representation of the fundamental group $\pi_{1}(M)$ is less than $\frac{n-2k}{2}$. Another existence result was given by Robert \cite{RobertGJMS} for the prescribed $Q$ curvature problem. 
\medskip

For the case $n=2k$,  which involves an exponential non-linearity, the existence of constant Q-curvature metrics in a given conformal class was shown using min-max techniques, Morse theoretic and algebraic topological arguments in Djadli and Malchiodi \cite{DjaMal} and Ndiaye \cite{Ndiaye1} for the case $k=2$, and Ndiaye \cite{Ndiaye2, Ndiaye3}  for all $n=2k$. Also see Brendle \cite{BreQ} (for a proof via the gradient flow), Baird, Fardoun
and Regbaoui \cite{BFR}, and Li, Li and Liu \cite{LLL}. A similar existence result for the fractional Yamabe problem on the conformal infinity of a Poincar\'e-Einstein manifold was obtained using the barycenter technique of Bahri-Coron in Ndiaye, Sire and Sun \cite{NSS}, and  Mayer and Ndiaye  \cite{MayNdi2}. Another problem where the method of Bahri-Coron was successfully applied is the global case of the CR Yamabe problem, as shown by Gamara \cite{Gamara} and Gamara and Yacoub \cite{GamYac}. 
\medskip

We now briefly discuss the steps of the proof of Theorem \ref{thm:main}.  We focus on the main ideas and will be flexible with our notations.  Proceeding by contradiction, we suppose that the energy functional 
$$\big\{ u\ge0 \text{ a.e.} : \|u\|_{L^{2^{*}_{k}}}=1\big\}\ni u\longmapsto\int_{M}u\,P_{g}u\,dv_{g}=:\mathcal{J}_{g,k}(u)$$ has no critical points (see Section \ref{sec1}). Let $\mathscr{W}_{c}:=\{u: \mathcal{J}_{g,k}(u)\leq c\}$ denote energy sublevels. Consider the barycenter $\mathscr{B}_{d}(M)$ of order $d\in\N$, which is the set of the convex combinations of Dirac deltas $\delta_{\xi_{i}}$ with $\xi_{i}\in M$ for  $1\le i\le d$. Via the global bubbles $V_{\xi,\mu}$,  $\xi\in M$, $\mu>0$ sufficiently small, defined in Section \ref{sec:Bubbles} we map the barycenter $\mathscr{B}_{d}(M)$ to the energy sublevel $\mathscr{W}_{c^{\star}_{d}}$, where $c^{\star}_{d}:=(d+1)^{2k/n}Y_{k}(\S^{n})$, for all $d\in\N$. By a Struwe-type decomposition, Palais-Smale sequences for $\mathcal{J}_{g,k}$ can be written as a sum of finitely many bubbles $\sum V_{\xi_{i},\mu_{i}}$ satisfying a certain configuration, plus a small error. The quantization of the energy levels implies that for every $d\ge 1$, the sublevel  $\mathscr{W}_{c^{\star}_{d}}$ deformation retracts to $\mathscr{W}_{c^{\star}_{d-1}+\varepsilon}$ for any $\varepsilon>0$ small. Moreover, we can project  $\mathscr{W}_{c^{\star}_{d-1}+\varepsilon}\setminus \mathscr{W}_{c^{\star}_{d-1}}$ back to $\mathscr{B}_{d}(M)$.

As $M$ is a closed manifold, it has a non-trivial topology ($\mathbb{Z}_{2}$ homology), and using algebraic topological arguments, we have the existence of a non-trivial homology of $\mathscr{B}_{d}(M)$. For $d=1$, this is the orientation class of $M$. Using the  cone structure, we thus obtain that the mapping between the pairs 
$$\mathscr{F}_{d}(\mu):\,\big(\mathscr{B}_{d}(M),\mathscr{B}_{d-1}(M)\big)\longrightarrow \big(\mathscr{W}_{c^{\star}_{d}}, \mathscr{W}_{c^{\star}_{d-1}}\big)$$ is always homologically nontrivial for all $d\in\N$. This procedure is described in Section \ref{sec:BahriCoron}. We remark that this part of the argument is generally applicable to a wide class of problems of a similar nature and does not use the exact equation, the dimension restriction, or even the geometry of the manifold. The only ingredients are: the deformation lemma, the energy quantization phenomenon, and the existence of a selection map, which in our setting is a consequence of the Struwe-type bubble tree decomposition. 

Coming back to the bubbles, in Section \ref{sec:Bubbles} we obtain a precise estimate on the error: $P_{g}V_{\xi,\mu}-V_{\xi,\mu}^{2^{*}_{k}}$ using the Juhl's expansion of the GJMS operator $P_{g}$. We then use this to obtain interaction estimates between bubbles in Section \ref{sec:Interaction}.  In Section \ref{sec:Energy}, using the interaction estimates we control the energy of a sum of bubbles and show that there exists $d_{\star}$ large such that 
$$\mathcal{J}_{g,k}\big(\sum_{i=1}^{d} V_{\xi_{i},\mu_{i}}\big)<d^{\,2k/n}\,Y_{k}(\S^{n})~\hbox{ for } d\ge d_{\star}.$$ 
This is the core of the paper, and these estimates are of independent interest. \emph{The interactions between the bubbles push the energy of the sum below the sum of the energies of the bubbles} and we need globally defined bubbles for the interaction phenomena to take effect.  But then this means that the map $\mathscr{F}_{d}(\mu)$ is homologically trivial for  $d\ge d_{\star}$, a contradiction.  The proof of Theorem \ref{thm:main} is completed in Section \ref{sec:BahriCoron}. 
The Green's function of $P_{g}$ plays a crucial role in the above analysis, and we exploit the nonlinearity to obtain the above strict inequality. In this method, we exploit the non-compactness to obtain a topological contradiction.  Note that gaining existence from the above inequality is reminiscent of the Aubin-type criterion for $d=1$ case. 
\medskip

The compactness problem for eq.\eqref{eq:one} for general $k\ge1$ was recently studied in Mazumdar and Premoselli \cite{MazPrem}. In general, one cannot eliminate non-compactness without additional assumptions, such as the positivity of mass.  We refer to Robert \cite{RobPoly1, RobPoly2} for a blow-up analysis for GJMS-type operators, and the forthcoming work of Carletti and Premoselli \cite{CarPrem} for compactness results for finite-energy solutions on open sets in $\R^{n}$. The problem of extremizing eigenvalues in a conformal class for GJMS operators was recently studied by Humbert, Petrides and Premoselli \cite{HumPetPre}. 
\medskip

The paper is organised as follows. In Section \ref{sec1}, we recall the variational set-up and some established results needed for our proof. In Section \ref{sec:Bubbles}, we define our bubbles and obtain the crucial error estimate. We calculate the interaction estimates in  Section \ref{sec:Interaction}, which is then used to estimate the energy of the sum of bubbles in Section \ref{sec:Energy}. In Section \ref{sec:BahriCoron}, we present the topological argument of Bahri and Coron and complete the proof of Theorem \ref{thm:main}. In Appendix \ref{A2} we give the expansion of the GJMS operator obtained in \cite{MazPrem}. 

\section{Notations and preliminaries}\label{sec1}

\subsection*{Notations:}
Throughout this paper, $C$ will denote a generic positive constant that depends on $n,k$ and possibly $(M,g)$.  $\text{A} \lesssim \text{B}$ will equivalently denote $\text{A}= \bigO(\text{B})$. The relation $\text{A}\approx \text{B}$ means that the quantities $\text{A}, \text{B}$ satisfies both $\text{A} \lesssim \text{B}$ and $\text{B}\lesssim \text{A}$. Additional notations will be introduced as and when we use them first. 
\medskip

In this section, we collect the analytical ingredients needed in the proof of our theorem. The results listed here are either well established now, or have already appeared elsewhere. 
\smallskip

\noindent
$H^{k}(M)$ will denote the $k$-th order Sobolev space with $L^{2}$ inner product, and can equivalently  be defined as the closure of $C^{\infty}(M)$ with respect to the norm:
\begin{align*}
\|u\|_{H^{k}(M)}^{2}:=\sum\limits_{\ell=0}^{k}\int_{M}|\Delta_{g}^{\ell/2}u|^{2}\, dv_{g}.
\end{align*}
Here $\Delta^{\ell/2}u:=\nabla\Delta^{(\ell-1)/2}u$ when $\ell$ is odd. By the Sobolev embedding theorem $H^{k}(M)\hookrightarrow L^{2^{*}_{k}}(M)$, where $2^{*}_{k}:=2n/\(n-2k\)$ is the critical Sobolev exponent, and the following energy functional is well defined.  
\begin{align}\label{energy}
\mathcal{J}_{g,k}(u):=\frac{\displaystyle{\int_{M}u\,P_{g}u\,dv_{g}}}{\(\displaystyle{\int_{M}|u|^{\,2^{*}_{k}}\,dv_{g}}\)^{2/2^{*}_{k}}}\quad \hbox{ for } u\in H^{k}(M)\setminus\{0\}.
\end{align}
It follows that eq. \eqref{eq:one} is variational and, up to multiplying by a constant,  non-trivial solutions in $H^{k}(M)$ correspond to critical points of $\mathcal{J}_{g}$. To find positive solutions, we will use the following formulation: Consider the Banach submanifold of $H^{k}(M)$
\begin{align}\label{constraint}
\mathcal{M}_{k}^{+}:=\Big\{u\in H^{k}(M): u\ge0 \text{ a.e.} \text{ and } \int_{M}u^{\,2^{*}_{k}}\,dv_{g}=1\Big\}.
\end{align}
The functional $ \mathcal{J}_{g,k}(u)$ is $C^{2}$ on $\mathcal{M}_{k}^{+}$ and our positivity condition \eqref{positivity} ensures that any critical point $u$ of $\mathcal{J}_{g,k}(u)$ on $\mathcal{M}_{k}^{+}$ is a smooth positive solution of eq. \eqref{eq:one} for some $\lambda>0$. 
\begin{proposition}\label{crit_pts}
Let $k\ge1$ be an integer, and $\(M,g\)$ be a smooth, closed Riemannian manifold of dimension $n\ge 2k+1$. Assume that the $2k$th order GJMS operator $P_{g}$ satisfies the positivity condition \eqref{positivity}. Then any critical point $u$ of $\mathcal{J}_{g,k}(u)$ on $\mathcal{M}_{k}^{+}$ is a smooth positive solution of eq. \eqref{eq:one}  for some $\lambda>0$. 
\end{proposition}
For proofs and details, we refer the reader to Theorem $1$ and Theorem $3$ in \cite{Maz1}. Also see Andrade, K\"onig, Ratzkin and Wei \cite{AndKoRatWei} (Section 3), and \cite{Maz1} (Proposition 8.3 in Appendix A and Section 6). 
\medskip

We define the canonical bubble centred at $x=0$ and with height $=1$ as follows.
\begin{align}\label{bubble1}
B_{0}\(x\):={\big(1+\mathfrak{c}_{n,k}^{-1}\left|x\right|^2\big)^{\frac{2k-n}{2}}}~\hbox{ for } x\in\R^n, \hbox{ where }\mathfrak{c}_{n,k}=\[\prod \limits_{j=-k}^{k-1}\(n+2j\)\]^{1/k}.
\end{align}
Let $\Delta_0=-\sum_{i=1}^n \partial_i^2$  denote the Euclidean Laplacian. From the classification result of Wei-Xu \cite{WeiXu} it follows that, up to translations and rescalings, $B_{0}$ is the unique $C^{2k}(\R^{n})$ positive solution of
\begin{align}
\Delta_{0}^k\,B_{0}=B_{0}^{\,2^{*}_{k}-1}~\hbox{ in $\R^n$}.
\end{align}
The fundamental solution of $\Delta_0^k$ in $\R^n$ centered at $0$ is given by $G_0(x) = b_{n,k} |x|^{2k-n}$, where 
\begin{equation} \label{def:bnk}
 b_{n,k}^{-1} = 2^{k-1}\(k-1\)!\(n-2\)\(n-4\)\dotsm\(n-2k\)\omega_{n-1},
 \end{equation}
$\omega_{n-1}$ denotes the volume of the standard round sphere $\mathbb{S}^{n-1}$. Moreover $\displaystyle \mathfrak{c}_{n,k}^{\frac{n-2k}{2}} = b_{n,k} \int_{\R^n} B_{0}^{2^{*}_{k}-1}\, dx$ (for instance, see  Premoselli \cite{Prem13}).
\smallskip

Recall the best constant in the Sobolev inequality 
\begin{align}
\frac{1}{K(n,k)}:=\inf\limits_{\varphi\in C^{\infty}_{c}(\R^{n})\setminus\{0\}}\frac{\displaystyle{\int_{\R^{n}}\varphi\,\Delta^{k}_{0}\varphi\,dx}}{\(\displaystyle{\int_{\R^{n}}|\varphi|^{\,2^{*}_{k}}\,dv_{g}}\)^{2/2^{*}_{k}}}.
\end{align}
It follows from Lions \cite{Lions} and Swanson \cite{Swan} that the extremal functions for the above Sobolev inequality are given by constant multiples of translations and rescalings of  $B_{0}$. By conformal invariance of the GJMS operator, it then follows that
\begin{align}
&Y_{k}(\S^{n}):=\inf\limits_{\varphi\in C^{\infty}(\S^{n})}\frac{\displaystyle{\int_{\S^{n}}u\,P_{g_{\text{rd}}}u\,dv_{g_{\text{rd}}}}}{\(\displaystyle{\int_{\S^{n}}|u|^{\,2^{*}_{k}}\,dv_{g_{\text{rd}}}}\)^{2/2^{*}_{k}}}=\frac{1}{K(n,k)}=\(\, \int_{\R^{n}}B_{0}^{\,2^{*}_{k}}\, dx\)^{2k/n}.
\end{align}
Here $g_{\text{rd}}$ denotes the standard round metric on the $n$-dimensional sphere $\S^{n}$. The explicit expression of the best constant $K(n,k)$ in terms of the gamma functions can be found in \cite{Swan}.
\medskip

We next recall some geometric preliminaries. Let  $\Scal_{g}$ denote the scalar curvature, and let $\Ricci_{g}$, $\Weyl_{g}$ denote the Ricci and Weyl curvature tensors of $\(M,g\)$. Fixing $N>2$ large, by  Lee-Parker \cite{LeePar} (also see Cao \cite{Cao}, G\"{u}nther \cite {Gun}) there exists $\Lambda\in C^{\infty}(M\times M)$ such that, defining $\Lambda_{\xi}:=\Lambda(\xi,\cdot)$, we have for all $\xi \in M$
\begin{align}\label{conf.1}
\Lambda_{\xi}>0, ~\Lambda_{\xi}(\xi)=1~\hbox{and}~  \nabla \Lambda_{\xi}(\xi)=0,
\end{align}
and the conformal metrics $g_{\xi}:=\Lambda_{\xi}^{\frac{4}{n-2k}}g$ satisfies 
\begin{equation}\label{conf.2}
\det g_{\xi}\(x\)=1+\bigO(|x|^{N}) 
\end{equation}
around $0$ in  geodesic normal coordinates given by the the exponential map $\exp_{\xi}^{g_{\xi}}$ at $\xi$ with respect to the metric $g_{\xi}$. Moreover
\begin{align}\label{conf.3}
&\qquad\Scal_{g_{\xi}}\(\xi\)=0,~\nabla \Scal_{g_{\xi}}\(\xi\)=0,~\Ricci_{g_{\xi}}\(\xi\)=0, ~\Delta_{g_{\xi}} \Scal_{g_{\xi}}\(\xi\)=\frac{1}{6}\left|\Weyl_{g_{\xi}}\(\xi\)\right|^2,\notag\\
&\Sym\nabla\Ricci_{g_{\xi}}\(\xi\)=0,~\Sym\((\Ricci_{g_{\xi}})_{ab;cd}\(\xi\)+\frac{2}{9}(\Weyl_{g_{\xi}})_{eabf}\(\xi\)\tensor{\Weyl}{^e_{cd}^f}\(\xi\)\)=0.
\end{align}
Here  $\Sym$ denotes the symmetric part.  
\medskip

We will need a precise expansion of the GJMS operator for our arguments. Note that the leading order term in the expansion of $P_g$ is $\Delta_g^k$. Fixing $\xi\in M$, Cartan's expansion of $g$ gives us in geodesic normal coordinates,
\begin{align} \label{expansion.Pg0}
P_{\exp_\xi^*g}v (x)=&\,\Delta_{0}^{k}\,v(x) +\bigO \big (|x|^2 |\nabla^{2k} v(x)|_{g} \big ) \notag\\
&~+ \bigO \big( |x| |\nabla^{2k-1} v(x)|_{g} \big) + \bigO\Big(\sum_{\ell=0}^{2k-2}  |\nabla^{\ell} v(x)|_{g}\Big), ~\hbox { for }x\in \R^{n}
\end{align}
and the constants in the $\bigO(\cdot)$ terms are independent of $\xi, x$. See Robert \cite {RobPoly1} and \cite{MazPrem}. 
A more precise expansion of $P_{g}$ , in conformal normal coordinates, was recently obtained in \cite{MazPrem} (Proposition B.1). For clarity and completeness, we prove the expression in Appendix \ref{A2}.
\medskip

The Green's function $G_{g}$ of the GJMS operator $P_{g}$ will play a pivotal role in our analysis.  We will throughout assume the positivity condition \eqref{positivity}.  The Green's function of $P_g$ is the unique positive function $G_g \in C^\infty(M \times M \backslash \{x=y\})$ satisfying $P_g G_g(x, \cdot) = \delta_x$ for all $x \in M$.  The following bounds are satisfied for all $x \neq y\in M$ 
\begin{align}\label{bounds.Green}
\frac{1}{C} d_g(x,y)^{2k-n} \le G_{g}(x,y) \le C d_g(x,y)^{2k-n} \text{ and } \big| \nabla^{\ell} G_{g}(x,y) \big|_{g} \le C_\ell d_g(x,y)^{2k-\ell-n},
\end{align}
with constants $C$, $C_{\ell}$ depending only on $(M,g)$ and $n,k, \ell$. Furthermore for any $\xi \in M$, one has the expansion
\begin{align}\label{expan.green1}
G_{g}\big(\exp_{\xi}^{g}(x), \exp_{\xi}^{g}(y) \big) & = \frac{b_{n,k}}{|x-y|^{n-2k}} \Big( 1 + \bigO^{(2k-1)}(|x-y|) \Big),
\end{align}
where $b_{n,k}$ is as in \eqref{def:bnk} and the constants in the $\bigO(\cdot )$ term are independent of $x,y$ and $\xi$. The notation $f = \bigO^{(q)}(r^{p})$ denotes that $f$ satisfies $|\nabla^\ell f(x)|_g \lesssim r^{p-\ell}$ for all $0 \le \ell \le q$ and $x \neq 0$. For details and proof, see Theorem C.1 in \cite{RobPoly1} and Appendix C in \cite{MazPrem}. An expansion for the Green’s function $G_{g}$ of the GJMS operator $P_{g}$, for all $k$,  in conformal normal coordinates, involving the mass, was recently obtained in \cite{MazPrem} (Proposition C.2).  
\medskip

\noindent
Using the conformal invariance property \eqref{conf.inv.Pg} of $P_{g}$ it folllows that  for any $\xi \in M$ and for any $x \neq y $ in $M$, we have 
$$ G_{g_\xi}(x,y) = \Lambda_\xi(x)^{-1} \Lambda_\xi(y)^{-1} G_g(x,y),$$ 
where $g_{\xi} := \Lambda_\xi^{\frac{4}{n-2k}} g$ is the conformal metric satisfying \eqref{conf.1}, \eqref{conf.2} and  \eqref{conf.3}. For $2k+1 \le n \le 2k+3$ or $(M,g)$ locally conformally flat, one has the following expansion as $x \to 0$
\begin{align}\label{expan.green2}
G_{g_\xi}\big(\exp_{\xi}^{g_\xi}(x),\xi\big) & = \frac{b_{n,k}}{|x|^{n-2k}} + A_\xi + \bigO^{(2k-1)}(|x|), 
\end{align}
where $A_\xi \in \R$ and is called the \emph{mass} of the operator $P_{g_{\xi}}$ at the point $\xi$ ,  and as before the constant $b_{n,k}$ is given by  \eqref{def:bnk} and the constants in the $\bigO(\cdot )$ term are independent of $x,y$ and $\xi$. See Lee and Parker~\cite{LeePar} for $k=1$, Gursky and Malchiodi~\cite{GurMal}, Hang and Yang \cite{HangYang2} and Gong, Kim and Wei \cite{GongKimWei} for $k=2$, Chen-Hou \cite{ChenHou} for $k=3$ and Michel~\cite{Mic} for $k\ge2$. Note that the sign of the mass is conformally invariant. We say that $P_g$ has positive mass at every point if, in the above expansion, we have  $A_\xi >0$ for all $ \xi \in M$.
\medskip

\section{Bubbles and Error estimate}\label{sec:Bubbles}

\noindent
We let $\chi:[0,+\infty)\to[0,1]$ be a smooth cutoff function such that $\chi\equiv1$ in $\[0,1\]$ and $\chi\equiv0$ in $\[2,+\infty\)$. Fix $\xi\in M$, $0<\delta<1$ and consider the conformal metric $g_{\xi}:=\Lambda_{\xi}^{\frac{4}{n-2k}}g$ satisfying \eqref{conf.1}, \eqref{conf.2} and  \eqref{conf.3}. We define 
\begin{align}\label{bubble2}
&\widetilde{U}_{\xi,\mu}\(x\):=\chi_{\xi,\delta}(x)\,B_{\xi,\mu}(x), ~\hbox{where } \chi_{\xi,\delta}(x):=\chi\(\delta^{-1}d_{g_{\xi}}\(x,\xi\)\)\notag\\
&~\hbox{and } B_{\xi,\mu}(x):=\frac{\mu^{\frac{n-2k}{2}}}{(\mu^{2}+\mathfrak{c}^{-1}_{n,k}\,d_{g_{\xi}}(x,\xi)^{2})^{\frac{n-2k}{2}}}\quad\forall x\in M.
\end{align}
Here $d_{g_{\xi}}$ is the geodesic distance with respect to the metric $g_{\xi}$ and $\mathfrak{c}_{n,k}$ is given by \eqref{bubble1}. We next define 
\begin{align}\label{bubble3}
\widetilde{V}_{\xi,\mu}\(x\)=\widetilde{U}_{\xi,\mu}\(x\)+\mathfrak{c}^{\frac{n-2k}{2}}_{n,k}\,b^{-1}_{n,k}\,\mu^{\frac{n-2k}{2}} \(1-\chi_{\xi,\delta}(x)\)G_{g_{\xi}}(x,\xi),
\end{align}
where $b_{n,k}$ is given by \eqref{def:bnk}. We then define our bubbles centred at $\xi$ as:
\begin{align}\label{bubble3b}
V_{\xi,\mu}\(x\):=\Lambda_{\xi}(x)\widetilde{V}_{\xi,\mu}\(x\).
\end{align}
Note that our bubbles are $C^{\infty}$ functions and are the $k$-th order version of the Schoen-Brendle bubble (eq. (203) in Brendle \cite{Bre1}). Also see eq. (3.4) in \cite{MazVetois} for a similar bubble with the mass term added. Additional comments on our choice of bubbles are provided in Remark \ref{rem:expn1} in the context of the energy expansion.  
\medskip

\noindent
We first obtain the following crucial error estimate, and one can compare this directly with the Yamabe case ($k=1$) in \cite{Bre1} (Corollary B.2.), and the Paneitz case ($k=2$) in Gong, Lee and Wei \cite{GongLeeWei} (Section 5).
\begin{proposition}\label{prop:error}
Let $k\ge1$ be an integer, and $\(M,g\)$ be a smooth, closed Riemannian manifold of dimension $2k+1\le n\le 2k+3$ or $(M,g)$ be locally conformally flat. Assume that the GJMS operator $P_{g}$ satisfies \eqref{positivity}.
Fix $\xi\in M$, $0<\delta<1$ and choose $\mu$ sufficiently small as follows:
\begin{enumerate}
\item[(i)]
$0<\mu<\big(\mathfrak{c}_{n,k}^{-1/2}\,\delta\big)^{3}$ when $2k+1\le n\le 2k+3$. 
\item[(ii)]
$0<\mu<\big(\mathfrak{c}_{n,k}^{-1/2}\,\delta\big)^{n/2}$ when $(M,g) $ is locally conformally flat.
\end{enumerate}
The constant $\mathfrak{c}_{n,k}$ is defined in \eqref{bubble1}. Then the following estimate holds  for all $x\in M$
\begin{align}\label{error1}
&\big|\big(P_{g}V_{\xi,\mu}-V_{\xi,\mu}^{2^{*}_{k}-1}\big)(x)\big|\lesssim\(\frac{\mu^{\frac{n-2k}{2}}}{\(\mu+d_{g_{\xi}}(x,\xi)\)^{n-4}}\)\mathds{1}_{\{d_{g_{\xi}}(x,\xi)\leq2\delta\}}~+\notag\\
&~\mu^{\frac{n-2k}{2}}\delta^{-2k}\mathds{1}_{\{\delta\leq d_{g_{\xi}}(x,\xi)\leq 2\delta\}}~+\(\frac{\mu^{\frac{n+2k}{2}}}{\(\mu+d_{g_{\xi}}(x,\xi)\)^{n+2k}}\)\mathds{1}_{\{\delta\leq d_{g_{\xi}}(x,\xi)\}},
\end{align}
with the constants in $\lesssim$ depending only on $n,k$ and $(M,g)$.
\end{proposition}
\begin{proof}
By the conformal invariance property \eqref{conf.inv.Pg} we have $P_{g}V_{\xi,\mu}-V_{\xi,\mu}^{2^{*}_{k}-1}=\Lambda_{\xi}^{\frac{n+2k}{n-2k}}(P_{g_{\xi}}\widetilde{V}_{\xi,\mu}-\widetilde{V}_{\xi,\mu}^{2^{*}_{k}-1})$, and we write
\begin{align*}
\widetilde{V}_{\xi,\mu}\(x\)=\chi_{\xi,\delta}(x)\Big(B_{\xi,\mu}(x)-\mu^{\frac{n-2k}{2}}\mathfrak{c}^{\frac{n-2k}{2}}_{n,k}\,b^{-1}_{n,k}G_{g_{\xi}}(x,\xi)\Big)+\mu^{\frac{n-2k}{2}}\mathfrak{c}^{\frac{n-2k}{2}}_{n,k}b^{-1}_{n,k}G_{g_{\xi}}(x,\xi).
\end{align*}
We write the expression in terms of the commutator relation $P_{g_{\xi}}[u,v]:=P_{g_{\xi}}(uv)-uP_{g_{\xi}}v$. 
\begin{align*}
P_{g_{\xi}}\widetilde{V}_{\xi,\mu}&\,=P_{g_{\xi}}\big(\,\chi_{\xi,\delta}\big(B_{\xi,\mu}-\mu^{\frac{n-2k}{2}}\mathfrak{c}^{\frac{n-2k}{2}}_{n,k}b^{-1}_{n,k}G_{g_{\xi}}(\cdot,\xi)\big)\big)\\
&\,=P_{g_{\xi}}\big[\chi_{\xi,\delta},\big(B_{\xi,\mu}-\mu^{\frac{n-2k}{2}}\mathfrak{c}^{\frac{n-2k}{2}}_{n,k}b^{-1}_{n,k}G_{g_{\xi}}(\cdot,\xi)\big)\big]+\chi_{\xi,\delta}P_{g_{\xi}}\,B_{\xi,\mu}.
\end{align*}
Note that  for $\mu<\mathfrak{c}_{n,k}^{-3/2}\,\delta^{3}$ 
\begin{align*}
\Big|\nabla^{j}\Big(\big({\mu^{2}+\mathfrak{c}^{-1}_{n,k}~|x|^{2}}\big)^{\frac{2k-n}{2}}-\mathfrak{c}^{\frac{n-2k}{2}}_{n,k}|x|^{2}\Big)\Big|\lesssim\delta^{-j}~\hbox{ in } B(0,9\delta/4)\setminus B(0,3\delta/4).
\end{align*}
Using the expansion of the Green's function \eqref{expan.green2}, we then obtain. 
$$\left|P_{g_{\xi}}\big[\chi_{\xi,\delta},\big(B_{\xi,\mu}-\mu^{\frac{n-2k}{2}}\mathfrak{c}^{\frac{n-2k}{2}}_{n,k}b^{-1}_{n,k}G_{g_{\xi}}(\cdot,\xi)\big)\big]\right|\lesssim\mu^{\frac{n-2k}{2}}\delta^{-2k}\mathds{1}_{\{\delta\leq d_{g_{\xi}}(x,\xi)\leq 2\delta\}},$$
and this implies that 
\begin{align}
&\big|\big(P_{g_{\xi}}\widetilde{V}_{\xi,\mu}-\chi_{\xi,\delta}P_{g_{\xi}}B_{\xi,\mu}\big)(x)\big|\lesssim\mu^{\frac{n-2k}{2}}\delta^{-2k}\mathds{1}_{\{\delta\leq d_{g_{\xi}}(x,\xi)\leq 2\delta\}}.
\end{align}
Then we have obtained that 
\begin{align*}
\big|P_{g_{\xi}}\widetilde{V}_{\xi,\mu}-\widetilde{V}_{\xi,\mu}^{2^{*}_{k}-1}\big|\lesssim&~\chi_{\xi,\delta}|P_{g_{\xi}}B_{\xi,\mu}-B_{\xi,\mu}^{2^{*}_{k}-1}|+\mu^{\frac{n-2k}{2}}\delta^{-2k}\mathds{1}_{\{\delta\leq d_{g_{\xi}}(x,\xi)\leq 2\delta\}}\\
&~+\chi_{\xi,\delta}\big(1-\chi_{\xi,\delta}^{2^{*}_{k}-2}\big)B_{\xi,\mu}^{2^{*}_{k}-1}+|\widetilde{U}_{\xi,\mu}^{2^{*}_{k}-1}-\widetilde{V}_{\xi,\mu}^{2^{*}_{k}-1}|.
\end{align*}
From the expansion of the GJMS operator $P_{g_{\xi}}$ in conformal coordinates \eqref{expansion.Pg1} it follows that
\begin{align}
\chi_{\xi,\delta}\big|P_{g_{\xi}}B_{\xi,\mu}-B_{\xi,\mu}^{2^{*}_{k}-1}\big|\lesssim\mu^{\frac{n-2k}{2}}(\mu+d_{g_{\xi}}(x,\xi))^{4-n}\mathds{1}_{\{d_{g_{\xi}}(x,\xi)\leq2\delta\}}.
\end{align}
Also see eq. (4.5) in \cite{MazPrem}. For the final term with nonlinearity, we obtain
\begin{align}
&\left|\widetilde{U}_{\xi,\mu}^{2^{*}_{k}-1}-\widetilde{V}_{\xi,\mu}^{2^{*}_{k}-1}\right|=\left|\big(\widetilde{U}_{\xi,\mu}+\mathfrak{c}_{n,k}\,b^{-1}_{n,k}\,\mu^{\frac{n-2k}{2}} \(1-\chi_{\xi,\delta}(x)\big)G_{g_{\xi}}(x,\xi)\)^{2^{*}_{k}-1}-\widetilde{U}_{\xi,\mu}^{2^{*}_{k}-1}\right|\notag\\
&\lesssim\Big[\mu^{\frac{n-2k}{2}}\widetilde{U}_{\xi,\mu}^{2^{*}_{k}-2}\,G_{g_{\xi}}(x,\xi)+\big(\mu^{\frac{n-2k}{2}}G_{g_{\xi}}(x,\xi)\big)^{2^{*}_{k}-1}\Big]\mathds{1}_{\{\delta\leq d_{g_{\xi}}(x,\xi)\}}\notag\\
&\lesssim\frac{\mu^{\frac{n+2k}{2}}}{(\mu+d_{g_{\xi}}(x,\xi))^{4k}}\,d_{g_{\xi}}(x,\xi)^{2k-n}\mathds{1}_{\{\delta\leq d_{g_{\xi}}(x,\xi)\}}+\frac{\mu^{\frac{n+2k}{2}}}{d_{g_{\xi}}(x,\xi)^{(n+2k)}}\mathds{1}_{\{\delta\leq d_{g_{\xi}}(x,\xi)\}}\notag\\
&\lesssim\(\frac{\mu^{\frac{n+2k}{2}}}{\(\mu+d_{g_{\xi}}(x,\xi)\)^{n+2k}}\)\mathds{1}_{\{\delta\leq d_{g_{\xi}}(x,\xi)\}}.
\end{align}
Note $\displaystyle \chi_{\xi,\delta}\big(1-\chi_{\xi,\delta}^{2^{*}_{k}-2}\big)B_{\xi,\mu}^{2^{*}_{k}-1}\lesssim\(\frac{\mu^{\frac{n+2k}{2}}}{\(\mu+d_{g_{\xi}}(x,\xi)\)^{n+2k}}\)\mathds{1}_{\{\delta\leq d_{g_{\xi}}(x,\xi)\}}$. 
Combining the last two estimates gives us \eqref{error1}.
\end{proof}
\begin{remark}
We note that the first term in the expansion \eqref{error1} is not present if $(M,g)$ is locally conformally flat. Similarly, the last two terms comes from the contribution of the Green's function in ${V}_{\xi,\mu}$. 
\end{remark}
\medskip

\section{Interaction estimates}\label{sec:Interaction}

In this section, we obtain estimates on the interactions between bubbles. We begin by calculating the self-interactions.
\begin{lemma} 
Let $k\ge1$ be an integer, and $\(M,g\)$ be a smooth, closed Riemannian manifold of dimension $2k+1\le n\le 2k+3$ or $(M,g) $ be locally conformally flat. Assume that the GJMS operator $P_{g}$ satisfies \eqref{positivity}. Fix $\xi\in M$, $0<\delta<1$ and choose $0<\mu<\delta$ sufficiently small, and such that the error estimate \eqref{error1} holds. Then we have the following.
\begin{align}\label{int.1}
\left|~\int_{M}(P_{g}V_{\xi,\mu}-V_{\xi,\mu}^{2^{*}_{k}-1})V_{\xi,\mu}\,dv_{g}\right|\lesssim&~~\mu^{4}\,\delta^{4+2k-n}+\mu^{n-2k}+\mu^{n}\delta^{-n},
\end{align}
with the constants in $\lesssim$ depending only on $n,k$ and $(M,g)$.
\end{lemma}
\begin{proof}
Using the error estimate  \eqref{error1} and the bounds on the Green's function \eqref{bounds.Green}, we get
\begin{align*}
&\left|~\int_{M}(P_{g}V_{\xi,\mu}-V_{\xi,\mu}^{2^{*}_{k}-1})V_{\xi,\mu}\,dv_{g}\right|\lesssim\int_{B_{g_{\xi}}(\xi,2\delta)}\frac{\mu^{\frac{n-2k}{2}}}{\(\mu+d_{g_{\xi}}(x,\xi)\)^{n-4}}\widetilde{V}_{\xi,\mu}\,dv_{g_{\xi}}~+\notag\\
&~\mu^{\frac{n-2k}{2}}\delta^{-2k}\int_{B_{g_{\xi}}(\xi,2\delta)\setminus B_{g_{\xi}}(\xi,\delta)}\widetilde{V}_{\xi,\mu}\,dv_{g_{\xi}}+\int_{M\setminus B_{g_{\xi}}(\xi,\delta)}\frac{\mu^{\frac{n+2k}{2}}}{\(\mu+d_{g_{\xi}}(x,\xi)\)^{n+2k}}\widetilde{V}_{\xi,\mu}\,dv_{g_{\xi}}\notag\\
\lesssim&\int_{B_{g_{\xi}}(\xi,2\delta)}\frac{\mu^{n-2k}}{\(\mu+d_{g_{\xi}}(x,\xi)\)^{2n-4-2k}}\,dv_{g_{\xi}}+\frac{\mu^{n-2k}}{\delta^{\,2k}}\int_{B_{g_{\xi}}(\xi,2\delta)\setminus B_{g_{\xi}}(\xi,\delta)}d_{g_{\xi}}(x,\xi)^{2k-n}\,dv_{g_{\xi}}\\\
&~+\mu^{n}\int_{M\setminus B_{g_{\xi}}(\xi,\delta)}\frac{1}{\(\mu+d_{g_{\xi}}(x,\xi)\)^{2n}}\,dv_{g_{\xi}}\notag\\
\lesssim&\int_{B(0,2\delta)}\frac{\mu^{n-2k}}{\(\mu+|x| \)^{2n-4-2k}}\,dx+\mu^{n-2k}\delta^{-2k}\int_{B(0,2\delta)\setminus B(0,\delta)}|x|^{2k-n}\,dx\notag\\
&+\mu^{n}\int_{B(0,2\delta)\setminus B(0,\delta)}|x|^{-2n}\,dx+\mu^{n}\delta^{-n}.\\
\lesssim&~\mu^{4}\,\delta^{4+2k-n}+\mu^{n-2k}+\mu^{n}\delta^{-n}.
\end{align*}
\end{proof}
\begin{remark}
Again, we note that the first term in \eqref{int.1} is not present if $(M,g)$ is locally conformally flat. Similarly, the last two terms come from the contribution of the Green's function in ${V}_{\xi,\mu}$. 
\end{remark}
\medskip

\noindent
Next, we estimate the difference of the nonlinear terms.
\begin{lemma}
Let $k\ge1$ be an integer, and $\(M,g\)$ be a smooth, closed Riemannian manifold of dimension $2k+1\le n\le 2k+3$ or $(M,g) $ be locally conformally flat. Assume that the GJMS operator $P_{g}$ satisfies \eqref{positivity}. Fix $\xi\in M$, $0<\delta<1$.  For $0<\mu<\delta$ sufficiently small, the following holds.
\begin{align}\label{int.2}
\left|~\int_{M}(V_{\xi,\mu}^{2^{*}_{k}}-B_{\xi,\mu}^{2^{*}_{k}})\,dv_{g}\right|\lesssim&~\(\frac{\mu}{\delta}\)^{n},
\end{align}
with the constants in $\lesssim$ depending only on $n,k$ and $(M,g)$.
\end{lemma}
\begin{proof}
We have
\begin{align*}
&|V_{\xi,\mu}^{2^{*}_{k}}-B_{\xi,\mu}^{2^{*}_{k}}|=\,\Big|\big[\chi_{\xi,\delta}B_{\xi,\mu}+\mathfrak{c}_{n,k}\,b^{-1}_{n,k}\,\mu^{\frac{n-2k}{2}} \(1-\chi_{\xi,\delta}(x)\)G_{g_{\xi}}(x,\xi)\big]^{2^{*}_{k}}\\
&\hspace{2.5cm}-\big(\chi_{\xi,\delta}B_{\xi,\mu}+(1-\chi_{\xi,\delta}B_{\xi,\mu})\big)^{2^{*}_{k}}\Big|\\
&\lesssim\Big[\mu^{\frac{n-2k}{2}}B_{\xi,\mu}^{2^{*}_{k}-1}\,G_{g_{\xi}}(x,\xi)+\big(\mu^{\frac{n-2k}{2}}G_{g_{\xi}}(x,\xi)\big)^{2^{*}_{k}}+B_{\xi,\mu}^{2^{*}_{k}}\Big]\mathds{1}_{\{\delta\leq d_{g_{\xi}}(x,\xi)\}}\\
&\lesssim\mu^{n}d_{g_{\xi}}(x,\xi)^{-2n}\mathds{1}_{\{\delta\leq d_{g_{\xi}}(x,\xi)\}}.
\end{align*}
Then 
\begin{align*}
\left|~\int_{M}(V_{\xi,\mu}^{2^{*}_{k}}-B_{\xi,\mu}^{2^{*}_{k}})\,dv_{g}\right|\lesssim\left|~\int_{M\setminus B_{g_{\xi}}(\xi,\delta)}(V_{\xi,\mu}^{2^{*}_{k}}-B_{\xi,\mu}^{2^{*}_{k}})\,dv_{g}\right|\lesssim\(\frac{\mu}{\delta}\)^{n}. 
\end{align*}
\end{proof}

\noindent
We now calculate the interactions between two different bubbles. When $k=1$, this was obtained in Lemma B.5. in \cite{Bre1}, and very recently in \cite{GongLeeWei} (Section 5) for the case $k=2$ 
\begin{lemma} 
Let $k\ge1$ be an integer, $\(M,g\)$ be a smooth, closed Riemannian manifold of dimension $2k+1\le n\le 2k+3$ or $(M,g) $ be locally conformally flat. Assume that the GJMS operator $P_{g}$ satisfies \eqref{positivity} and fix  $0<\delta<1$. Then for all $\xi_{i}, \xi_{j}\in M$ and $0<\mu_{i},\mu_{j}<\delta$ sufficiently small and such that the error estimate \eqref{error1} holds, we obtain the following interaction estimate.
\begin{align}\label{int.3}
&\left|~\int_{M}(P_{g}V_{\xi_{i},\mu_{i}}-V_{\xi_{i},\mu_{i}}^{2^{*}_{k}-1})V_{\xi_{j},\mu_{j}}\,dv_{g}\right|\lesssim\(\frac{\mu_{i}}{\mu_{j}}+\frac{d_{g}(\xi_{i},\xi_{j})}{\mu_{i}\mu_{j}}\)^{\frac{2k-n}{2}}\(\mu_{i}^{4}+\delta^{n-2k}+\(\frac{\mu_{i}}{\delta}\)^{2k}\),\notag\\
\end{align}
with the constants in $\lesssim$ depending only on $n,k$ and $(M,g)$.
\end{lemma}
\begin{proof}
Using the error estimate \eqref{error1} we can write
\begin{align}
&\int_{M}\big(P_{g}V_{\xi_{i},\mu_{i}}-V_{\xi_{i},\mu_{i}}^{2^{*}_{k}-1}\big)V_{\xi_{j},\mu_{j}}\,dv_{g}=\int_{M}\big(P_{g_{\xi_{i}}}\widetilde{V}_{\xi_{i},\mu_{i}}-\widetilde{V}_{\xi_{i},\mu_{i}}^{2^{*}_{k}-1}\big)\Lambda_{\xi_{i}}^{-1}V_{\xi_{j},\mu_{j}}\,dv_{g_{\xi_{i}}}\notag\\
&\lesssim\int_{B_{g_{\xi_{i}}}(\xi_{i},2\delta)}\frac{\mu_{i}^{\frac{n-2k}{2}}}{\big(\mu_{i}+d_{g_{\xi_{i}}}(x,\xi_{i})\big)^{n-4}}\Lambda_{\xi_{i}}^{-1}V_{\xi_{j},\mu_{j}}\,dv_{g_{\xi_{i}}}\notag\\
&~+\mu_{i}^{\frac{n-2k}{2}}\delta^{-2k}\int_{B_{g_{\xi_{i}}}(\xi_{i},2\delta)\setminus B_{g_{\xi_{i}}}(\xi_{i},\delta)}\Lambda_{\xi_{i}}^{-1}V_{\xi_{j},\mu_{j}}\,dv_{g_{\xi_{i}}}\notag\\
&~+\int_{M\setminus B_{g_{\xi_{i}}}(\xi_{i},\delta)}\frac{\mu_{i}^{\frac{n+2k}{2}}}{\big(\mu_{i}+d_{g_{\xi_{i}}}(x,\xi_{i})\big)^{n+2k}}\Lambda_{\xi_{i}}^{-1}V_{\xi_{j},\mu_{j}}\,dv_{g_{\xi_{i}}}:={\mathscr{I}}_{\text{C}}+{\mathscr{I}}_{\text{N}}+{\mathscr{I}}_{\text{E}}.
\end{align}
We estimate term-by-term. Note that
$$V_{\xi,\mu}\lesssim\(\frac{\mu^{\frac{n-2k}{2}}}{\big(\mu+d_{g_{\xi}}(x,\xi)\big)^{n-2k}}\)\mathds{1}_{\{d_{g_{\xi}}(x,\xi)\leq2\delta\}}+b^{-1}_{n,k}\,\mu^{\frac{n-2k}{2}}G_{g_{\xi}}(x,\xi)\mathds{1}_{\{\delta\leq d_{g_{\xi}}(x,\xi)\}}.$$
\medskip

\noindent
We start with the terms in the ``core region''. We write
\begin{align*}
{\mathscr{I}}_{\text{C}}&=\int_{B_{g_{\xi_{i}}}(\xi_{i},2\delta)}\frac{\mu_{i}^{\frac{n-2k}{2}}}{\(\mu_{i}+d_{g_{\xi_{i}}}(x,\xi_{i})\)^{n-4}}\Lambda_{\xi_{i}}^{-1}V_{\xi_{j},\mu_{j}}\,dv_{g_{\xi_{i}}}\\
&\le\int_{B_{g_{\xi_{i}}}(\xi_{i},2\delta)\cap B_{g_{\xi_{j}}}(\xi_{j},2\delta)}+\int_{B_{g_{\xi_{i}}}(\xi_{i},2\delta)\setminus B_{g_{\xi_{j}}}(\xi_{j},\delta)}:={\mathscr{I}}_{\text{C}}^{a}+{\mathscr{I}}_{\text{C}}^{b}.
\end{align*}
First assume $d_{g_{\xi_{i}}}(\xi_{i},\xi_{j})\leq 3\delta$. This allows us to write 
\begin{align}
&{\mathscr{I}}^{a}_{\text{C}}\leq\notag\\
&\int_{B_{g_{\xi_{i}}}(\xi_{i},2\delta)\cap B_{g_{\xi_{j}}}(\xi_{j},2\delta)}\Lambda_{\xi_{i}}^{-1}\Lambda_{\xi_{j}}\frac{\mu_{i}^{\frac{n-2k}{2}}}{\(\mu_{i}+d_{g_{\xi_{i}}}(x,\xi_{i})\)^{n-4}}\frac{\mu_{j}^{\frac{n-2k}{2}}}{(\mu_{j}^{2}+\mathfrak{c}^{-1}_{n,k}~d_{g_{\xi_{j}}}(x,\xi_{j})^{2})^{\frac{n-2k}{2}}}\,dv_{g_{\xi_{i}}}\notag\\
&\lesssim\int_{B{(0,2\delta)}}\Lambda_{\xi_{i}}^{-1}\Lambda_{\xi_{j}}(\exp_{\xi_{i}}^{g_{\xi_{i}}}(x))\frac{\mu_{i}^{\frac{n-2k}{2}}}{\(\mu_{i}+|x|\)^{n-4}}\frac{\mu_{j}^{\frac{n-2k}{2}}}{(\mu_{j}^{2}+~d_{g_{\xi_{j}}}(\exp_{\xi_{i}}^{g_{\xi_{i}}}(x),\xi_{j})^{2})^{\frac{n-2k}{2}}}\,dx\notag\\
&\lesssim\int_{B{(0,2\delta)}}\frac{\mu_{i}^{\frac{n-2k}{2}}}{\(\mu_{i}+|x|\)^{n-4}}\frac{\mu_{j}^{\frac{n-2k}{2}}}{(\mu_{j}^{2}+|x-\Theta_{\xi_{i}\xi_{j}}|^{2})^{\frac{n-2k}{2}}}\,dx.
\end{align}
Here $\Theta_{\xi_{i}\xi_{j}}:=(\exp_{\xi_{i}}^{g_{\xi_{i}}})^{-1}(\xi_{j})$ and by equivalence of metrics $|\Theta_{\xi_{i}\xi_{j}}|=d_{g_{\xi_{i}}}(\xi_{i},\xi_{j})\approx d_{g}(\xi_{i},\xi_{j})$. And thus
\begin{align}
&{\mathscr{I}}^{a}_{\text{C}}~\lesssim\int_{B{(0,2\delta)}}\frac{\mu_{i}^{\frac{n-2k}{2}}}{\(\mu_{i}+|x|\)^{n-4}}\frac{\mu_{j}^{\frac{n-2k}{2}}}{(\mu_{j}^{2}+|x-\Theta_{\xi_{i}\xi_{j}}|^{2})^{\frac{n-2k}{2}}}\,dx\notag\\
&\,\lesssim\int_{\{2|x-\Theta_{\xi_{i}\xi_{j}}|\geq \mu_{i}+|\Theta_{\xi_{i}\xi_{j}}|\}\cap B{(0,2\delta)}}+\int_{\{2|x-\Theta_{\xi_{i}\xi_{j}}|\leq \mu_{i}+|\Theta_{\xi_{i}\xi_{j}}|\}\cap B{(0,2\delta)}}\notag\\
&\,\lesssim\(\frac{\mu_{i}\mu_{j}}{\mu_{i}^{2}+|\Theta_{\xi_{i}\xi_{j}}|^{2}}\)^{\frac{n-2k}{2}}\int_{B{(0,2\delta)}}\frac{1}{\(\mu_{i}+|x|\)^{n-4}}\,dx\notag\\
&~+\(\frac{\mu_{i}\mu_{j}}{\mu_{i}^{2}+|\Theta_{\xi_{i}\xi_{j}}|^{2}}\)^{\frac{n-2k}{2}}\int_{B{(\Theta_{\xi_{i}\xi_{j}},(\mu_{i}+|\Theta_{\xi_{i}\xi_{j}}|)/2)}}\frac{\(\mu_{i}+|\Theta_{\xi_{i}\xi_{j}}|\)^{4-2k}}{\(\mu^{2}_{j}+|x-\Theta_{\xi_{i}\xi_{j}}|^{2}\)^{\frac{n-2k}{2}}}\,dx\notag\\
&\,\lesssim\(\frac{\mu_{i}\mu_{j}}{\mu_{i}^{2}+|\Theta_{\xi_{i}\xi_{j}}|^{2}}\)^{\frac{n-2k}{2}}\(\delta^{4}+\(\mu_{i}+|\Theta_{\xi_{i}\xi_{j}}|\)^{4}\)\lesssim\(\frac{\mu_{i}\mu_{j}}{\mu_{i}^{2}+d_{g}(\xi_{i},\xi_{j})^{2}}\)^{\frac{n-2k}{2}}\(\mu_{i}^{4}+\delta^{4}\).
\end{align}
Next, still assuming $d_{g_{\xi_{i}}}(\xi_{i},\xi_{j})\leq 3\delta$, we can write
\begin{align}
&{\mathscr{I}}^{b}_{\text{C}}\leq \int_{B_{g_{\xi_{i}}}(\xi_{i},2\delta)\setminus B_{g_{\xi_{j}}}(\xi_{j},\delta)}\Lambda_{\xi_{i}}^{-1}\Lambda_{\xi_{j}}\frac{\mu_{i}^{\frac{n-2k}{2}}}{\(\mu_{i}+d_{g_{\xi_{i}}}(x,\xi_{i})\)^{n-4}}\frac{\mu_{j}^{\frac{n-2k}{2}}}{\(d_{g_{\xi_{j}}}(x,\xi_{j})\)^{n-2k}}\,dv_{g_{\xi_{i}}}\notag\\
&\lesssim\int_{B{(0,2\delta)}}\frac{\mu_{i}^{\frac{n-2k}{2}}}{\(\mu_{i}+|x|\)^{n-4}}\frac{\mu_{j}^{\frac{n-2k}{2}}}{|x-\Theta_{\xi_{i}\xi_{j}}|^{n-2k}}\,dx\lesssim (\mu_{i}\mu_{j})^{\frac{n-2k}{2}}\delta^{2k+4-n}\notag\\
&\lesssim\(\frac{\mu_{i}\mu_{j}}{\mu_{i}^{2}+d_{g}(\xi_{i},\xi_{j})^{2}}\)^{\frac{n-2k}{2}}\delta^{4}.
\end{align}
For $d_{g_{\xi_{i}}}(\xi_{i},\xi_{j})\geq 3\delta$ we similarly have
\begin{align}
&{\mathscr{I}}_{\text{C}}\leq\int_{B_{g_{\xi_{i}}}(\xi_{i},2\delta)}\Lambda_{\xi_{i}}^{-1}\Lambda_{\xi_{j}}\frac{\mu_{i}^{\frac{n-2k}{2}}}{\(\mu_{i}+d_{g_{\xi_{i}}}(x,\xi_{i})\)^{n-4}}\frac{\mu_{j}^{\frac{n-2k}{2}}}{\(d_{g_{\xi_{j}}}(x,\xi_{j})\)^{n-2k}}\,dv_{g_{\xi_{i}}}\notag\\
&\lesssim\(\frac{\mu_{i}\mu_{j}}{\mu_{i}^{2}+d_{g}(\xi_{i},\xi_{j})^{2}}\)^{\frac{n-2k}{2}}\delta^{4}.
\end{align}
So we have obtained that
\begin{align}
&{\mathscr{I}}_{\text{C}}\lesssim\(\frac{\mu_{i}\mu_{j}}{\mu_{i}^{2}+d_{g}(\xi_{i},\xi_{j})^{2}}\)^{\frac{n-2k}{2}}(\mu_{i}^{4}+\delta^{4}).
\end{align}
\medskip

\noindent
Next, we estimate the terms in the ``exterior-region'' in the interaction.
\begin{align*}
{\mathscr{I}}_{\text{E}}&=\int_{M\setminus B_{g_{\xi_{i}}}(\xi_{i},\delta)}\frac{\mu_{i}^{\frac{n+2k}{2}}}{\(\mu_{i}+d_{g_{\xi_{i}}}(x,\xi_{i})\)^{n+2k}}\Lambda_{\xi_{i}}^{-1}V_{\xi_{j},\mu_{j}}\,dv_{g_{\xi_{i}}}\\
&\le\int_{B_{g_{\xi_{j}}}(\xi_{j},2\delta)\setminus B_{g_{\xi_{i}}}(\xi_{i},\delta)}+\int_{M\setminus(B_{g_{\xi_{j}}}(\xi_{j},\delta)\cup B_{g_{\xi_{i}}}(\xi_{i},\delta))}:={\mathscr{I}}^{a}_{\text{E}}+{\mathscr{I}}^{b}_{\text{E}}.
\end{align*}
Assume again $d_{g_{\xi_{j}}}(\xi_{i},\xi_{j})\leq 3\delta$ and let $\Theta_{\xi_{j}\xi_{i}}:=(\exp_{\xi_{j}}^{g_{\xi_{j}}})^{-1}(\xi_{i})$.  Then we can write 
\begin{align}
{\mathscr{I}}^{a}_{\text{E}}&~\lesssim\int_{B{(0,2\delta)\setminus}B(\Theta_{\xi_{j}\xi_{i}},\delta/2)}\frac{\mu_{i}^{\frac{n+2k}{2}}}{\(\mu_{i}+|x-\Theta_{\xi_{j}\xi_{i}}|\)^{n+2k}}\frac{\mu_{j}^{\frac{n-2k}{2}}}{(\mu_{j}^{2}+|x|^{2})^{\frac{n-2k}{2}}}\,dx\notag\\
&\lesssim\int_{\{4|x|\geq \mu_{i}+|\Theta_{\xi_{j}\xi_{i}}|\}\cap B(0,2\delta)\setminus B(\Theta_{\xi_{j}\xi_{i}},\delta/2)}+\int_{\{4|x|\leq \mu_{i}+|\Theta_{\xi_{j}\xi_{i}}|\}\cap B{(0,2\delta)\setminus}B(\Theta_{\xi_{j}\xi_{i}},\delta/2)}\notag\\
&\lesssim\(\frac{\mu_{i}\mu_{j}}{\mu_{i}^{2}+|\Theta_{\xi_{j}\xi_{i}}|^{2}}\)^{\frac{n-2k}{2}}\int_{B(\Theta_{\xi_{j}\xi_{i}},6\delta)\setminus B(\Theta_{\xi_{j}\xi_{i}},\delta/2)}\frac{\mu_{i}^{2k}}{\(\mu_{i}+|x-\Theta_{\xi_{j}\xi_{i}}|\)^{n+2k}}\,dx\notag\\
&~+\(\frac{\mu_{i}\mu_{j}}{\mu_{i}^{2}+|\Theta_{\xi_{j}\xi_{i}}|^{2}}\)^{\frac{n-2k}{2}}\mu_{i}^{2k}\int_{B{(0,(\mu_{i}+|\Theta_{\xi_{j}\xi_{i}}|)/2)}}\frac{\(\mu_{i}+|\Theta_{\xi_{j}\xi_{i}}|\)^{-4k}}{\(\mu^{2}_{j}+|x|^{2}\)^{\frac{n-2k}{2}}}\,dx\notag\\
&\lesssim\(\frac{\mu_{i}\mu_{j}}{\mu_{i}^{2}+d_{g}(\xi_{i},\xi_{j})^{2}}\)^{\frac{n-2k}{2}}\(\frac{\mu_{i}}{\delta}\)^{2k}.
\end{align}
And, still assuming $d_{g_{\xi_{j}}}(\xi_{i},\xi_{j})\leq 3\delta$ we can write
\begin{align}
{\mathscr{I}}^{b}_{\text{E}}&~=\int_{M\setminus(B_{g_{\xi_{j}}}(\xi_{j},\delta)\cup B_{g_{\xi_{i}}}(\xi_{i},\delta))}\frac{\mu_{i}^{\frac{n+2k}{2}}}{\(\mu_{i}+d_{g_{\xi_{i}}}(x,\xi_{i})\)^{n+2k}}\Lambda_{\xi_{i}}^{-1}V_{\xi_{j},\mu_{j}}~dv_{g_{\xi_{i}}}\notag\\
&\lesssim\int_{M\setminus(B_{g_{\xi_{j}}}(\xi_{j},\delta)\cup B_{g_{\xi_{i}}}(\xi_{i},\delta))}\frac{\mu_{i}^{\frac{n+2k}{2}}}{\(\mu_{i}+d_{g_{\xi_{i}}}(x,\xi_{i})\)^{n+2k}}\frac{\mu_{j}^{\frac{n-2k}{2}}}{\(d_{g_{\xi_{j}}}(x,\xi_{j})\)^{n-2k}}~dv_{g_{\xi_{i}}}.\notag\\
&\lesssim\mu_{j}^{\frac{n-2k}{2}}\delta^{2k-n}\int_{M\setminus B_{g_{\xi_{i}}}(\xi_{i},\delta/2)}\frac{\mu_{i}^{\frac{n+2k}{2}}}{\(\mu_{i}+d_{g_{\xi_{i}}}(x,\xi_{i})\)^{n+2k}}\,dv_{g_{\xi_{i}}}\notag\\
&\lesssim\mu_{j}^{\frac{n-2k}{2}}\delta^{2k-n}\[~\frac{\mu_{i}^{\frac{n+2k}{2}}}{\delta^{2k}}+\mu_{i}^{\frac{n+2k}{2}}\]\lesssim\(\frac{\mu_{i}\mu_{j}}{\mu_{i}^{2}+d_{g}(\xi_{i},\xi_{j})^{2}}\)^{\frac{n-2k}{2}}\(\frac{\mu_{i}}{\delta}\)^{2k}.
\end{align}
For $d_{g_{\xi_{j}}}(\xi_{i},\xi_{j})\geq 3\delta$ we proceed as follows. 
\begin{align}
&~{\mathscr{I}}_{\text{E}}=\int_{M\setminus B_{g_{\xi_{i}}}(\xi_{i},\delta)}\frac{\mu_{i}^{\frac{n+2k}{2}}}{\(\mu_{i}+d_{g_{\xi_{i}}}(x,\xi_{i})\)^{n+2k}}\Lambda_{\xi_{i}}^{-1}V_{\xi_{j},\mu_{j}}\,dv_{g_{\xi_{i}}}\notag\\
&~=\int_{B_{g_{\xi_{j}}}(\xi_{j},2\delta)\setminus B_{g_{\xi_{i}}}(\xi_{i},\delta)}+\int_{M\setminus(B_{g_{\xi_{j}}}(\xi_{j},\delta)\cup B_{g_{\xi_{i}}}(\xi_{i},\delta))}\notag\\
&\lesssim\int_{B_{g_{\xi_{j}}}(\xi_{j},2\delta)\setminus B_{g_{\xi_{i}}}(\xi_{i},\delta)}\frac{\mu_{i}^{\frac{n+2k}{2}}}{\(\mu_{i}+d_{g_{\xi_{i}}}(x,\xi_{i})\)^{n+2k}}\(\frac{\mu_{j}}{\mu_{j}^{2}+d_{g_{\xi_{j}}}(x,\xi_{j})^{2}}\)^{\frac{n-2k}{2}}\,dv_{g_{\xi_{j}}}\notag\\
&+\int_{M\setminus(B_{g_{\xi_{j}}}(\xi_{j},\delta)\cup B_{g_{\xi_{i}}}(\xi_{i},\delta))}\frac{\mu_{i}^{\frac{n+2k}{2}}}{\(\mu_{i}+d_{g_{\xi_{i}}}(x,\xi_{i})\)^{n+2k}}\frac{\mu_{j}^{\frac{n-2k}{2}}}{d_{g_{\xi_{j}}}(x,\xi_{j})^{n-2k}}\,dv_{g_{\xi_{i}}}\notag\\
&\lesssim(\mu_{i}\mu_{j})^{\frac{n-2k}{2}}\frac{\mu_{i}^{2k}}{d_{g_{\xi_{j}}}(\xi_{i},\xi_{j})^{n+2k}}\delta^{2k}+(\mu_{i}\mu_{j})^{\frac{n-2k}{2}}\mu_{j}^{2k}\[1+\frac{\delta^{-2k}}{d_{g_{\xi_{j}}}(\xi_{i},\xi_{j})^{n-2k}}\].
\end{align}
Hence 
\begin{align}\label{Int:ext}
&{\mathscr{I}}_{\text{E}}\lesssim\(\frac{\mu_{i}\mu_{j}}{\mu_{i}^{2}+d_{g}(\xi_{i},\xi_{j})^{2}}\)^{\frac{n-2k}{2}}\(\frac{\mu_{i}}{\delta}\)^{2k}.
\end{align}
\medskip

\noindent
In  the ``neck-region'' we get
\begin{align}
&{\mathscr{I}}_{\text{N}}=\mu_{i}^{\frac{n-2k}{2}}\delta^{-2k}\int_{B_{g_{\xi_{i}}}(\xi_{i},2\delta)\setminus B_{g_{\xi_{i}}}(\xi_{i},\delta)}\Lambda_{\xi_{i}}^{-1}V_{\xi_{j},\mu_{j}}\,dv_{g_{\xi_{i}}}\notag\\
&~\lesssim(\mu_{i}\mu_{j})^{\frac{n-2k}{2}}\delta^{-2k}\int_{B_{g_{\xi_{i}}}(\xi_{i},2\delta)\setminus B_{g_{\xi_{i}}}(\xi_{i},\delta)}d_{g_{\xi_{j}}}(x,\xi_{j})^{2k-n}\,dv_{g_{\xi_{i}}}\notag\\
&~\lesssim(\mu_{i}\mu_{j})^{\frac{n-2k}{2}}\mathds{1}_{\{d_{g_{\xi_{i}}}(\xi_{i},\xi_{j})\leq3\delta\}}+\(\frac{\mu_{i}\mu_{j}}{d_{g}(\xi_{i},\xi_{j})^{2}}\)^{\frac{n-2k}{2}}\delta^{n-2k}\mathds{1}_{\{d_{g_{\xi_{i}}}(\xi_{i},\xi_{j})\geq3\delta\}}\notag\\
&~\lesssim\(\frac{\mu_{i}\mu_{j}}{\mu_{i}^{2}+d_{g}(\xi_{i},\xi_{j})^{2}}\)^{\frac{n-2k}{2}}\delta^{n-2k}.
\end{align}
\smallskip

\noindent
So we have obtained that
\begin{align*}
\int_{M}|P_{g}V_{\xi_{i},\mu_{i}}-&~V_{\xi_{i},\mu_{i}}^{2^{*}_{k}-1}|V_{\xi_{j},\mu_{j}}\,dv_{g}\lesssim\\ 
&\(\frac{\mu_{i}\mu_{j}}{\mu_{i}^{2}+d_{g}(\xi_{i},\xi_{j})^{2}}\)^{\frac{n-2k}{2}}\[(\mu_{i}^{4}+\delta^{4})+\delta^{n-2k}+\(\frac{\mu_{i}}{\delta}\)^{2k}\].
\end{align*}
\end{proof}

\begin{remark}
As in the earlier estimate, the first term in \eqref{int.3} is not present if $(M,g)$ is locally conformally flat. And similarly
\begin{align*}
&\left|~\int_{M}(P_{g}U_{\xi_{i},\mu_{i}}-U_{\xi_{i},\mu_{i}}^{2^{*}_{k}-1})U_{\xi_{j},\mu_{j}}\,dv_{g}\right|\lesssim\(\frac{\mu_{i}}{\mu_{j}}+\frac{d_{g}(\xi_{i},\xi_{j})}{\mu_{i}\mu_{j}}\)^{\frac{2k-n}{2}}\(\mu_{i}^{4}+\delta^{4}\). 
\end{align*}
\end{remark}
\bigskip

Consider a colllection of $d\in\N$ points $\xi_{i}\in M$, $1\le i\le d$, and $d\in\N$ scales $0<\mu_{i}<\delta$, $1\le i\le d$ such that the error estimate \eqref{error1} holds. Before proceeding further,  we define the following quantities for $i\ne j$.
\begin{align}\label{int.quant}
&\mathfrak{L}_{ij}:=\int_{M}V_{\xi_{j},\mu_{j}}\,P_{g}V_{\xi_{i},\mu_{i}}\,dv_{g},~\mathcal{Q}_{ij}:=\int_{M}V_{\xi_{i}\mu_{i}}^{\,2^{*}_{k}-1}\,V_{\xi_{j},\mu_{j}}\,dv_{g}~\hbox{ and }\notag\\
&\varepsilon_{ij}:=\(\frac{\mu_{i}}{\mu_{j}}+\frac{\mu_{j}}{\mu_{i}}+\mathfrak{c}_{n,k}^{-1}\frac{(b_{n,k}^{-1}G_{g}(\xi_{i},\xi_{j}))^{\frac{2}{2k-n}}}{\mu_{i}\mu_{j}}\)^{\frac{2k-n}{2}},
\end{align}
with the convention that $G_{g}(\xi,\xi)^{{2}/{2k-n}}:=0$. Note that $\mathfrak{L}_{ji}=\mathfrak{L}_{ij}$ (by the self-adjointness of $P_{g}$) and $\varepsilon_{ji}=\varepsilon_{ij}$ by the symmetry of the Green's function. Also, by the bounds on the Green's function \eqref{bounds.Green} one has
\begin{align*}
\varepsilon_{ij}\approx\(\frac{\mu_{i}}{\mu_{j}}+\frac{\mu_{j}}{\mu_{i}}+\mathfrak{c}_{n,k}^{-1}\frac{d_{g}(\xi_{i},\xi_{j})^{2}}{\mu_{i}\mu_{j}}\)^{\frac{2k-n}{2}} .
\end{align*}
We first note the following.
\begin{corollary} 
Let $k\ge1$ be an integer, and $\(M,g\)$ be a smooth, closed Riemannian manifold of dimension $2k+1\le n\le 2k+3$ or $(M,g) $ be locally conformally flat. Assume that the GJMS operator $P_{g}$ satisfies \eqref{positivity}, and fix  $0<\delta<1$. Then for all $\xi_{i}, \xi_{j}\in M$ and $0<\mu_{j}\leq\mu_{i}<\delta$ sufficiently small, and such that the error estimate \eqref{error1} holds, we obtain the following relation.
\begin{align}\label{int.4}
&\mathfrak{L}_{ij}= \mathcal{Q}_{ij}+\varepsilon_{ij}\bigO\(\mu_{i}^{4}+\delta^{n-2k}+\(\frac{\mu_{i}}{\delta}\)^{2k}\),
\end{align}
with the constants in $\bigO(\cdot)$ depending only on $n,k$ and $(M,g)$.
\end{corollary} 
\begin{proof}
We write 
We have $\mathfrak{L}_{ij}=\displaystyle \int_{M}(P_{g}V_{\xi_{i},\mu_{i}}-\,V_{\xi_{i},\mu_{i}}^{2^{*}_{k}-1})V_{\xi_{j},\mu_{j}}\,dv_{g}+ \mathcal{Q}_{ij}$. From \eqref{int.3} and since $\mu_{j}\leq\mu_{i}$ we have
\begin{align*}
\left|\,\mathfrak{L}_{ij}- \mathcal{Q}_{ij}\right|&\,\lesssim\(\frac{\mu_{i}}{\mu_{j}}+\frac{d_{g}(\xi_{i},\xi_{j})}{\mu_{i}\mu_{j}}\)^{\frac{2k-n}{2}}\(\mu_{i}^{4}+\delta^{n-2k}+\(\frac{\mu_{i}}{\delta}\)^{2k}\)\notag\\
&\lesssim\varepsilon_{ij}\(\mu_{i}^{4}+\delta^{n-2k}+\(\frac{\mu_{i}}{\delta}\)^{2k}\).
\end{align*}
\end{proof}
\medskip

\noindent
Next, we obtain the following crucial reation between the interaction quantities. 
\begin{lemma} 
Let $k\ge1$ be an integer, and $\(M,g\)$ be a smooth, closed Riemannian manifold of dimension $2k+1\le n\le 2k+3$ or $(M,g) $ be locally conformally flat. Assume that the GJMS operator $P_{g}$ satisfies \eqref{positivity}, and let $0<\delta<1$. Then for all $\xi_{i},\xi_{j}\in M$ and $0<\mu_{i}\leq\mu_{j}<\delta$ sufficiently small, and such that the error estimate \eqref{error1} holds, we have the following estimate as $\varepsilon_{ij}\to0$ for $i\ne j$.
\begin{align}\label{int.5}
\mathcal{Q}_{ij}=\varepsilon_{ij}\|B_{0}\|_{2^{*}_{k}-1}^{2^{*}_{k}-1}\,+\bigO\Bigg(\frac{\mu_{i}}{\mu_{j}}\Big(\varepsilon_{ij}^{\frac{2}{n-2k}}+\varepsilon_{ij}^{\frac{2}{n-2k}}\log\varepsilon_{ij}^{-1}\,\mathds{1}_{\{k=1\}}\Big)+\delta+\mu_{j}^{2}\,\delta^{-2}\Bigg),
\end{align}
with the constants in $\bigO(\cdot)$ depending only on $n,k$ and $(M,g)$. Here $B_{0}$ is the canonical bubble defined in \eqref{bubble1}. 
\end{lemma}
\begin{proof}
We write 
\begin{align*}
\mathcal{Q}_{ij}:=\int_{M}V_{\xi_{i},\mu_{i}}^{\,2^{*}_{k}-1}\,V_{\xi_{j},\mu_{j}}\,dv_{g}=\int_{B_{g_{\xi_{i}}}(\xi_{i},\delta)}+\int_{M\setminus B_{g_{\xi_{i}}}(\xi_{i},\delta)}
\end{align*}
Taking  $\delta$ sufficiently small, we can rewrite our bubbles in terms of the Green's function as:
\begin{align}\label{bubble4}
\widetilde{U}_{\xi,\mu}\(x\)&=\chi_{\xi,\delta}(x)\,\(\frac{\mu}{\mu^{2}+\mathfrak{c}^{-1}_{n,k}\big(b_{n,k}^{-1}G_{g_{\xi}}(x,\xi)\big)^{\frac{2}{2k-n}}}\)^{\frac{n-2k}{2}}(1+\bigO(\delta)),\notag\\
\widetilde{V}_{\xi,\mu}\(x\)&=\(\frac{\mu}{\mu^{2}+\mathfrak{c}^{-1}_{n,k}\big(b_{n,k}^{-1}G_{g_{\xi}}(x,\xi)\big)^{\frac{2}{2k-n}}}\)^{\frac{n-2k}{2}}\(1+\bigO(\delta)+\bigO\(\frac{\mu^{2}}{\delta^{2}}\)\).
\end{align}
\smallskip

\noindent
For the outer term, proceeding as in \eqref{Int:ext}, we obtain 
\begin{align}
\int_{M\setminus B_{g_{\xi_{i}}}(\xi_{i},\delta)}V_{\xi_{i},\mu_{i}}^{\,2^{*}_{k}-1}\,V_{\xi_{j},\mu_{j}}\,dv_{g}\lesssim\(\frac{\mu_{i}\mu_{j}}{\mu_{i}^{2}+d_{g}(\xi_{i},\xi_{j})^{2}}\)^{\frac{n-2k}{2}}\(\frac{\mu_{i}}{\delta}\)^{2k}.
\end{align}
\smallskip

\noindent
For the main term, we write 
$$\int_{B_{g_{\xi_{i}}}(\xi_{i},\delta)}V_{\xi_{i},\mu_{i}}^{\,2^{*}_{k}-1}\,V_{\xi_{j},\mu_{j}}\,dv_{g}=\(1+\bigO(\delta)+\bigO\bigg(\frac{\mu_{j}^{2}}{\delta^{2}}\bigg)\)\mathscr{I}_{ij},$$
where 
\begin{align}\label{q_{ij}eq1}
\mathscr{I}_{ij}:=&\bigintss_{B_{g_{\xi_{i}}}(\xi_{i},\delta)}\frac{\Lambda_{\xi_{j}}(x)}{\Lambda_{\xi_{i}}(x)}\(\frac{\mu_{i}}{\mu_{i}^{2}+\mathfrak{c}^{-1}_{n,k}d_{g_{\xi_{i}}}(x,\xi_{i})^{2}}\)^{\frac{n+2k}{2}}\notag\\
&\hspace{4cm}\times \(\frac{\mu_{j}}{\mu_{j}^{2}+\mathfrak{c}^{-1}_{n,k}\big(b_{n,k}^{-1}G_{g_{\xi_{j}}}(x,,\xi_{j})\big)^{\frac{2}{2k-n}}}\)^{\frac{n-2k}{2}}dv_{g_{\xi_{i}}}\notag\\
&=\bigintss_{B(0,\delta/\mu_{i})}\frac{\Lambda_{\xi_{j}}(\exp_{\xi_{i}}^{g_{\xi_{i}}}(\mu_{i}x))}{\Lambda_{\xi_{i}}(\exp_{\xi_{i}}^{g_{\xi_{i}}}(\mu_{i}x))}\(\frac{1}{1+\mathfrak{c}^{-1}_{n,k}|x|^{2}}\)^{\frac{n+2k}{2}}\notag\\
&\hspace{3cm}\times \(\,\dfrac{\mu_{j}}{\mu_{i}}+\mathfrak{c}^{-1}_{n,k}\dfrac{\big(b_{n,k}^{-1}G_{g_{\xi_{j}}}\big(\exp_{\xi_{i}}^{g_{\xi_{i}}}(\mu_{i}x),\xi_{j}\big)\big)^{\frac{2}{2k-n}}}{\mu_{i}\mu_{j}}\)^{\frac{2k-n}{2}}dx.
\end{align}
For convinience we let $\mathscr{K}_{ji}(x):=\big(b_{n,k}^{-1}G_{g_{\xi_{j}}}(\exp_{\xi_{i}}^{g_{\xi_{i}}}(x),\xi_{j})\big)^{2/(2k-n)}$. From the expansion of the Green's function \eqref{expan.green1} we have for $\xi_{i}\neq \xi_{j}$
\begin{align}
\big(b_{n,k}^{-1}G_{g_{\xi_{j}}}\big(\exp_{\xi_{i}}^{g_{\xi_{i}}}(\mu_{i}x),\xi_{j}\big)\big)^{\frac{2}{2k-n}}=&\,\big(b_{n,k}^{-1}G_{g_{\xi_{j}}}(\xi_{i},\xi_{j})\big)^{\frac{2}{2k-n}}\notag\\
&+\mu_{i}\left\langle\nabla\mathscr{K}_{ji}(0), x\right\rangle+\bigO\Big(\mu_{i}^{2}|x|^{2} \underbrace{|\nabla^{2}\mathscr{K}_{ji}(\mu_{i}x)|}_{\le1 \text{ a.e }}\Big).
\end{align}
Note, by \eqref{bounds.Green} we indeed have $\big(b_{n,k}^{-1}G_{g_{\xi}}\big(\exp_{\xi}^{g_{\xi}}(\mu x),\xi \big)\big)^{\frac{2}{2k-n}}=\bigO((\mu ^{2}|x|^{2})$, and the bound $|\nabla\mathscr{K}_{ji}(0)|\lesssim\big(b_{n,k}^{-1}G_{g_{\xi_{j}}}(\xi_{i},\xi_{j})\big)^{\frac{1}{2k-n}}$ for $\xi_{i}\ne\xi_{j}$. Then 
\begin{align*}
&\frac{\mu_{j}}{\mu_{i}}+\mathfrak{c}^{-1}_{n,k}\dfrac{\big(b_{n,k}^{-1}G_{g_{\xi_{j}}}\big(\exp_{\xi_{i}}^{g_{\xi_{i}}}(\mu_{i}x),\xi_{j}\big)\big)^{\frac{2}{2k-n}}}{\mu_{i}\mu_{j}}=\(\frac{\mu_{j}}{\mu_{i}}+\mathfrak{c}^{-1}_{n,k}\frac{\big(b_{n,k}^{-1}G_{g_{\xi_{j}}}(\xi_{i},\xi_{j})\big)^{\frac{2}{2k-n}}}{\mu_{i}\mu_{j}}\)\\
&\qquad\times\(1+\frac{\left\langle\nabla\mathscr{K}_{ji}(0),x/\mu_{j}\right\rangle+ \bigO\bigg(\dfrac{\mu_{i}}{\mu_{j}}|x|^{2}\bigg)}{\dfrac{\mu_{j}}{\mu_{i}}+\mathfrak{c}^{-1}_{n,k}\dfrac{\big(b_{n,k}^{-1}G_{g_{\xi_{j}}}(\xi_{i},\xi_{j})\big)^{\frac{2}{2k-n}}}{\mu_{i}\mu_{j}}}\).
\end{align*}

\noindent
Fix $\epsilon>0$ small and set $\mathcal{T}_{\epsilon}:=\left\{x\in\R^{n}:|x|\leq \epsilon \(\dfrac{\mu_{j}}{\mu_{i}}+\dfrac{\big(b_{n,k}^{-1}G_{g_{\xi_{j}}}(\xi_{i},\xi_{j})\big)^{\frac{1}{2k-n}}}{\mu_{i}}\)\right\}$. 
Then  by a Taylor expansion for $x\in \mathcal{T}_{\epsilon}$ we obtain 
\begin{align}\label{q_{ij}eq2}
&\(\,\frac{\mu_{j}}{\mu_{i}}+\mathfrak{c}^{-1}_{n,k}\dfrac{\big(b_{n,k}^{-1}G_{g_{\xi_{j}}}\big(\exp_{\xi_{i}}^{g_{\xi_{i}}}(\mu_{i}x),\xi_{j}\big)\big)^{\frac{2}{2k-n}}}{\mu_{i}\mu_{j}}\)^{\frac{2k-n}{2}}=\notag\\
&\(\frac{\mu_{j}}{\mu_{i}}+\mathfrak{c}^{-1}_{n,k}\frac{\big(b_{n,k}^{-1}G_{g_{\xi_{j}}}(\xi_{i},\xi_{j})\big)^{\frac{2}{2k-n}}}{\mu_{i}\mu_{j}}\)^{\frac{2k-n}{2}}\(1+\frac{\left\langle\nabla\mathscr{K}_{ji}(0),x/\mu_{j}\right\rangle+ \bigO\bigg(\dfrac{\mu_{i}}{\mu_{j}}|x|^{2}\bigg)}{\dfrac{\mu_{j}}{\mu_{i}}+\mathfrak{c}^{-1}_{n,k}\dfrac{\big(b_{n,k}^{-1}G_{g_{\xi_{j}}}(\xi_{i},\xi_{j})\big)^{\frac{2}{2k-n}}}{\mu_{i}\mu_{j}}}\).
\end{align}
We now break the integral as:  
$$\mathscr{I}_{ij}=\int_{B(0,\delta/\mu_{i})\cap \mathcal{T}_{\epsilon}}+\int_{B(0,\delta/\mu_{i})\setminus \mathcal{T}_{\epsilon}}:=\mathscr{I}_{ij}^{a}+\mathscr{I}_{ij}^{b}.$$
\smallskip

\noindent
From \eqref{q_{ij}eq2} it follows that 
\begin{align*}
&\mathscr{I}^{\text{a}}_{ij}=\(\frac{\mu_{j}}{\mu_{i}}+\mathfrak{c}^{-1}_{n,k}\frac{\big(b_{n,k}^{-1}G_{g_{\xi_{j}}}(\xi_{i},\xi_{j})\big)^{\frac{2}{2k-n}}}{\mu_{i}\mu_{j}}\)^{\frac{2k-n}{2}}\bigintss_{B(0,\delta/\mu_{i})\cap \mathcal{T}_{\epsilon}} \frac{\Lambda_{\xi_{j}}(\exp_{\xi_{i}}^{g_{\xi_{i}}}(\mu_{i}x))}{\Lambda_{\xi_{i}}(\exp_{\xi_{i}}^{g_{\xi_{i}}}(\mu_{i}x))}\notag\\
&~\times \(\frac{1}{1+\mathfrak{c}^{-1}_{n,k}|x|^{2}}\)^{\frac{n+2k}{2}}\(1+\frac{\left\langle\nabla\mathscr{K}_{ji}(0),x/\mu_{j}\right\rangle+ \bigO\bigg(\dfrac{\mu_{i}}{\mu_{j}}|x|^{2}\bigg)}{\dfrac{\mu_{j}}{\mu_{i}}+\mathfrak{c}^{-1}_{n,k}\dfrac{\big(b_{n,k}^{-1}G_{g_{\xi_{j}}}(\xi_{i},\xi_{j})\big)^{\frac{2}{2k-n}}}{\mu_{i}\mu_{j}}}\)\,dx.
\end{align*}
Using $\eqref{conf.1}$ we obtain 
\begin{align}\label{q_{ij}eq3}
&\mathscr{I}^{\text{a}}_{ij}=\(\frac{\mu_{j}}{\mu_{i}}+\mathfrak{c}^{-1}_{n,k}\frac{\big(b_{n,k}^{-1}G_{g_{\xi_{j}}}(\xi_{i},\xi_{j})\big)^{\frac{2}{2k-n}}}{\mu_{i}\mu_{j}}\)^{\frac{2k-n}{2}}\Lambda_{\xi_{j}}(\xi_{i})\bigintss_{B(0,\delta/\mu_{i})\cap \mathcal{T}_{\epsilon}} \big(1+\bigO(\mu_{i}|x|)\big)\notag\\
&\times \(\frac{1}{1+\mathfrak{c}^{-1}_{n,k}|x|^{2}}\)^{\frac{n+2k}{2}}\(1+\frac{\left\langle\nabla\mathscr{K}_{ji}(0),x/\mu_{j}\right\rangle+ \bigO\bigg(\dfrac{\mu_{i}}{\mu_{j}}|x|^{2}\bigg)}{\dfrac{\mu_{j}}{\mu_{i}}+\mathfrak{c}^{-1}_{n,k}\dfrac{\big(b_{n,k}^{-1}G_{g_{\xi_{j}}}(\xi_{i},\xi_{j})\big)^{\frac{2}{2k-n}}}{\mu_{i}\mu_{j}}}\)\,dx\notag\\
=&\(\frac{\mu_{j}}{\mu_{i}}+\mathfrak{c}^{-1}_{n,k}\frac{\big(b_{n,k}^{-1}G_{g_{\xi_{j}}}(\xi_{i},\xi_{j})\big)^{\frac{2}{2k-n}}}{\mu_{i}\mu_{j}}\)^{\frac{2k-n}{2}}\Lambda_{\xi_{j}}(\xi_{i})\bigintss_{B(0,\delta/\mu_{i})\cap \mathcal{T}_{\epsilon}}\(\frac{1}{1+\mathfrak{c}^{-1}_{n,k}|x|^{2}}\)^{\frac{n+2k}{2}}\times\notag\\
&\(1+\frac{\left\langle\nabla\mathscr{K}_{ji}(0),x/\mu_{j}\right\rangle+ \bigO\bigg(\dfrac{\mu_{i}}{\mu_{j}}|x|^{2}\bigg)}{\dfrac{\mu_{j}}{\mu_{i}}+\mathfrak{c}^{-1}_{n,k}\dfrac{\big(b_{n,k}^{-1}G_{g_{\xi_{j}}}(\xi_{i},\xi_{j})\big)^{\frac{2}{2k-n}}}{\mu_{i}\mu_{j}}}\)\,dx+\(\frac{\mu_{j}}{\mu_{i}}+\mathfrak{c}^{-1}_{n,k}\frac{\big(b_{n,k}^{-1}G_{g_{\xi_{j}}}(\xi_{i},\xi_{j})\big)^{\frac{2}{2k-n}}}{\mu_{i}\mu_{j}}\)^{\frac{2k-n}{2}}\notag\\
&\times\bigO\(\mu_{i}\bigintss_{B(0,\delta/\mu_{i})\cap \mathcal{T}_{\epsilon}}\frac{|x|}{\(1+\mathfrak{c}^{-1}_{n,k}|x|^{2}\)^{\frac{n+2k}{2}}}\).
\end{align}
\smallskip

\noindent
We estimate each of the above terms. 
\begin{align*}
&\int_{B(0,\delta/\mu_{i})\cap \mathcal{T}_{\epsilon}} \(\frac{1}{1+\mathfrak{c}^{-1}_{n,k}|x|^{2}}\)^{\frac{n+2k}{2}}dx=\|B_{0}\|_{2^{*}_{k}-1}^{2^{*}_{k}-1}\,+\bigO\bigg(\frac{\mu_{i}^{\,2k}}{\delta^{\,2k}}\bigg)\\
&\hspace{2.5cm}+\,\(\dfrac{\mu_{j}}{\mu_{i}}+\dfrac{\big(b_{n,k}^{-1}G_{g_{\xi_{j}}}(\xi_{i},\xi_{j})\big)^{\frac{2}{2k-n}}}{\mu_{i}\mu_{j}}\)^{-k}\bigO\bigg(\frac{\mu_{i}^{\,k}}{\mu_{j}^{\,k}}\bigg).
\end{align*}
By symmetry
\begin{align*}
&\int_{B(0,\delta/\mu_{i})\cap \mathcal{T}_{\epsilon}}\left\langle\nabla\mathscr{K}_{ji}(0),x/\mu_{j}\right\rangle\(\frac{1}{1+\mathfrak{c}^{-1}_{n,k}|x|^{2}}\)^{\frac{n+2k}{2}}dx=0.
\end{align*}
And
\begin{align*}
\int_{B(0,\delta/\mu_{i})\cap\mathcal{T}_{\epsilon}}&\frac{|x|^{2}}{\(1+\mathfrak{c}^{-1}_{n,k}|x|^{2}\)^{\frac{n+2k}{2}}}\,dx\,\lesssim&\\
&\left\{\begin{aligned}\\&~1~&&\text{for } k\ge2,\\&\log\(2\Bigg(\frac{\mu_{j}}{\mu_{i}}+\mathfrak{c}^{-1}_{n,k}\frac{\big(b_{n,k}^{-1}G_{g_{\xi_{j}}}(\xi_{i},\xi_{j})\big)^{\frac{2}{2k-n}}}{\mu_{i}\mu_{j}}\Bigg)\)&&\text{for }k=1.\end{aligned}\right.
\end{align*}
\smallskip

\noindent
Plugging the above estimates in \eqref{q_{ij}eq3} yeilds 
\begin{align}\label{q_{ij}eq4}
&\mathscr{I}^{\text{a}}_{ij}=\(\frac{\mu_{j}}{\mu_{i}}+\mathfrak{c}^{-1}_{n,k}\frac{\big(b_{n,k}^{-1}G_{g_{\xi_{j}}}(\xi_{i},\xi_{j})\big)^{\frac{2}{2k-n}}}{\mu_{i}\mu_{j}}\)^{\frac{2k-n}{2}}\Bigg[\Lambda_{\xi_{j}}(\xi_{i})\|B_{0}\|_{2^{*}_{k}-1}^{2^{*}_{k}-1}+\bigO(\mu_{i})+\bigO\bigg(\frac{\mu_{i}^{\,2k}}{\delta^{\,2k}}\bigg)\notag\\
&+\Bigg(\frac{\mu_{j}}{\mu_{i}}+\frac{\big(b_{n,k}^{-1}G_{g_{\xi_{j}}}(\xi_{i},\xi_{j})\big)^{\frac{2}{2k-n}}}{\mu_{i}\mu_{j}}\Bigg)^{-k}\bigO\bigg(\frac{\mu_{i}^{\,k}}{\mu_{j}^{\,k}}\bigg)+\Bigg(\frac{\mu_{j}}{\mu_{i}}+\frac{\big(b_{n,k}^{-1}G_{g_{\xi_{j}}}(\xi_{i}\,\xi_{j})\big)^{\frac{2}{2k-n}}}{\mu_{i}\mu_{j}}\Bigg)^{-1}\notag\\
&\times\Bigg(\bigO\bigg(\frac{\mu_{i}}{\mu_{j}}\bigg)\mathds{1}_{\{k\geq2\}}+\log\Bigg(2\Bigg(\frac{\mu_{j}}{\mu_{i}}+\mathfrak{c}^{-1}_{n,k}\frac{\big(b_{n,k}^{-1}G_{g_{\xi_{j}}}(\xi_{i},\xi_{j})\big)^{\frac{2}{2k-n}}}{\mu_{i}\mu_{j}}\Bigg)\Bigg)\bigO\bigg(\frac{\mu_{i}}{\mu_{j}}\bigg)\mathds{1}_{\{k=1\}}\Bigg)\Bigg].\notag\\
\end{align}
\smallskip

\noindent
We now deal with the $\mathscr{I}^{\text{b}}_{ij}$ term. For the case $d_{g_{\xi_{i}}}(\xi_{i},\xi_{j})\geq 2\delta$, using the bounds on the Green's function \eqref{bounds.Green} and the  equivalence of metrics we obtain
\begin{align*}
&\mathscr{I}^{\text{b}}_{ij}\lesssim\bigintss_{B(0,\delta/\mu_{i})\setminus \mathcal{T}_{\epsilon}}\(\frac{1}{1+\mathfrak{c}^{-1}_{n,k}|x|^{2}}\)^{\frac{n+2k}{2}}\notag\\
&\hspace{4cm}\times\(\,\frac{\mu_{j}}{\mu_{i}}+\mathfrak{c}^{-1}_{n,k}\frac{\big(b_{n,k}^{-1}G_{g_{\xi_{j}}}(\exp_{\xi_{i}}^{g_{\xi_{i}}}(\mu_{i}x),\xi_{j})\big)^{\frac{2}{2k-n}}}{\mu_{i}\mu_{j}}\)^{\frac{2k-n}{2}}dx\notag\\
&\lesssim\(\frac{\mu_{j}}{\mu_{i}}+\mathfrak{c}^{-1}_{n,k}\frac{\big(b_{n,k}^{-1}G_{g_{\xi_{j}}}(\xi_{i},\xi_{j})\big)^{\frac{2}{2k-n}}}{\mu_{i}\mu_{j}}\)^{\frac{2k-n}{2}}\bigintss_{B(0,\delta/\mu_{i})\setminus \mathcal{T}_{\epsilon}}\(\frac{1}{1+\mathfrak{c}^{-1}_{n,k}|x|^{2}}\)^{\frac{n+2k}{2}}\notag\\
&\lesssim\(\frac{\mu_{j}}{\mu_{i}}+\mathfrak{c}^{-1}_{n,k}\frac{\big(b_{n,k}^{-1}G_{g_{\xi_{j}}}(\xi_{i},\xi_{j})\big)^{\frac{2}{2k-n}}}{\mu_{i}\mu_{j}}\)^{-n/2} \(\frac{\mu_{i}}{\mu_{j}}\)^{k}.
\end{align*}
Now assume that $d_{g_{\xi_{i}}}(\xi_{i},\xi_{j})\le 2\delta$. Denote $\rho_{ji}:=\(\dfrac{\mu_{j}}{\mu_{i}}+\dfrac{\big(b_{n,k}^{-1}G_{g_{\xi_{j}}}(\xi_{i},\xi_{j})\big)^{\frac{1}{2k-n}}}{\mu_{i}}\)$ and $\Theta_{\xi_{i}\xi_{j}}:=(\exp_{\xi_{i}}^{g_{\xi_{i}}})^{-1}(\xi_{j})$.  For an appropriately chosen large constant $C>0$, we obtain
\begin{align*}
&\mathscr{I}^{\text{b}}_{ij}\lesssim\bigintss_{B(0,\delta/\mu_{i})\setminus \mathcal{T}_{\epsilon}}\(\frac{1}{1+\mathfrak{c}^{-1}_{n,k}|x|^{2}}\)^{\frac{n+2k}{2}}\notag\\
&\hspace{4cm}\times\(\,\frac{\mu_{j}}{\mu_{i}}+\mathfrak{c}^{-1}_{n,k}\frac{\big(b_{n,k}^{-1}G_{g_{\xi_{j}}}(\exp_{\xi_{i}}^{g_{\xi_{i}}}(\mu_{i}x),\xi_{j})\big)^{\frac{2}{2k-n}}}{\mu_{i}\mu_{j}}\)^{\frac{2k-n}{2}}dx\notag\\
&\lesssim\bigintss_{\{\epsilon \rho_{ji}\le |x|\le 2C\rho_{ji}\}\cap B(0,\delta/\mu_{i})}+\bigintss_{\{|x|\ge C\rho_{ji}\}\cap B(0,\delta/\mu_{i})}
\end{align*}
\begin{align*}
&\lesssim\,(\epsilon\rho_{ji})^{-(n+2k)}\bigints_{\{|x-\Theta_{\xi_{i}\xi_{j}}/\mu_{i}|\le 4C\rho_{ji}\}}\(\frac{\mu_{j}/\mu_{i}}{{\mu_{j}^{2}}/{\mu_{i}^{2}}+\Big|x-\frac{\Theta_{\xi_{i}\xi_{j}}}{\mu_{i}}\Big|^{2}}\)^{\frac{n-2k}{2}}+\notag\\
&~\(\frac{\mu_{j}}{\mu_{i}}+\mathfrak{c}^{-1}_{n,k}\frac{\big(b_{n,k}^{-1}G_{g_{\xi_{j}}}(\xi_{i},\xi_{j})\big)^{\frac{2}{2k-n}}}{\mu_{i}\mu_{j}}\)^{\frac{2k-n}{2}}\bigints_{|x|\ge C\rho_{ji}}\frac{1}{\(1+\mathfrak{c}^{-1}_{n,k}|x|^{2}\)^{\frac{n+2k}{2}}}\,dx\notag\\
&\lesssim (\epsilon\rho_{ji})^{-n}\(\frac{\mu_{j}}{\mu_{i}}\)^{\frac{n-2k}{2}}+\(\frac{\mu_{j}}{\mu_{i}}+\mathfrak{c}^{-1}_{n,k}\frac{\big(b_{n,k}^{-1}G_{g_{\xi_{j}}}(\xi_{i},\xi_{j})\big)^{\frac{2}{2k-n}}}{\mu_{i}\mu_{j}}\)^{\frac{2k-n}{2}}\rho_{ji}^{-2k}\notag\\
&\lesssim\(\frac{\mu_{j}}{\mu_{i}}+\mathfrak{c}^{-1}_{n,k}\frac{\big(b_{n,k}^{-1}G_{g_{\xi_{j}}}(\xi_{i},\xi_{j})\big)^{\frac{2}{2k-n}}}{\mu_{i}\mu_{j}}\)^{-n/2} \(\frac{\mu_{i}}{\mu_{j}}\)^{k}.
\end{align*}
Thus
\begin{align}\label{q_{ij}eq5}
\mathscr{I}^{\text{b}}_{ij}\lesssim\(\frac{\mu_{j}}{\mu_{i}}+\mathfrak{c}^{-1}_{n,k}\frac{\big(b_{n,k}^{-1}G_{g_{\xi_{j}}}(\xi_{i},\xi_{j})\big)^{\frac{2}{2k-n}}}{\mu_{i}\mu_{j}}\)^{-n/2} \(\frac{\mu_{i}}{\mu_{j}}\)^{k}.
\end{align}
\smallskip

\noindent
Combining \eqref{q_{ij}eq4} and \eqref{q_{ij}eq5} we get that 
\begin{align}
&\mathscr{I}_{ij}=\(\frac{\mu_{j}}{\mu_{i}}+\mathfrak{c}^{-1}_{n,k}\frac{\big(b_{n,k}^{-1}G_{g_{\xi_{j}}}(\xi_{i},\xi_{j})\big)^{\frac{2}{2k-n}}}{\mu_{i}\mu_{j}}\)^{\frac{2k-n}{2}}\Bigg[\Lambda_{\xi_{j}}(\xi_{i})\|B_{0}\|_{2^{*}_{k}-1}^{2^{*}_{k}-1}+\bigO(\mu_{i})+\bigO\bigg(\frac{\mu_{i}^{\,2k}}{\delta^{\,2k}}\bigg)\notag\\
&+\Bigg(\frac{\mu_{j}}{\mu_{i}}+\frac{\big(b_{n,k}^{-1}G_{g_{\xi_{j}}}(\xi_{i},\xi_{j})\big)^{\frac{2}{2k-n}}}{\mu_{i}\mu_{j}}\Bigg)^{-k}\bigO\bigg(\frac{\mu_{i}^{\,k}}{\mu_{j}^{\,k}}\bigg)+\Bigg(\frac{\mu_{j}}{\mu_{i}}+\frac{\big(b_{n,k}^{-1}G_{g_{\xi_{j}}}(\xi_{i}\,\xi_{j})\big)^{\frac{2}{2k-n}}}{\mu_{i}\mu_{j}}\Bigg)^{-1}\notag\\
&\times\Bigg(\bigO\bigg(\frac{\mu_{i}}{\mu_{j}}\bigg)\mathds{1}_{\{k\geq2\}}+\log\Bigg(2\Bigg(\frac{\mu_{j}}{\mu_{i}}+\mathfrak{c}^{-1}_{n,k}\frac{\big(b_{n,k}^{-1}G_{g_{\xi_{j}}}(\xi_{i},\xi_{j})\big)^{\frac{2}{2k-n}}}{\mu_{i}\mu_{j}}\Bigg)\Bigg)\bigO\bigg(\frac{\mu_{i}}{\mu_{j}}\bigg)\mathds{1}_{\{k=1\}}\Bigg)\Bigg].\notag\\
\end{align}
Consequently
\begin{align}\label{q_{ij}eq6}
&\mathcal{Q}_{ij}=\(\frac{\mu_{j}}{\mu_{i}}+\mathfrak{c}^{-1}_{n,k}\frac{\big(b_{n,k}^{-1}G_{g_{\xi_{j}}}(\xi_{i},\xi_{j})\big)^{\frac{2}{2k-n}}}{\mu_{i}\mu_{j}}\)^{\frac{2k-n}{2}}\Bigg[\Lambda_{\xi_{j}}(\xi_{i})\|B_{0}\|_{2^{*}_{k}-1}^{2^{*}_{k}-1}+\bigO(\mu_{i})\notag\\
&~+\bigO(\delta)+\bigO\Bigg(\frac{\mu_{j}^{2}}{\delta^{2}}\Bigg)+\bigO\bigg(\frac{\mu_{i}^{\,2k}}{\delta^{\,2k}}\bigg)+\Bigg(\frac{\mu_{j}}{\mu_{i}}+\frac{\big(b_{n,k}^{-1}G_{g_{\xi_{j}}}(\xi_{i},\xi_{j})\big)^{\frac{2}{2k-n}}}{\mu_{i}\mu_{j}}\Bigg)^{-k}\bigO\bigg(\frac{\mu_{i}^{\,k}}{\mu_{j}^{\,k}}\bigg)\notag\\
&~+\Bigg(\frac{\mu_{j}}{\mu_{i}}+\frac{\big(b_{n,k}^{-1}G_{g_{\xi_{j}}}(\xi_{i}\,\xi_{j})\big)^{\frac{2}{2k-n}}}{\mu_{i}\mu_{j}}\Bigg)^{-1}\Bigg(\bigO\bigg(\frac{\mu_{i}}{\mu_{j}}\bigg)\mathds{1}_{\{k\geq2\}}\,+\notag\\
&\qquad\log\Bigg(2\Bigg(\frac{\mu_{j}}{\mu_{i}}\mathfrak{c}^{-1}_{n,k}\frac{\big(b_{n,k}^{-1}G_{g_{\xi_{j}}}(\xi_{i},\xi_{j})\big)^{\frac{2}{2k-n}}}{\mu_{i}\mu_{j}}\Bigg)\Bigg)\bigO\bigg(\frac{\mu_{i}}{\mu_{j}}\bigg)\mathds{1}_{\{k=1\}}\Bigg)\Bigg].
\end{align}
\smallskip

\noindent
When  $\varepsilon_{ij}\to0$ we thus obtain
\begin{align*}
&\mathcal{Q}_{ij}=\varepsilon_{ij}\|B_{0}\|_{2^{*}_{k}-1}^{2^{*}_{k}-1}\Bigg[1+\varepsilon_{ij}^{\frac{2}{n-2k}}\,\bigO\bigg(\frac{\mu_{i}}{\mu_{j}}\bigg)-\varepsilon_{ij}^{\frac{2}{n-2k}}\log\varepsilon_{ij}\,\bigO\bigg(\frac{\mu_{i}}{\mu_{j}}\bigg)\mathds{1}_{\{k=1\}}\\
&~+\varepsilon_{ij}^{\frac{2k}{n-2k}}\,\bigO\bigg(\frac{\mu_{i}^{k}}{\mu_{j}^{k}}\bigg)+\bigO\bigg(\frac{\mu_{i}^{\,2k}}{\delta^{\,2k}}\bigg)+\bigO\bigg(\frac{\mu_{j}^{2}}{\delta^{2}}\bigg)+\bigO(\delta)\Bigg].
\end{align*}
\medskip

\noindent
We note that \eqref{q_{ij}eq6} yields 
\begin{align*}
\mathcal{Q}_{ij}/\varepsilon_{ij}\le C,
\end{align*}
for some constant depending only on $n,k$ and $(M,g)$. From \eqref{q_{ij}eq1} and the equivalence of metrics, we also obtain for $\delta$ sufficiently small
\begin{align*}
&\mathcal{Q}_{ij}\gtrsim \bigintss_{B(0,1)}\(\frac{1}{1+\mathfrak{c}^{-1}_{n,k}|x|^{2}}\)^{\frac{n+2k}{2}}\(\frac{\mu_{j}}{\mu_{i}}+\frac{\big(d_{g_{\xi_{i}}}\big(\exp_{\xi_{i}}^{g_{\xi_{i}}}(\mu_{i}x),\xi_{j}\big)\big)^{2}}{\mu_{i}\mu_{j}}\)^{\frac{2k-n}{2}}dx\notag\\
&\mathcal{Q}_{ij}\gtrsim\varepsilon_{ij}\bigintss_{B(0,1)}\(\frac{1}{1+\mathfrak{c}^{-1}_{n,k}|x|^{2}}\)^{\frac{n+2k}{2}}dx.
\end{align*}
Hence we always have that 
\begin{align}\label{int.5b}
\mathcal{Q}_{ij}\approx\varepsilon_{ij}
\end{align}
\end{proof}
\medskip


\noindent
The following interaction estimates follow from the proof of Lemma 3.4 in \cite{MayNdi2}.  
\begin{lemma}\label{p,q int lemma}
Let $k\ge1$ be an integer, and $\(M,g\)$ be a smooth, closed Riemannian manifold of dimension $2k+1\le n\le 2k+3$ or $(M,g) $ be locally conformally flat. Assume that the GJMS operator $P_{g}$ satisfies \eqref{positivity}, and let $0<\delta<1$. Let $p,q>1$ with $p+q=2^{*}_{k}$. For all $\xi_{i}, \xi_{j}\in M$, $0<\mu_{i}, \mu_{j}<\delta$ sufficiently small, and such that the error estimate \eqref{error1} holds, we have the following interaction estimates for $i\ne j$.
\begin{align}\label{int.6}
\int_{M}V_{\xi_{i}\mu_{i}}^{\,p}V_{\xi_{j},\mu_{j}}^{\,q}\,dv_{g}\lesssim\left\{\begin{aligned}&\varepsilon_{ij}^{\,\frac{n}{n-2k}}\log\(\frac{1}{\varepsilon_{ij}}\)&&\text{if $p=q$},\\&\varepsilon_{ij}^{\,q}&&\text{if $p>2^{*}_{k}/{2}>q$},\end{aligned}\right.
\end{align}
with the constants in $\lesssim$ depending only on $n,k$, $(M,g)$ and $|p-q|$.
\end{lemma}
These estimates are now standard, and we also refer the interested reader to Lemma 3.1 in \cite{MayNdi2} for the interactions between Euclidean bubbles. These results are in the spirit of the interaction estimates of Bahri \cite{Bahri}. 
\medskip

\noindent
Combining \eqref{int.4}, \eqref{int.5} and \eqref{int.5b}, we have thus obtained the following relations between the various interaction quantities defined in \eqref{int.quant}. This is more readily useful in our energy expansions next. But, first we introduce following notations: $\text{b}=\smallo_{\beta}(1)$ denotes $\text{b}=\smallo(1)$ as the parameter $\beta\to0$, $\text{A}=\bigO_{\zeta_{1},\ldots,\zeta_{k}}(\text{B})$ means that $\text{A}=\bigO(\text{B})$ uniformly with respect to the quantities $\zeta_{1},\ldots,\zeta_{k}$. Similarly $\text{b}=\smallo_{\zeta_{1},\ldots,\zeta_{k};\beta}(\text{1})$ means that $\text{b}=\smallo(1)$ as the parameter  $\beta\to0$  uniformly with respect to the quantites $\zeta_{1},\ldots,\zeta_{k}$.
\smallskip

\begin{proposition}\label{prop:interact}
Let $k\ge1$ be an integer, and $\(M,g\)$ be a smooth, closed Riemannian manifold of dimension $2k+1\le n\le 2k+3$ or $(M,g) $ be locally conformally flat. Assume that the GJMS operator $P_{g}$ satisfies \eqref{positivity}, and fix  $0<\delta<1$. Let $d\in \N$, $\Xi_{d}:=\(\xi_{1},\ldots,\xi_{d}\)\in M^{d}$ and let  $\mu_{i}=\mu\in (0,\delta)$ for all $1\leq i\leq d$ be sufficiently small, and such that the error estimate \eqref{error1} holds. Then the following holds. 
\begin{itemize}
\item
There exists $C>0$ such that for all $d,\,\Xi_{d}$ and $\mu$ we have for all $i\neq j$
\begin{align}\label{int.7}
C^{-1}\leq{\mathcal{Q}_{ij}}/{\varepsilon_{ij}}\leq C.
\end{align}
\item As $\mu\to0$ we have for all $1\le i\le d$
\begin{align}\label{int.8}
\int_{M}V_{\xi_{i},\mu}\,P_{g}V_{\xi_{i},\mu}\,dv_{g}=\int_{M}V_{\xi_{i},\mu}^{2^{*}_{k}}\,dv_{g}\(1+\smallo_{\xi_{i},\mu}(1)\).
\end{align}
\item
As $\mu\to0$ we have for $1\leq i\neq j \leq d$ 
\begin{align}\label{int.9}
&\mathfrak{L}_{ij}= \mathcal{Q}_{ij}+\varepsilon_{ij}\bigO_{\Xi_{d}}\(\mu_{i}^{4}+\delta^{n-2k}+\(\frac{\mu_{i}}{\delta}\)^{2k}\),\hbox{ or }\notag\\
&\mathfrak{L}_{ij}= \mathcal{Q}_{ij}\[1+\smallo_{\Xi_{d},\mu}(1)+\bigO_{\Xi_{d}}\(\delta^{n-2k}\)\].
\end{align}
\item
For $1\leq i\neq j \leq d$ we have as $\varepsilon_{ij}\to 0$ 
\begin{align}\label{int.10}
\mathcal{Q}_{ij}=\varepsilon_{ij}\|B_{0}\|_{2^{*}_{k}-1}^{2^{*}_{k}-1}\[1+\smallo_{\Xi_{d};\varepsilon_{ij}}(1)+\bigO_{\Xi_{d}}(\delta)\].
\end{align}
\item
For $1\leq i\neq j \leq d$ we have as $\varepsilon_{ij}\to 0$ 
\begin{align}\label{int.11}
&\mathfrak{L}_{ij}=\varepsilon_{ij}\|B_{0}\|_{2^{*}_{k}-1}^{2^{*}_{k}-1}\[1+\smallo_{\Xi_{d};\varepsilon_{ij}}(1)+\bigO_{\Xi_{d}}(\delta)\].
\end{align}
Here the constants in $\bigO(\cdot)$ depends only on $n,k$ and $(M,g)$ and $B_{0}$ is the canonical bubble defined in \eqref{bubble1}. 
\end{itemize}
The estimates in Proposition \ref{prop:interact} will be instrumental in obtaining the energy expansions next. 
\end{proposition}
\medskip

\section{Energy expansion}\label{sec:Energy}
The goal of this section is to estimate the energy of the sum of bubbles using the interactions. Recall the energy
\begin{align*}
\mathcal{J}_{g,k}(u):=\frac{\displaystyle{\int_{M}u\,P_{g}u\,dv_{g}}}{\(\displaystyle{\int_{M}|u|^{\,2^{*}_{k}}\,dv_{g}}\)^{2/2^{*}_{k}}}.
\end{align*}
Straightforward calculations utilizing \eqref{int.1} and \eqref{int.2} gives the following estimate on the energy of a bubble.

\begin{lemma}
Let $k\ge1$ be an integer, and $\(M,g\)$ be a smooth, closed Riemannian manifold of dimension $2k+1\le n\le 2k+3$ or $(M,g) $ be locally conformally flat. Assume that the GJMS operator $P_{g}$ satisfies \eqref{positivity}. Fix $\xi\in M$, $0<\delta<1$ and choose $0<\mu<\delta$ sufficiently small, and such that the error estimate \eqref{error1} holds. We have 
\begin{align}\label{energy0}
\mathcal{J}_{g,k}(V_{\xi,\mu})=Y_{k}(\S^{n})\(1+\bigO\Big(\mu^{n-2k}+\delta^{4+2k-n}\mu^{n-2k}+\Big(\frac{\mu}{\delta}\Big)^{n}\Big)\).
\end{align}
\end{lemma}
We omit the proof of the above lemma and refer the interested reader to Proposition 3.1 in \cite{MazVetois} for an expression of the energy of the bubble with the mass.
\begin{remark}\label{rem:expn1}
Expansion \eqref{energy0} is a rough estimate and the constant in front of $\bigO(\mu^{n-2k})$ term in \eqref{energy0} is not explicit; it could be given in terms of the mass of $P_{g}$ as in Proposition B.3. of \cite{Bre1} for the $k=1$ case. Since we do not assume any information on the sign of the mass of $P_{g}$ and to keep our calculations succinct, we simply write the constant in front of $\bigO(\mu^{n-2k})$ term as a real number.
\end{remark}
\medskip

One of the key results in this paper is the following expansion of the energy of the sum of $d$-many bubbles all with the same heights $1/\mu\to+\infty$. 
\begin{proposition}
Let $k\ge1$ be an integer, and $\(M,g\)$ be a smooth, closed Riemannian manifold of dimension $2k+1\le n\le 2k+3$ or $(M,g) $ be locally conformally flat. Assume that the GJMS operator $P_{g}$ satisfies \eqref{positivity}. Let $d\in \N$, and $\xi_{i}\in M$ for  $1\leq i\leq d$, $\mu\in (0,\delta)$ sufficiently small, and  such that the error estimate \eqref{error1} holds. There exists $\varepsilon_{0}>0$, $\mu_{0}>0$ and $\delta>0$ small such that for every positive integer $d\geq2$, if $\sum \limits_{i\neq j}^{d}\varepsilon_{ij}<\varepsilon_{0}$ and $0<\mu<\mu_{0}$,  we have the following energy expansion.
\begin{align}\label{energy_sum0}
&\mathcal{J}_{g,k}\Big(\sum \limits_{i=1}^{d}V_{\xi_{i},\mu}\Big)\leq d^{2k/n}Y_{k}(\S^{n})\Bigg(1-\Big(\mathfrak{m}_{g}+(d-1)\,\mathfrak{b}_{n,k}\,\min\limits_{x\neq y}G_{g}(x,y)\Big)\mu^{n-2k}\Bigg),
\end{align}
for some  constant $\mathfrak{m}_{g}\in\R$, $\mathfrak{b}_{n,k}=\dfrac{\kappa_{2^{*}_{k}}-1}{4d}\dfrac{\|B_{0}\|_{2^{*}_{k}-1}^{2^{*}_{k}-1}}{\|B_{0}\|_{2^{*}_{k}}^{2^{*}_{k}}}\mathfrak{c}_{n,k}^{\frac{n-2k}{2}}\,b_{n,k}^{-1}>0$ for some constant $\kappa_{2^{*}_{k}}>1$,  and $B_{0}$ is the canonical bubble defined in \eqref{bubble1}. 
\end{proposition}
The proof of the above energy expansion is contained in expression \eqref{energy_sum3_main} in Proposition \ref{main:energy esti}.
\medskip

\noindent
Let $(a_{1},\ldots,a_{d})\in (0,+\infty)^{d}$ and consider the sum of $d\in\N$ many bubbles with weights $a_{i}$ given by $=\sum \limits_{i=1}^{d}a_{i}V_{\xi_{i}\mu}$, where $\xi_{i}\in M$, $1\le i \le d$. We now estimate the energy of the sum of bubbles with weights. This is one of the main results of this paper and makes the topological arguments in the next section work. 
\noindent

\begin{proposition}\label{main:energy esti}
Let $k\ge1$ be an integer, and $\(M,g\)$ be a smooth, closed Riemannian manifold of dimension $2k+1\le n\le 2k+3$ or $(M,g) $ be locally conformally flat. Assume that the GJMS operator $P_{g}$ satisfies \eqref{positivity}, and let $d\in\N$. Let $\xi_{1},\cdots,\xi_{d}\in M$ and $\mu\in (0,\delta)$ sufficiently small, and such that the error estimate \eqref{error1} holds. We have the following estimates on the energy of the sum of  $d\in\N$ bubbles $\sum \limits_{i=1}^{d}a_{i}V_{\xi_{i}\mu}$, with weights $a_{1},\ldots,a_{d}\in (0,+\infty)$.
\begin{enumerate}
\item
There exists $\mu_{*},\delta_{*}>0$ such that for all $0<\mu\le\mu_{*}$, $0<\delta\le\delta_{*}$ the following bound on the energy hold for all $d\ge2$. 
\begin{align}\label{energy_sum0}
\mathcal{J}_{g,k}\Big(\sum \limits_{i=1}^{d}a_{i}V_{\xi_{i},\mu_{i}}\Big)\le~\(d+1/2\)^{\,2k/n}Y_{k}(\S^{n}).
\end{align}

\item
For every $\varepsilon>0$ small and $d\ge2$, there exists  $\mu_{*}(\varepsilon,d),\delta_{*}(\varepsilon,d)>0$ such that for all $0<\mu\le\mu_{*}(\varepsilon,d)$, $0<\delta\le\delta_{*}(\varepsilon,d)$,  if $\displaystyle \sum\limits_{1\le i\neq j \le d}\varepsilon_{ij}>\varepsilon$, we have:
\begin{align}\label{energy_sum1}
\mathcal{J}_{g,k}\Big(\sum \limits_{i=1}^{d}a_{i}V_{\xi_{i}\mu}\Big)< d^{\,2k/n}Y_{k}(\S^{n}).
\end{align}
\item
For every $\tau>1, \varepsilon>0$ small and $d\ge2$, there exists  $\mu=\mu_{*}(\tau,\varepsilon,d),\delta_{*}(\tau,\varepsilon,d)>0$  such that for all $0<\mu\le\mu_{*}(\varepsilon,d,\tau)$, $0<\delta\le\delta_{*}(\varepsilon,d,\tau)$,  if $\dfrac{a_{i_{0}}}{a_{j_{0}}}>\tau$ for some $i_{0}\neq j_{0}$ and $\displaystyle\sum\limits_{1\le i\neq j\le d}\varepsilon_{ij}\le\varepsilon$,  we have:
\begin{align}\label{energy_sum2}
\mathcal{J}_{g,k}\Big(\sum \limits_{i=1}^{d}a_{i}V_{\xi_{i}\mu_{i}}\Big)< d^{\,2k/n}Y_{k}(\S^{n}).
\end{align}
\item
There exists $\tau_{*}>1$ close to $1$, $\varepsilon_{*}>0$, $\mu_{*}>0$ and $\delta_{*}>0$ small such that for all $d\ge2$, if $\dfrac{a_{i}}{a_{j}}\le\tau_{*}$ for all $i\neq j$ and  $\sum \limits_{1\le i\neq j\le d}\varepsilon_{ij}<\varepsilon_{*}$, then for all $0<\mu\le\mu_{*}$, $0<\delta\le \delta_{*}$ we have the following energy expansion.
\begin{align}\label{energy_sum3_main}
&\mathcal{J}_{g,k}\Big(\sum \limits_{i=1}^{d}a_{i}V_{\xi_{i},\mu}\Big)\leq d^{\,2k/n}\,Y_{k}(\S^{n})\[1-\(\mathfrak{m}_{g}+(d-1)\,\mathfrak{b}_{n,k}\,\min\limits_{x\neq y}G_{g}(x,y)\)\mu^{n-2k}\],
\end{align}
for some  constant $\mathfrak{m}_{g}\in\R$ and $\mathfrak{b}_{n,k}:=\dfrac{\kappa_{\#}}{4}\dfrac{\|B_{0}\|_{2^{*}_{k}-1}^{2^{*}_{k}-1}}{\|B_{0}\|_{2^{*}_{k}}^{2^{*}_{k}}}\mathfrak{c}_{n,k}^{\frac{n-2k}{2}}\,b_{n,k}^{-1}>0$ for some constant $\kappa_{\#}>0$.
\end{enumerate}
\end{proposition}
\begin{proof}
We follow the strategy in Bahri-Brezis \cite{BahriBrezis} and Mayer-Ndiaye \cite{MayNdi1}. We write
\begin{align*}
\mathcal{J}_{g,k}\Big(\sum \limits_{i=1}^{d}a_{i}V_{\xi_{i},\mu}\Big)=&\,\frac{\displaystyle{\int_{M}\Big(\sum \limits_{i=1}^{d}a_{i}V_{\xi_{i},\mu}\Big)\,P_{g}\Big(\sum \limits_{j=1}^{d}a_{j}V_{\xi_{j},\mu}\Big)\,dv_{g}}}{\(\displaystyle{\int_{M}\Big(\sum \limits_{i=1}^{d}a_{i}V_{\xi_{i},\mu}\Big)^{\,2^{*}_{k}}\,dv_{g}}\)^{2/2^{*}_{k}}}:=\frac{\mathcal{N}\Big(\sum \limits_{i=1}^{d}a_{i}V_{\xi_{i},\mu}\Big)}{\mathcal{D}\Big(\sum \limits_{i=1}^{d}a_{i}V_{\xi_{i},\mu}\Big)}.
\end{align*}

\noindent
{\bf{Proof of \eqref{energy_sum1}}:} We first control the numerator using the estimates obtained in Proposition \ref{prop:interact} as follows.
\begin{align}\label{num1}
&\mathcal{N}\Big(\sum \limits_{i=1}^{d}a_{i}V_{\xi_{i},\mu}\Big)=\int_{M}\Big(\sum \limits_{i=1}^{d}a_{i}V_{\xi_{i},\mu}\Big)\,P_{g}\Big(\sum \limits_{j=1}^{d}a_{j}V_{\xi_{j},\mu}\Big)\,dv_{g}\notag\\
&=\int_{M}\Big(\sum \limits_{i=1}^{d}a_{i}V_{\xi_{i},\mu}\Big)\,\Big(\sum \limits_{j=1}^{d}a_{j}V_{\xi_{j},\mu}^{\,2^{*}_{k}-1}\Big)\,dv_{g}+\sum \limits_{i,j=1}^{d}a_{i}a_{j}\int_{M}V_{\xi_{i},\mu}\,\Big(P_{g}V_{\xi_{j},\mu}-V_{\xi_{j},\mu}^{\,2^{*}_{k}-1}\Big)\,dv_{g}\notag\\
&=\(1+\smallo_{\Xi_{d},\mu}(1)+\bigO_{\Xi_{d}}\(\delta^{n-2k}\)\)\int_{M}\Big(\sum \limits_{i=1}^{d}a_{i}V_{\xi_{i},\mu}\Big)\,\Big(\sum \limits_{j=1}^{d}a_{j}V_{\xi_{j},\mu}^{\,2^{*}_{k}-1}\Big)\,dv_{g}\notag\\
&=\(1+\smallo_{\Xi_{d},\mu}(1)+\bigO_{\Xi_{d}}\(\delta^{n-2k}\)\)\int_{M}\Big(\sum \limits_{i=1}^{d}a_{i}V_{\xi_{i},\mu}\Big)^{2}\,\Big[{\Big(\sum \limits_{j=1}^{d}a_{j}V_{\xi_{j},\mu}^{\,2^{*}_{k}-1}\Big)}\Big(\sum \limits_{i=1}^{d}a_{i}V_{\xi_{i},\mu}\Big)^{-1}\,\Big]\,dv_{g}\notag\\
&\le\(1+\smallo_{\Xi_{d},\mu}(1)+\bigO_{\Xi_{d}}\(\delta^{n-2k}\)\)\mathcal{D}\Big(\sum \limits_{i=1}^{d}a_{i}V_{\xi_{i},\mu}\Big)\[\bigints_{M}\(\frac{\sum \limits_{j=1}^{d}a_{j}V_{\xi_{j},\mu}^{\,2^{*}_{k}-1}}{\sum \limits_{i=1}^{d}a_{i}V_{\xi_{i},\mu}}\)^{n/2k}dv_{g}\]^{2k/n}.
\end{align}
Thus,
\begin{align}\label{num/dum1}
\mathcal{J}_{g,k}\Big(\sum \limits_{i=1}^{d}a_{i}V_{\xi_{i},\mu}\Big)\le\(1+\smallo_{\Xi_{d},\mu}(1)+\bigO_{\Xi_{d}}\(\delta^{n-2k}\)\)\[\bigints_{M}\(\frac{\sum \limits_{j=1}^{d}a_{j}V_{\xi_{j},\mu}^{\,2^{*}_{k}-1}}{\sum \limits_{i=1}^{d}a_{i}V_{\xi_{i},\mu}}\)^{n/2k}dv_{g}\]^{2k/n}.
\end{align}
We next estimate the term
\begin{align*}
\bigints_{M}\(\frac{\sum \limits_{j=1}^{d}a_{j}V_{\xi_{j},\mu}^{\,2^{*}_{k}-1}}{\sum \limits_{i=1}^{d}a_{i}V_{\xi_{i},\mu}}\)^{n/2k}dv_{g}=\bigints_{M}\(\sum \limits_{j=1}^{d}\frac{a_{j}V_{\xi_{j},\mu}}{\sum \limits_{i=1}^{d}a_{i}V_{\xi_{i},\mu}}V_{\xi_{j},\mu}^{\,2^{*}_{k}-2}\)^{n/2k}\,dv_{g}.
\end{align*}
By the convexity of $x\mapsto x^{n/2k}$ and since $\frac{2^{*}_{k}-2}{2^{*}_{k}}=\frac{2k}{n}$, we obtain
\begin{align}\label{est:mix:1}
&\bigints_{M}\(\frac{\sum \limits_{j=1}^{d}a_{j}V_{\xi_{j},\mu}^{\,2^{*}_{k}-1}}{\sum \limits_{i=1}^{d}a_{i}V_{\xi_{i},\mu}}\)^{n/2k}\,dv_{g}\le\sum \limits_{j=1}^{d}\bigintss_{M}\frac{a_{j}V_{\xi_{j},\mu}}{\sum \limits_{i=1}^{d}a_{i}V_{\xi_{i},\mu}}V_{\xi_{j},\mu}^{\,2^{*}_{k}}\,dv_{g}\notag\\
&\le\sum \limits_{j=1}^{d}\bigintss_{M}\frac{a_{j}V_{\xi_{j},\mu}}{a_{j}V_{\xi_{j},\mu}+\sum \limits_{i\neq j}^{d}a_{i}V_{\xi_{i},\mu}}V_{\xi_{j},\mu}^{\,2^{*}_{k}}\,dv_{g}\notag\\
&\leq\sum \limits_{j=1}^{d}\bigintsss_{M}V_{\xi_{j},\mu}^{\,2^{*}_{k}}\,dv_{g}-\bigintsss_{M}\frac{a_{\frak{i}}V_{\xi_{\frak{i}},\mu}}{a_{\frak{j}}V_{\xi_{\frak{j}},\mu}+a_{\frak{i}}V_{\xi_{\frak{i}},\mu}}V_{\xi_{\frak{j}},\mu}^{\,2^{*}_{k}}\,dv_{g}~\hbox{ for some fixed $\frak{i}\neq \frak{j}$}.
\end{align}
By symmetry we may assume that $a_{\frak{j}}\leq a_{\frak{i}}$. Fix $\theta>0$ and set 
$$\mathcal{R}_{\frak{i},\frak{j}}:=\left\{x\in M:a_{\frak{i}}V_{\xi_{\frak{j}},\mu}(x)\geq \theta \(a_{\frak{j}}V_{\xi_{\frak{j}},\mu}(x)+a_{\frak{i}}V_{\xi_{\frak{i}},\mu}(x)\)\right\}.$$
Then we have the following control from below for some fixed $\frak{i}\neq \frak{j}$.
\begin{align}\label{est:mix:2}
&\bigintsss_{M}\frac{a_{\frak{i}}V_{\xi_{\frak{i}},\mu}}{a_{\frak{j}}V_{\xi_{\frak{j}},\mu}+a_{\frak{i}}V_{\xi_{\frak{i}},\mu}}V_{\xi_{\frak{j}},\mu}^{\,2^{*}_{k}}\,dv_{g}\ge\bigintsss_{\mathcal{R}_{\frak{i},\frak{j}}}\frac{a_{\frak{i}}V_{\xi_{\frak{i}},\mu}}{a_{\frak{j}}V_{\xi_{\frak{j}},\mu}+a_{\frak{i}}V_{\xi_{\frak{i}},\mu}}V_{\xi_{\frak{j}},\mu}^{\,2^{*}_{k}}\,dv_{g}\notag\\
&\ge\theta\int_{\mathcal{R}_{\frak{i},\frak{j}}}V_{\xi_{\frak{i}},\mu}\,V_{\xi_{\frak{j}},\mu}^{\,2^{*}_{k}-1}\,dv_{g}\geq\theta\[\,\mathcal{Q}_{\frak{j}\frak{i}}-\int_{M\setminus\mathcal{R}_{\frak{i},\frak{j}}}V_{\xi_{\frak{i}},\mu}V_{\xi_{\frak{j}},\mu}^{\,2^{*}_{k}-1}\,dv_{g}\]\notag\\
&\ge\theta\[\,\mathcal{Q}_{\frak{j}\frak{i}}-\theta^{2^{*}_{k}-2}\bigintsss_{M\setminus\mathcal{R}_{\frak{i},\frak{j}}}\(\frac{a_{\frak{j}}}{a_{\frak{i}}}V_{\xi_{\frak{j}},\mu}+V_{\xi_{\frak{i}},\mu}\)^{\,2^{*}_{k}-2}V_{\xi_{\frak{i}},\mu}V_{\xi_{\frak{j}},\mu}\,dv_{g}\]\notag\\
&\ge\theta\[\mathcal{Q}_{\frak{j}\frak{i}}-\widetilde{C}\theta^{2^{*}_{k}-2}\int_{M}\[\(\frac{a_{\frak{j}}}{a_{\frak{i}}}V_{\xi_{\frak{j}},\mu}\)^{\,2^{*}_{k}-2}+V_{\xi_{\frak{i}},\mu}^{\,2^{*}_{k}-2}\] V_{\xi_{\frak{i}},\mu}V_{\xi_{\frak{j}},\mu}\,dv_{g}\]~\hbox{ for some $\widetilde{C}>0$}\notag\\
&\ge\theta\[\mathcal{Q}_{\frak{j}\frak{i}}-\widetilde{C}\theta^{2^{*}_{k}-2}\(\mathcal{Q}_{\frak{j}\frak{i}}+\mathcal{Q}_{\frak{i}\frak{j}}\)\]
\end{align}
Therefore for any $\frak{i}\neq \frak{j}$, from \eqref{est:mix:1} and \eqref{est:mix:2}, choosing $\theta$ small and using \eqref{int.7} from Proposition \ref{prop:interact}, we can write for some $\theta_{\sharp}>0$
\begin{align}
\bigints_{M}\(\frac{\sum \limits_{j=1}^{d}a_{j}V_{\xi_{j},\mu}^{\,2^{*}_{k}-1}}{\sum \limits_{i=1}^{d}a_{i}V_{\xi_{i},\mu}}\)^{n/2k}\,dv_{g}\le&\,\sum \limits_{j=1}^{d}\int_{M}V_{\xi_{j},\mu}^{\,2^{*}_{k}}\,dv_{g}-\theta\[\mathcal{Q}_{\frak{j}\frak{i}}-\widetilde{C}\theta^{2^{*}_{k}-2}\(\mathcal{Q}_{\frak{j}\frak{i}}+\mathcal{Q}_{\frak{i}\frak{j}}\)\]\notag\\
\le&\,\sum \limits_{j=1}^{d}\int_{M}V_{\xi_{j},\mu}^{\,2^{*}_{k}}\,dv_{g}-\theta_{\sharp}\varepsilon_{\frak{i}\frak{j}}.
\end{align}
Coming back to \eqref{num/dum1}, we have then obtained for some $\theta_{\sharp}>0$(making it smaller if necessary), that for any $\frak{i}\neq \frak{j}$
\begin{align}\label{num/dum2}
\mathcal{J}_{g,k}\Big(\sum \limits_{i=1}^{d}a_{i}V_{\xi_{i},\mu}\Big)&\le~d^{2k/n}Y_{k}(\S^{n})\notag\\
&\times\(1+\smallo_{\Xi_{d},\mu}(1)+\bigO_{\Xi_{d}}\(\delta^{n-2k}\)\)\(1+\bigO_{\Xi_{d}}(\mu^{n-2k})-\frac{\theta_{\sharp}}{d \|B_{0}\|_{2^{*}_{k}}^{2^{*}_{k}}}\varepsilon_{\frak{i}\frak{j}}\)^{2k/n}\notag\\
\end{align}
Taking  $\mu$ and $\delta$ sufficiently small  we obtain \eqref{energy_sum0}.
$$\mathcal{J}_{g,k}\(\sum \limits_{i=1}^{d}a_{i}V_{\xi_{i},\mu_{i}}\)\le~\(d+\frac{1}{2}\)^{2k/n}Y_{k}(\S^{n}).$$
Furthermore, in case $\displaystyle \sum\limits_{1\le i\neq j\le d}\varepsilon_{ij}>\varepsilon$, \eqref{num/dum2} implies for some small constant $\theta_{\sharp}>0$(with a slight abuse of notation), that 
\begin{align}\label{num/dum3}
\mathcal{J}_{g,k}\(\sum \limits_{i=1}^{d}a_{i}V_{\xi_{i},\mu_{i}}\)\leq&\, d^{2k/n}Y_{k}(\S^{n})\(1+\smallo_{\Xi_{d},\mu}(1)+\bigO_{\Xi_{d}}\(\delta^{n-2k}\)-\frac{\theta_{\sharp}}{d^{2}(d-1)}\varepsilon\).
\end{align}
Taking  $\mu$ and $\delta$ sufficiently small (depending on $d$ and $\varepsilon$) gives us \eqref{energy_sum1}. \qed
\medskip

\noindent
{\bf{Proof of \eqref{energy_sum2}}:} We now consider the case when the sum of interactions $\sum\limits_{1\le i\neq j\le d}\varepsilon_{ij}$ is small.  We have
\begin{align}\label{num2}
&\mathcal{N}\Big(\sum \limits_{i=1}^{d}a_{i}V_{\xi_{i},\mu_{i}}\Big)=\sum \limits_{i}^{d}a_{i}^{2}\int_{M}V_{\xi_{i},\mu}\,P_{g}V_{\xi_{i},\mu}\,dv_{g}+\sum \limits_{i\neq j}^{d}a_{i}a_{j}\mathfrak{L}_{ij}\notag\\
&~=\sum \limits_{i}^{d}a_{i}^{2}\mathcal{J}_{g,k}(V_{\xi_{i},\mu})\mathcal{D}(V_{\xi_{i},\mu})+\|B_{0}\|_{2^{*}_{k}-1}^{2^{*}_{k}-1}\sum\limits_{1\le i\neq j\le d}a_{i}a_{j}\,\varepsilon_{ij}\(1+\smallo_{\Xi_{d};\varepsilon_{ij}}(1)+\bigO_{\Xi_{d}}(\delta)\).
\end{align}
\medskip

\noindent
The following inequality gives us a lower bound on the denominator involving the interactions in a straightforward manner.  Moreover, the estimates this yields are uniform in $d$, the number of bubbles. 
\smallskip

Fix $p>2$. There exists $\kappa_{p}>1$ such that any positive integer $d$ and $X_{1},X_{2},\ldots,X_{d}\geq0$, we have
\begin{align}
\Big(\sum\limits_{i=1}^{d}X_{i}\Big)^{p}\ge \sum\limits_{i=1}^{d}X_{i}^{\,p}+\frac{p}{2}\kappa_{p} \sum\limits_{1\le j\ne i \le d}X_{i}^{p-1}X_{j}.
\end{align}
See, for instance, Lemma A2. in \cite{BahriBrezis}. 
\smallskip

\noindent
Then, along with  \eqref{int.10}, this inequality implies that 
\begin{align}
&\Bigg[\mathcal{D}\Big(\sum \limits_{i=1}^{d}a_{i}V_{\xi_{i},\mu}\Big)\Bigg]^{2^{*}_{k}/2}\,=\int_{M}\Big(\sum \limits_{i=1}^{d}a_{i}V_{\xi_{i},\mu}\Big)^{2^{*}_{k}}\,dv_{g}\notag\\
&\ge\sum \limits_{i=1}^{d}a_{i}^{\,2^{*}_{k}}\int_{M}V_{\xi_{i},\mu}^{\,2^{*}_{k}}\,dv_{g}+\frac{2^{*}_{k}}{2}\kappa_{2^{*}_{k}}\sum \limits_{1\leq i\neq j\leq d} a_{i}^{2^{*}_{k}-1}\,a_{j}\int_{M}V_{\xi_{j},\mu}\,V_{\xi_{i},\mu}^{\,2^{*}_{k}-1}\,dv_{g}\notag\\
&\ge\sum \limits_{i=1}^{d}a_{i}^{2^{*}_{k}}\int_{M}V_{\xi_{i},\mu}^{\,2^{*}_{k}}\,dv_{g}+\frac{2^{*}_{k}}{2}\kappa_{2^{*}_{k}}\|B_{0}\|_{2^{*}_{k}-1}^{2^{*}_{k}-1}\,\sum \limits_{1\leq i\neq j\leq d}a_{i}^{2^{*}_{k}-1}\,a_{j}\,\varepsilon_{ij}\(1+\smallo_{\Xi_{d};\varepsilon_{ij}}(1)+\bigO_{\Xi_{d}}(\delta)\).
\end{align}
Combining the estimates for the numerator and the denominator, we obtain the following bound on the energy for all $d\ge2$ for $\sum \limits_{1\le i\neq j\le d}\varepsilon_{ij}$ sufficiently small. 
\begin{align}\label{num/dum4}
&\mathcal{J}_{g,k}\Big(\sum \limits_{i=1}^{d}a_{i}V_{\xi_{i},\mu}\Big)\leq\[\sum\limits_{i}^{d}a_{i}^{2}\mathcal{J}_{g,k}(V_{\xi_{i},\mu})\mathcal{D}(V_{\xi_{i},\mu})+\sum\limits_{1\le i\ne j \le d}a_{i}a_{j}\varepsilon_{ij}\|B_{0}\|_{2^{*}_{k}-1}^{2^{*}_{k}-1}\right.\notag\\
&~\times\(1+\smallo_{\Xi_{d};\varepsilon_{ij}}(1)+\bigO_{\Xi_{d}}(\delta)\)\Bigg]\[\Big(\sum \limits_{i=1}^{d}a_{i}^{2^{*}_{k}}\mathcal{D}\(V_{\xi_{i},\mu}\)^{\frac{2^{*}_{k}}{2}}\Big)^{-2/2^{*}_{k}}\right.\notag\\
&~\left.-\kappa_{2^{*}_{k}}\,\|B_{0}\|_{2^{*}_{k}-1}^{2^{*}_{k}-1}\frac{\sum \limits_{1\leq i\neq j\leq d}a_{j}\,a_{i}^{2^{*}_{k}-1}\varepsilon_{ij}\(1+\smallo_{\Xi_{d};\varepsilon_{ij}}(1)+\bigO_{\Xi_{d}}(\delta)\)}{\(\sum\limits_{i=1}^{d}a_{i}^{2^{*}_{k}}\displaystyle\int_{M}V_{\xi_{i},\mu}^{\,2^{*}_{k}}\,dv_{g}\)^{\frac{2^{*}_{k}+2}{2^{*}_{k}}}}\]\notag\\
\leq&\frac{\displaystyle\sum\limits_{i}^{d}a_{i}^{2}\mathcal{J}_{g,k}(V_{\xi_{i},\mu})\mathcal{D}(V_{\xi_{i},\mu})}{\(\sum \limits_{i=1}^{d}a_{i}^{2^{*}_{k}}\mathcal{D}\(V_{\xi_{i},\mu}\)^{\frac{2^{*}_{k}}{2}}\)^{2/2^{*}_{k}}}+\frac{\displaystyle\sum \limits_{1\le i\neq j\le d}a_{i}a_{j}\varepsilon_{ij}\|B_{0}\|_{2^{*}_{k}-1}^{2^{*}_{k}-1}\(1+\smallo_{\Xi_{d};\varepsilon_{ij}}(1)+\bigO_{\Xi_{d}}(\delta)\)}{\(\sum \limits_{i=1}^{d}a_{i}^{2^{*}_{k}}\displaystyle\int_{M}V_{\xi_{i},\mu}^{\,2^{*}_{k}}\,dv_{g}\)^{2/2^{*}_{k}}}\notag\\
&-\kappa_{2^{*}_{k}}\|B_{0}\|_{2^{*}_{k}-1}^{2^{*}_{k}-1}\|B_{0}\|_{2^{*}_{k}}^{2^{*}_{k}}\,\Big(\sum\limits_{i=1}^{d}a_{i}^{2}\Big)\frac{\sum \limits_{1\leq i\neq j\leq d}a_{j}\,a_{i}^{2^{*}_{k}-1}\varepsilon_{ij}\(1+\smallo_{\Xi_{d};\mu}(1)+\smallo_{\Xi_{d};\varepsilon_{ij}}(1)+\bigO_{\Xi_{d}}(\delta)\)}{\(\sum\limits_{i=1}^{d}a_{i}^{2^{*}_{k}}\displaystyle\int_{M}V_{\xi_{i},\mu}^{\,2^{*}_{k}}\,dv_{g}\)^{\frac{2^{*}_{k}+2}{2^{*}_{k}}}}.
\end{align}
In the last line we have used the estimate for $\mathcal{N}(V_{\xi_{i},\mu})$.
Rearranging the terms and also using the estimate for $\mathcal{D}(V_{\xi_{i},\mu})$, we obtain that
\begin{align}\label{num/dum5}
&\mathcal{J}_{g,k}\Big(\sum \limits_{i=1}^{d}a_{i}V_{\xi_{i},\mu}\Big)\leq \max\limits_{1\le i\le d}J_{g}(V_{\xi_{i},\mu})\frac{\sum\limits_{i}^{d}a_{i}^{2}\,\mathcal{D}(V_{\xi_{i},\mu})}{\(\sum \limits_{i=1}^{d}a_{i}^{2^{*}_{k}}\mathcal{D}\(V_{\xi_{i},\mu}\)^{\frac{2^{*}_{k}}{2}}\)^{2/2^{*}_{k}}}+\|B_{0}\|_{2^{*}_{k}-1}^{2^{*}_{k}-1}\notag\\
&~\times\sum\limits_{1\le i\ne j \le d}\(1-\kappa_{2^{*}_{k}}\frac{a_{i}^{2^{*}_{k}-2}\Big(\sum\limits_{i=1}^{d}a_{i}^{2}\Big)}{\sum\limits_{i=1}^{d}a_{i}^{2^{*}_{k}}}\)a_{i}a_{j}\varepsilon_{ij}\frac{\(1+\smallo_{\Xi_{d};\mu}(1)+\smallo_{\Xi_{d};\varepsilon_{ij}}(1)+\bigO_{\Xi_{d}}(\delta)\)}{\(\sum \limits_{i=1}^{d}a_{i}^{2^{*}_{k}}\displaystyle\int_{M}V_{\xi_{i},\mu}^{\,2^{*}_{k}}\,dv_{g}\)^{2/2^{*}_{k}}}.
\end{align}
Here $\kappa_{2^{*}_{k}}>1$.  Using the estimates for $\mathcal{J}_{g,k}(V_{\xi_{i},\mu})$ and $\mathcal{D}(V_{\xi_{i},\mu})$, we can extract the following bound on the energy.
\begin{align}\label{num/dum6}
\mathcal{J}_{g,k}\Big(\sum \limits_{i=1}^{d}a_{i}V_{\xi_{i},\mu}\Big)\leq&~Y_{k}(\S^{n})\(1+\bigO_{\Xi_{d}}(\mu^{n-2k})\)\frac{\sum\limits_{i}^{d}a_{i}^{2}}{\(\sum \limits_{i=1}^{d}a_{i}^{2^{*}_{k}}\)^{2/2^{*}_{k}}}+\frac{\|B_{0}\|_{2^{*}_{k}-1}^{2^{*}_{k}-1}}{\|B_{0}\|_{2^{*}_{k}}^{2}}\notag\\
&\times\sum\limits_{1\le i\neq j\le d}\frac{a_{i}a_{j}}{\(\sum \limits_{i=1}^{d}a_{i}^{2^{*}_{k}}\)^{2/2^{*}_{k}}}\(1+\smallo_{\Xi_{d};\xi,\mu}(1)+\smallo_{\Xi_{d};\varepsilon_{ij}}(1)+\bigO_{\Xi_{d}}(\delta)\)\varepsilon_{ij}\notag\\
\leq&~\frac{\sum\limits_{i}^{d}a_{i}^{2}}{\(\sum \limits_{i=1}^{d}a_{i}^{2^{*}_{k}}\)^{2/2^{*}_{k}}}Y_{k}(\S^{n})\Bigg[1+\bigO_{\Xi_{d}}(\mu^{n-2k})+\frac{\|B_{0}\|_{2^{*}_{k}-1}^{2^{*}_{k}-1}}{\|B_{0}\|_{2^{*}_{k}}^{2^{*}_{k}}}\notag\\
&\times\sum\limits_{1\le i\neq j\le d}\frac{a_{i}a_{j}}{\sum\limits_{i}^{d}a_{i}^{2}} \varepsilon_{ij}\(1+\smallo_{\Xi_{d},\mu}(1)+\smallo_{\Xi_{d}\varepsilon_{ij}}(1)+\bigO_{\Xi_{d}}(\delta)\)\Bigg].\notag\\
\end{align}
If $\dfrac{a_{i_{0}}}{a_{j_{0}}}>\tau>1$ for some $i_{0}\neq j_{0}$, then $\dfrac{\sum\limits_{i=1}^{d}a_{i}^{2}}{\Big(\sum \limits_{i=1}^{d}a_{i}^{2^{*}_{k}}\Big)^{2/2^{*}_{k}}}\leq d^{2k/n}\(1-\dfrac{\theta_{\tau}}{d}\)$ for some $\theta_{\tau}>0$  (for instance, see Lemma A4. in \cite{BahriBrezis}). This inequality then implies the following bound on the energy.
\begin{align}\label{num/dum7}
\mathcal{J}_{g,k}\Big(\sum \limits_{i=1}^{d}a_{i}V_{\xi_{i},\mu}\Big)\leq&~d^{2k/n}\,Y_{k}(\S^{n})\Bigg[1+\bigO_{\Xi_{d}}(\mu^{n-2k})-\dfrac{\theta_{\tau}}{d}\notag\\
&+\frac{\|B_{0}\|_{2^{*}_{k}-1}^{2^{*}_{k}-1}}{\|B_{0}\|_{2^{*}_{k}}^{2^{*}_{k}}}\sum\limits_{1\le i\neq j\le d}\varepsilon_{ij}\(1+\smallo_{\Xi_{d};\mu}(1)+\smallo_{\Xi_{d};\varepsilon_{ij}}(1)+\bigO_{\Xi_{d}}(\delta)\)\Bigg].
\end{align}
We then obtain \eqref{energy_sum2} by taking  $\mu$ and $\delta$ sufficiently small (depending on $d$, $\varepsilon$ and $\tau$) along with \eqref{energy_sum1}.  \qed
\medskip

\noindent
{\bf{Proof of \eqref{energy_sum3_main}}:} We again assume that $d\ge2$ and $\sum \limits_{i\neq j}^{d}\varepsilon_{ij}$ sufficiently small. Moreover, we suppose that $\left|1-\dfrac{a_{i}}{a_{j}}\right|\le\widehat{\tau}$ for all $i\neq j$ with $\widehat{\tau}$ small. From \eqref{num/dum5} we then deduce the following. 
\begin{align}\label{num/dum8}
&\mathcal{J}_{g}\Big(\sum \limits_{i=1}^{d}a_{i}V_{\xi_{i},\mu}\Big)\leq\,\frac{\sum\limits_{i}^{d}a_{i}^{2}}{\(\sum \limits_{i=1}^{d}a_{i}^{2^{*}_{k}}\)^{2/2^{*}_{k}}}Y_{k}(\S^{n})\left[1+\bigO_{\Xi_{d}}(\mu^{n-2k})+\frac{\|B_{0}\|_{2^{*}_{k}-1}^{2^{*}_{k}-1}}{\|B_{0}\|_{2^{*}_{k}}^{2^{*}_{k}}}\right.\notag\\
&\left.\times\sum\limits_{1\le i\ne j \le d}\(1-\kappa_{2^{*}_{k}}\frac{a_{i}^{2^{*}_{k}-2}\Big(\sum\limits_{i=1}^{d}a_{i}^{2}\Big)}{\sum\limits_{i=1}^{d}a_{i}^{2^{*}_{k}}}\)\frac{a_{i}a_{j}}{\sum\limits_{i}^{d}a_{i}^{2}} \varepsilon_{ij}\(1+\smallo_{\Xi_{d};\mu}(1)+\smallo_{\Xi_{d};\varepsilon_{ij}}(1)+\bigO_{\Xi_{d}}(\delta)\)\right]\notag\\
&\leq\,d^{2k/n}\,Y_{k}(\S^{n})\left[1+\bigO_{\Xi_{d}}(\mu^{n-2k})+\(1-\kappa_{2^{*}_{k}}
\(\frac{1-\widehat{\tau}}{1+\widehat{\tau}}\)^{2^{*}_{k}}\)\frac{\|B_{0}\|_{2^{*}_{k}-1}^{2^{*}_{k}-1}}{\|B_{0}\|_{2^{*}_{k}}^{2^{*}_{k}}}\right.\notag\\
&~\left.\times\sum\limits_{1\le i\ne j \le d}\frac{a_{i}a_{j}}{\sum\limits_{i}^{d}a_{i}^{2}} \varepsilon_{ij}\(1+\smallo_{\Xi_{d};\mu}(1)+\smallo_{\Xi_{d};\varepsilon_{ij}}(1)+\bigO_{\Xi_{d}}(\delta)\)\right]\notag\\
&\leq\,d^{2k/n}\,Y_{k}(\S^{n})\left[1+\bigO(\mu^{n-2k})+\(1-\kappa_{2^{*}_{k}}
\Bigg(\frac{1-\widehat{\tau}}{1+\widehat{\tau}}\)^{2^{*}_{k}}\Bigg)\(\frac{1-\widehat{\tau}}{1+\widehat{\tau}}\)^{2}\right.\notag\\
&~\left.\times\frac{\|B_{0}\|_{2^{*}_{k}-1}^{2^{*}_{k}-1}}{d\,\|B_{0}\|_{2^{*}_{k}}^{2^{*}_{k}}}\sum\limits_{1\le i\ne j \le d}\varepsilon_{ij}\(1+\smallo_{\Xi_{d};\mu}(1)+\smallo_{\Xi_{d};\varepsilon_{ij}}(1)+\bigO_{\Xi_{d}}(\delta)\)\right].
\end{align}
Here we have used that $\dfrac{\sum\limits_{i}^{d}a_{i}^{2}}{\(\sum \limits_{i=1}^{d}a_{i}^{2^{*}_{k}}\)^{2/2^{*}_{k}}}\leq d^{2k/n}$. So, taking $\widehat{\tau}$ sufficiently close to $0$, we obtain that for all $d\ge1$ the following bound on the energy holds.  
\begin{align}\label{num/dum9}
\mathcal{J}_{g}\Big(\sum \limits_{i=1}^{d}a_{i}V_{\xi_{i},\mu}\Big)&\leq\,d^{2k/n}\,Y_{k}(\S^{n})\left[1+\bigO_{\Xi_{d}}(\mu^{n-2k})-\kappa_{\#}\frac{\|B_{0}\|_{2^{*}_{k}-1}^{2^{*}_{k}-1}}{d\,\|B_{0}\|_{2^{*}_{k}}^{2^{*}_{k}}}\right.\notag\\
&~\left.\times\sum\limits_{1\le i\ne j \le d}\varepsilon_{ij}\(1+\smallo_{\Xi_{d};\mu}(1)+\smallo_{\Xi_{d};\varepsilon_{ij}}(1)+\bigO_{\Xi_{d}}(\delta)\)\right],
\end{align}
for some constant $\kappa_{\#}>0$. Taking $\mu,\delta>0$ sufficiently small we have thus obtained that 
\begin{align}\label{num/dum10}
\mathcal{J}_{g}\Big(\sum \limits_{i=1}^{d}a_{i}V_{\xi_{i},\mu}\Big)&\leq\,d^{2k/n}\,Y_{k}(\S^{n})\left[1-\mathfrak{m}_{g}\mu^{n-2k}-\kappa_{\#}\frac{\|B_{0}\|_{2^{*}_{k}-1}^{2^{*}_{k}-1}}{2d\,\|B_{0}\|_{2^{*}_{k}}^{2^{*}_{k}}}\sum\limits_{1\le i\ne j \le d}\varepsilon_{ij}\right],\notag\\
\end{align}
for some constant $\mathfrak{m}_{g}\in\R$. Recall 
$$\varepsilon_{ij}=\(2+\mathfrak{c}_{n,k}^{-1}\frac{(b_{n,k}^{-1}G_{g}(\xi_{i},\xi_{j}))^{\frac{2}{2k-n}}}{\mu^{2}}\)^{\frac{2k-n}{2}}=\mathfrak{c}_{n,k}^{\frac{n-2k}{2}}\,b_{n,k}^{-1}G_{g}(\xi_{i},\xi_{j})\,\mu^{n-2k}\(1+\smallo_{\varepsilon_{ij}}(1)\).$$
Thus for $\sum \limits_{i\neq j}^{d}\varepsilon_{ij}$ sufficiently small  one has
\begin{align}
\sum \limits_{i\neq j}^{d}\varepsilon_{ij}\ge \frac{1}{2}d(d-1)\mathfrak{c}_{n,k}^{\frac{n-2k}{2}}\,b_{n,k}^{-1}\min\limits_{x\neq y}G_{g}(x,y)\,\mu^{n-2k}.
\end{align}
This gives us the claimed energy bound  \eqref{energy_sum3_main}  in terms of the Green's function for all $d\ge2$.
\begin{align*}
&\mathcal{J}_{g,k}\bigg(\sum \limits_{i=1}^{d}V_{\xi_{i},\mu}\bigg)\leq d^{2k/n}\,Y_{k}(\S^{n})\[1-\(\mathfrak{m}_{g}+(d-1)\,\mathfrak{b}_{n,k}\,\min\limits_{x\neq y}G_{g}(x,y)\)\mu^{n-2k}\],
\end{align*}
with $\mathfrak{b}_{n,k}:=\dfrac{\kappa_{\#}}{4}\dfrac{\|B_{0}\|_{2^{*}_{k}-1}^{2^{*}_{k}-1}}{\|B_{0}\|_{2^{*}_{k}}^{2^{*}_{k}}}\mathfrak{c}_{n,k}^{\frac{n-2k}{2}}\,b_{n,k}^{-1}>0$ and some  constant $\mathfrak{m}_{g}\in\R$. \qed

\medskip
This completes the proof of the Proposition \ref{main:energy esti}.
\end{proof}
\medskip

\noindent
An immediate consequence of Proposition \ref{main:energy esti} and the positivity of the Green's function \eqref{positivity}  is the following bound on the energy of the sum of bubbles, which signifies that \emph{interactions overwhelm the contribution of mass} at very high energy levels in the energy expansion. 
\begin{corollary}\label{energy loss}
Let $k\ge1$ be an integer, and $\(M,g\)$ be a smooth, closed Riemannian manifold of dimension $2k+1\le n\le 2k+3$ or $(M,g) $ be locally conformally flat. Assume that the GJMS operator $P_{g}$ satisfies \eqref{positivity}, and let $d\in\N$. Let $\xi_{1},\cdots,\xi_{d}\in M$ and $\mu\in (0,\delta)$ be sufficiently small such that the error estimate \eqref{error1} holds. There exists $d_{\star}\in N$ large and  $\mu_{\star},\delta_{\star}>0$ small such that, for all $0<\mu\le\mu_{\star}$, $0<\delta\le\delta_{\star}$,
the energy of the sum of $d$ bubbles $\sum \limits_{i=1}^{d}a_{i}V_{\xi_{i}\mu}$, with weights $a_{1},\ldots,a_{d}\in (0,+\infty)$ satisfies:
\begin{align}\label{energy_sum_loss}
\mathcal{J}_{g,k}\Big(\sum \limits_{i=1}^{d}a_{i}V_{\xi_{i},\mu_{i}}\Big)\le&\,\(d+1/2\)^{2k/n}Y_{k}(\S^{n})\hbox{ for  }d\ge 1,\notag\\
\mathcal{J}_{g,k}\Big(\sum \limits_{i=1}^{d}a_{i}V_{\xi_{i},\mu_{i}}\Big)<&\,d^{\,2k/n}\,Y_{k}(\S^{n})~\hbox{ for  }d\ge d_{\star}.
\end{align}
\end{corollary}
\bigskip

\section{Existence via the Bahri-Coron technique}\label{sec:BahriCoron}

The solutions of eq. \eqref{eq:one} are obtained as critical points of the energy functional $\mathcal{J}_{g,k}$ (defined in \eqref{energy}) on $\mathcal{M}^{+}_{k}$ (defined in \eqref{constraint}).  To prove Theorem \ref{thm:main}, we proceed by contradiction and assume that eq. \eqref{eq:one} has no smooth positive solutions, or equivalently by Proposition \ref{crit_pts}, the functional $\mathcal{J}_{g,k}$ has no critical points in $\mathcal{M}_{k}^{+}$. 
A standard technique to identify a critical point is to use a (pseudo) gradient flow associated to \eqref{energy}. However, our problem is noncompact, and the flow develops \emph{bubbles}.  
\medskip

\noindent
The technique developed and pioneered by Bahri \cite{Bahri}, Bahri-Coron \cite{BahriCoron} exploits this non-compactness to map barycenters of $M$ to the energy sublevels of $\mathcal{J}_{g,k}$ in a topologically non-trivial way. The energy expansion in Section \ref{sec:Energy} will then be used to obtain a contradiction.  Accordingly, we borrow heavily from their work and refer the reader to \cite{Bahri}, \cite{BahriCoron}, \cite{BahriBrezis}, \cite{MayNdi1} for proofs of the arguments that are well established and have been widely used in this context. We also refer to texts Dold \cite{Dold} and Chang \cite{KC Chang} for a review of algebraic topology, Banach manifolds and deformation lemmas. 
\medskip

As first introduced in \cite{Bahri, BahriCoron} and more recently used in \cite{MayNdi1}, we define a neighbourhood of non-compact flow lines or {\emph{critical points at infinity}}  of the energy functional $\mathcal{J}_{g,k}$ around the sum of $d\in\N$ bubbles,  for $\varepsilon>0$ small, as follows.  
\begin{align}\label{v-ngbd}
\mathscr{V}(d,\varepsilon):=&\bigg\{u\in H^{k}(M),\,u\ge0 \text{ a.e.}:  \exists (\xi_{1},\ldots,\xi_{d})\in M^{d},\, (\mu_{1},\ldots, \mu_{d})\in (0,\varepsilon)^{d},\notag\\
&~\text{ with } \Big\|u-\sum \limits_{i=1}^{d}V_{\xi_{i},\mu_{i}}\Big\|_{H^{k}(M)}\le\varepsilon~\text{ and} \sum\limits_{1\le i\ne j\le d}\varepsilon_{ij}\le \varepsilon\,\bigg\}.
\end{align}
The interactions $\varepsilon_{ij}$ are defined by \eqref{int.quant} and the buubles $V_{\xi,\mu_{i}}$ are defined by \eqref{bubble3b}, with $\delta>0$ chosen sufficiently small so that the error estimate \eqref{error1} holds for all $1\le i\le d$.  We have the now-classical profile or bubble tree decomposition. 
\begin{proposition}[Struwe type decomposition]\label{bubble decomp} 
Let $k\ge1$ be an integer, and let $\(M,g\)$ be a smooth, closed Riemannian manifold of dimension $2k+1\le n\le 2k+3$ or $(M,g) $ be locally conformally flat. Assume that the GJMS operator $P_{g}$ satisfies \eqref{positivity}.

Let $(u_{\alpha})_{\alpha}$ be a sequence in $\mathcal{M}_{k}^{+}$ such that $\mathcal{J}_{g,k}(u_{\alpha})$ is bounded, $\nabla_{\mathcal{M}_{k}^{+}}\mathcal{J}_{g,k}(u_{\alpha})\to 0$ as $\alpha \to+\infty$, and suppose the functional $\mathcal{J}_{g,k}$ has no critical points in $\mathcal{M}_{k}^{+}$. Then there exists  $d\in \N$ and a sequence $(\varepsilon_{\alpha})_{\alpha}$ with  $\varepsilon_{\alpha}>0$ for all $\alpha$ and $\varepsilon_{\alpha}\to0$ as $\alpha \to+\infty$ such that $u_{\alpha}\in\mathscr{V}(d,\varepsilon_{\alpha})$ up to a subsequence. Note then $\mathcal{J}_{g,k}(u_{\alpha})\to d^{2k/n}Y_{k}(\S^{n})$ as $\alpha \to+\infty$.
\end{proposition}

For $k=1$ and domains in $\R^{n}$, this is the pioneering result of Struwe \cite{Struwe}. Also see  \cite{BahriCoron} and Brezis and Coron \cite{BrezisCoron} for $H$-systems. For a more recent exposition in a book form, see Hebey \cite{Hebey1} (Chapter 3).  For $k=2$ this was obtained by Hebey and Robert in \cite{HebeyRob}. References to other related works in this direction can be found, for instance, in \cite{Maz2}. 

For general $k\ge1$, the proof of the above proposition follows from Theorem 1.3 in \cite{Maz2}, with the modification that we can interchange $V_{\xi,\mu}$ and $U_{\xi,\mu}$ since $\|V_{\xi,\mu}-V_{\xi,\mu}\|_{H^{k}(M)}\to0$ as $\mu\to0$. The proof of the above proposition also implies the uniqueness of the decomposition, up to permutations, in the following sense.

Let $d\in \N$, let $(\varepsilon_{\alpha})_{\alpha}$ be a positive sequence with $\varepsilon_{\alpha}\to0$ as $\alpha \to+\infty$, and let $\sum\limits_{i=1}^{d}a_{i,\alpha}V_{\xi_{i,\alpha},\,\mu_{i,\alpha}},\,\sum\limits_{i=1}^{d}\tilde{a}_{i,\alpha}V_{\tilde{\xi}_{i,\alpha},\,\tilde{\mu}_{i,\alpha}}\in\mathscr{V}(d,\varepsilon_{\alpha})$ be two sequences such that 
\begin{align*}
\Big\|\sum\limits_{i=1}^{d}a_{i,\alpha}V_{\xi_{i,\alpha},\,\mu_{i,\alpha}}-\sum\limits_{i=1}^{d}\tilde{a}_{i,\alpha}V_{\tilde{\xi}_{i,\alpha},\,\tilde{\mu}_{i,\alpha}}\Big\|_{H^{k}(M)}=\smallo(1)~\hbox{ as }\alpha\to+\infty.
\end{align*}
Then, up to permutations, for all $1\le i\le d$
\begin{align*}
 |a_{i,\alpha}-\tilde{a}_{i,\alpha}|=\smallo(1),\, \frac{\tilde{\mu}_{i,\alpha}}{\mu_{i,\alpha}}=1+\smallo(1),\, \frac{d_{g}(\tilde{\xi}_{i,\alpha},\,\xi_{i,\alpha})^{2}}{\tilde{\mu}_{i,\alpha}\,\mu_{i,\alpha}}=\smallo(1).
\end{align*}
See Lemma A.1 in \cite{BahriCoron}, Lemma 11 in \cite{BahriBrezis}, or Lemma 3.12 in \cite{Hebey1}. With this uniqueness of the decomposition and using the orthogonality conditions implied by a minimzer, we can parametrize the set $\mathscr{V}(d,\varepsilon)$ 
as in Proposition 7 in \cite{BahriCoron}  (also see Section 6 in \cite{BahriBrezis}) and obtain the following lemma. We skip its proof for brevity since it proceeds similarly with minor modifications.   Here $\sigma_{d}$ will denote the permutation group with $d$ elements.
\begin{lemma}[Selection Map]\label{selection map}  
Let $k\ge1$ be an integer, and let $\(M,g\)$ be a smooth, closed Riemannian manifold of dimension $2k+1\le n\le 2k+3$ or $(M,g) $ be locally conformally flat. Assume that the GJMS operator $P_{g}$ satisfies \eqref{positivity} and and suppose the functional $\mathcal{J}_{g,k}$ has no critical points in $\mathcal{M}_{k}^{+}$. 

Then  for every $d\in \N$ there exists $\varepsilon_{*}=\varepsilon_{*}(d)>0$, $\widehat{\varepsilon}_{*}=\widehat{\varepsilon}_{*}(d)>0$ such that for all $u\in\mathscr{V}(d,\varepsilon)$ with $\varepsilon< \varepsilon_{*}$ the minimization problem
\begin{align*}
\min \limits_{\mathscr{R}(d,\widehat{\varepsilon}_{*})} \Big\|u-\sum\limits_{i=1}^{d}a_{i}V_{\xi_{i},\mu_{i}}\Big\|_{H^{k}(M)}
\end{align*}
has a solution which is unique up to permutations $\sigma_{d}$, where 
\begin{align*}
\mathscr{R}(d,\widehat{\varepsilon}_{*}):=&\bigg\{(\xi_{1},\ldots,\xi_{d})\in M^{d},\,(\mu_{1},\ldots, \mu_{d})\in (0,\widehat{\varepsilon}_{*})^{\,d} \, (a_{1},\ldots, a_{d})\in (0,+\infty)^{d}\\
&: \frac{a_{i}}{a_{j}}\le 2, \text{ and } \varepsilon_{ij}\le \widehat{\varepsilon}_{*}\,\text{ for } 1\le i\neq j\le d.\bigg\}. 
\end{align*}
Denoting the unique solution of the above minimization problem by $\mathfrak{s}_{d}(u)$, we define the {\bf{selection map}} $\mathfrak{s}: \mathscr{V}(d,\varepsilon)\to M^{d}/\sigma_{d}$ by $u\mapsto \big(\xi_{1}^{u},\ldots,\xi_{d}^{u}\big)$ with $\mathfrak{s}_{d}(u)=\sum\limits_{i=1}^{d}a_{i}^{u}V_{\xi_{i}^{u},\mu_{i}^{u}}$.
\end{lemma}
\medskip

We next define the energy sublevel sets for $d\in\N\cup\{0\}$
\begin{align}\label{sublevel} 
\mathscr{W}^{d}:=\Big\{u \in \mathcal{M}_{k}^{+} : \mathcal{J}_{g,k}(u)\leq (d+1)^{2k/n}Y_{k}(\S^{n})\Big\}. 
\end{align}
By adapting the classical deformation lemma to account for the bubbling phenomena, as in Proposition 6 in \cite{BahriCoron} and Section 9 in \cite{BahriBrezis}, one obtains the following deformation lemma. We remark that the positivity condition \eqref{positivity} ensures that $\big\{u\in H^{k}(M): u\ge0 \text{ a.e.}\big\}$ is invariant under the flow associated to $-\nabla \mathcal{J}_{g,k}$, for example, via the Brezis-Martin type condition: Brezis \cite{BrezisFlow} (for finite-dimensional spaces), Martin \cite{Martin}, Theorem 5.2 in Deimling \cite{Deimling} and Theorem 6.3 in Chang \cite{KC Chang}. 

\begin{proposition}[A Deformation lemma]\label{deform lemma}
Let $k\ge1$ be an integer, $\(M,g\)$ be a smooth, closed Riemannian manifold of dimension $2k+1\le n\le 2k+3$ or $(M,g) $ be locally conformally flat. Assume that the GJMS operator $P_{g}$ satisfies \eqref{positivity}. 

Suppose the functional $\mathcal{J}_{g,k}$ has no critical points in $\mathcal{M}_{k}^{+}$. Then for all $d\ge1$, up to taking $\varepsilon>0$ small, $(\mathscr{W}^{d},\mathscr{W}^{d-1})$ deformation retracts onto $(\mathscr{W}^{d-1}\cup \mathscr{O}_{d,\varepsilon}\,, \mathscr{W}^{d-1})$, where $\mathscr{V}(d,\widetilde{\varepsilon}\,)\subset\mathscr{O}_{d, \varepsilon}\subset\mathscr{V}(d,\varepsilon)$ with $\widetilde{\varepsilon}>0$ sufficiently small. 
\end{proposition}
\bigskip

For a topological space $X$, we let $\mathscr{H}_{*}(X)$ denote the singular homology and $\mathscr{H}^{*}(X)$ denote the cohomology of $X$ with $\Z_{2}$ coefficients. Similarly, for given topological spaces $X$ and $Y,$ $\mathscr{H}_{*}(X,Y)$ will denote the relative homology.  For a map $f:X\to Y$, where $X$ and $Y$ are topological spaces, we denote by $f_{*}$ the induced map in homology and by $f^{*}$  the induced map in cohomology. Note that $M$ is $\Z_{2}$ orientable, that is, it admits non-vanishing top-dimensional homology, and $\mathscr{H}_{n}(M)\cong \Z_{2}$ if $M$ is connected. In other words, $M$ has a nontrivial fundamental class or the orientation class $[M]_{*}\in \mathscr{H}_{n}(M)$.  Poincar\'e duality relates the homology and the cohomology of $M$ via the cap product with the fundamental class $\beta\in \mathscr{H}^{k}(M)\longmapsto\beta \frown [M]_{*} \in \mathscr{H}_{n-k}(M)$ and this gives  $\mathscr{H}^{\ell}(M) \cong \mathscr{H}_{n-\ell}(M)$. 
\medskip

\noindent
In what follows, we will let $d$ be a positive integer. We denote the standard $(d-1)$ dimensional simplex with $d$-vertices as 
\begin{align}\label{simplex}
\Delta_{d-1}:=\Big\{(a_{1}, \ldots,a_{d})\in \R^{d}: a_{i}\ge0\, \text{ for } 1\le i\le d\, \text{ with } \sum \limits_{i=1}^{d} a_{i}=1 \Big\}.
\end{align}
Denoting the Dirac delta at $\xi \in M$ by $\delta_{\xi}$, we define the set of formal barycenters of $M$ of order $d\in N$ as follows:
\begin{align}\label{barycenter}
\mathscr{B}_{d}(M):=\Big\{\sum \limits_{i=1}^{d} a_{i}\delta_{\xi_{i}}: \xi_{i}\in M,\, a_{i}\ge0\, \text{ for } 1\le i\le d\, \text{ with } \sum \limits_{i=1}^{d} a_{i}=1 \Big\},
\end{align}
and set $\mathscr{B}_{0}(M)=\emptyset$. $\mathscr{B}_{d}(M)$ is given the weak topology of measures, and can be identified with the quotient space  $M^{d} \underset{\sigma_{d}}\times \Delta_{d-1}$.  We refer to Kallel and Karoui \cite{Kallel et el} for a detailed study of the homotopy type of the space of barycenters.  
\medskip

\noindent
There is a natural way to transfer the cohomology orientation class $[M]^{*}\neq0\in\mathscr{H}^{n}(M)$ as a nontrivial element of $\mathscr{H}^{n}(M^{d}/\sigma_{d})$. We give a quick summary of the construction. Let $\sigma_{1}\times\sigma_{d-1}$ be the subgroup of $\sigma_{d}$ consisting of elements such that $1\mapsto1$.
Then we have the following equivalence 
\begin{align*}
M^{d}/\sigma_{d}\cong M^{d}/(\sigma_{1}\times\sigma_{d-1})\Big\slash(\sigma_{d}/\sigma_{1}\times\sigma_{d-1})
\end{align*}
and the associated projection $\mathfrak{q}: M^{d}/(\sigma_{1}\times\sigma_{d-1})\to M^{d}/\sigma_{d}$. 
\begin{eqnarray*}
\begin{tikzcd}[column sep=large, row sep=large]
M^{d}/(\sigma_{1}\times\sigma_{d-1}) \arrow[r, "\pi"] \arrow[d, "\mathfrak{q}"'] & M \\
M^{d}/\sigma_{d} &
\end{tikzcd}
\end{eqnarray*}
We then use the transfer maps (see \cite{BahriCoron, BahriBrezis, GamYac, MayNdi1}), which is the transformed function of $\mathfrak{q}$.
\begin{align}\label{transfer}
tr: \mathscr{H}^{*}(M^{d}/(\sigma_{1}\times\sigma_{d-1}))\to \mathscr{H}^{*}(M^{d}/\sigma_{d}).
\end{align}
Then $ \mho^{*}_{d}:=tr\,\circ \pi^{*}([M]^{*})$ is nontrivial element of $\mathscr{H}^{n}(M^{d}/\sigma_{d})$.
\medskip

\noindent
On the other hand, it is well known that for all $d\in\N$, there exists a nontrivial orientation classes 
\begin{align}\label{Oclass}
\Omega_{d}\in\mathscr{H}_{d\,n+d-1}\(\mathscr{B}_{d}(M), \mathscr{B}_{d-1}(M)\)\hbox{ for all }d\ge1.
\end{align}
The cap product and the boundary operator act as follows
\begin{align*}
&\mathscr{H}^{\ell}(M^{d}/\sigma_{d})\times\mathscr{H}_{i}\(\mathscr{B}_{d}(M), \mathscr{B}_{d-1}(M)\)\overset{\frown}\longrightarrow\mathscr{H}_{i-\ell}\(\mathscr{B}_{d}(M), \mathscr{B}_{d-1}(M)\)\\
&~\overset{\partial}\longrightarrow\mathscr{H}_{i-\ell-1}\(\mathscr{B}_{d}(M), \mathscr{B}_{d-1}(M)\).
\end{align*}
Furthermore, the classes  $\Omega_{d}$'s are related via the relation 
\begin{align}\label{transfer relation}
\Omega_{d-1}=\partial\(\mho^{*}_{d}\frown\Omega_{d}\).
\end{align}
See \cite{BahriCoron, BahriBrezis, GamYac, MayNdi1}.
\medskip

\noindent
We define $\mathscr{F}_{d}(\mu)=\mathscr{P}_{2^{*}_{k}}\circ\widetilde{\mathscr{F}}_{d}(\mu)$ where $\mathscr{P}_{2^{*}_{k}}:u \to u/\|u\|_{L^{2^{*}_{k}}}$, $u\not\equiv0$, and $\widetilde{\mathscr{F}}_{d}(\mu): \mathscr{B}_{d}(M)\to\{u\in H^{k}(M): u\geq 0 \text{ a.e.}\}$ is given by
\begin{align*}
\widetilde{\mathscr{F}}_{d}(\mu):\sum \limits_{i=1}^{d} a_{i}\delta_{\xi_{i}}\longmapsto \sum \limits_{i=1}^{d}a_{i}V_{\xi_{i},\mu}.
\end{align*}
Here $V_{\xi_{i},\mu}$ is defined in \eqref{bubble3b} and $\mu, \delta>0$ are chosen sufficiently small as in Corollary \ref{energy loss}. We have then $\mathscr{F}_{d}(\mu): \mathscr{B}_{d}(M)\to \mathscr{W}^{d}$ for all $d\ge1$.  For our topological arguments, we will view $\mathscr{F}_{d}(\mu)$ as a pair map, that is
$$\mathscr{F}_{d}(\mu): \(\mathscr{B}_{d}(M), \mathscr{B}_{d-1}(M)\) \longrightarrow  (\mathscr{W}^{d}, \mathscr{W}^{d-1}).$$
\medskip

We will show that if the functional $\mathcal{J}_{g,k}$ has no critical points in $\mathcal{M}_{k}^{+}$, then, for $\mu$ small enough, the map $\mathscr{F}_{1}(\mu)$ maps the background manifold $M$ to $(\mathscr{W}^{1}, \mathscr{W}^{0})$ in a non-trivial way. We then show that $\mathscr{F}_{d}(\mu)_{*}\not\equiv 0$, for all $ d \in \mathbb{N}$ by using relation \eqref{transfer relation} and induction. On the other hand, from the interaction estimates in Proposition \ref{main:energy esti} it follows that for $\mu$ small enough the map $\mathscr{F}_{d}(\mu)$ is homologically trivial, which is a contradiction. 

\begin{proposition}\label{topo prop}
Let $k\ge1$ be an integer, and let $\(M,g\)$ be a smooth, closed Riemannian manifold of dimension $2k+1\le n\le 2k+3$ or $(M,g) $ be locally conformally flat. Assume that the GJMS operator $P_{g}$ satisfies \eqref{positivity} and suppose that the functional $\mathcal{J}_{g,k}$ has no critical points in $\mathcal{M}_{k}^{+}$. Choosing $0<\mu\le\mu_{\star}$ given by Corollary \ref{energy loss}, the mapping
\begin{align*}
\mathscr{F}_{d}(\mu): \(\mathscr{B}_{d}(M), \mathscr{B}_{d-1}(M)\) \longrightarrow (\mathscr{W}^{d}, \mathscr{W}^{d-1})
\end{align*}
is well defined for all  $d\in\N$. Furthermore,  $\mathscr{F}_{d}(\mu)_{*}(\Omega_{d})\neq 0$ in $\mathscr{H}_{d\,n+d-1}(\mathscr{W}^{d},\mathscr{W}^{d-1})$  for all $d\geq 1$. 
Here $\Omega_{d}\in\mathscr{H}_{d\,n+d-1}\(\mathscr{B}_{d}(M), \mathscr{B}_{d-1}(M)\)$ given by \eqref{Oclass} is non-trivial. 
\end{proposition}
\begin{proof}
Definitions and Proposition \ref{main:energy esti} give us the following commutative diagram
\begin{eqnarray}
\begin{tikzcd}[column sep=large, row sep=large]
\mathscr{H}_{*}\(\mathscr{B}_{d}(M),\, \mathscr{B}_{d-1}(M)\) \arrow[r, "\mathscr{F}_{d}(\mu)_{*}"] \arrow[d, "\partial"'] &\mathscr{H}_{*}\(\mathscr{W}^{d},\, \mathscr{W}^{d-1}\) \arrow[d, "\partial"] \\
\mathscr{H}_{*\,-1}\(\mathscr{B}_{d-1}(M),\, \mathscr{B}_{d-2}(M)\) \arrow[r, "\mathscr{F}_{d-1}(\mu)_{*}"] & \mathscr{H}_{*\,-1}\(\mathscr{W}^{d-1},\, \mathscr{W}^{d-2}\).
\end{tikzcd}
\end{eqnarray}
Here $\partial$  denotes the connecting homomorphisms. 
\smallskip

\noindent
We have, as in \cite{BahriCoron, BahriBrezis, MayNdi1}, that $\mho^{*}_{d}=\mathscr{F}_{d}(\mu)^{*}( \mathfrak{s}_{d}^{*}(\mho^{*}_{d}))$. Recalling the relation $\Omega_{d-1}=\partial\(\mho^{*}_{d}\frown\Omega_{d}\)$, we obtain from the above commutative diagram. 
\begin{align}
&\mathscr{F}_{d-1}(\mu)_{*}(\Omega_{d-1})=\mathscr{F}_{d-1}(\mu)_{*}\(\partial\(\mho^{*}_{d}\frown\Omega_{d}\)\)\notag\\
&=\partial\Big(\mathscr{F}_{d}(\mu)_{*}\big(~\mathscr{F}_{d}(\mu)^{*}( \mathfrak{s}_{d}^{*}(\mho^{*}_{d}))\frown\Omega_{d}\big)\Big)=\partial\,\big(\mathfrak{s}_{d}^{*}(\mho^{*}_{d})\frown\mathscr{F}_{d}(\mu)_{*}(\Omega_{d})\big).
\end{align}
So $\mathscr{F}_{d-1}(\mu)_{*}(\Omega_{d-1})\neq0$ in $\mathscr{H}_{(d-1)(n+1)-1}(\mathscr{W}^{d-1},\mathscr{W}^{d-2})$ for $d\ge2$  implies that $\mathscr{F}_{d}(\mu)_{*}(\Omega_{d})\neq0$ in $\mathscr{H}_{d(n+1)-1}(\mathscr{W}^{d-1}\cup \mathscr{O}_{d,\varepsilon},\mathscr{W}^{d-1})\simeq\mathscr{H}_{d(n+1)-1}(\mathscr{W}^{d},\mathscr{W}^{d-1})$ for $\varepsilon$ small by the deformation lemma \ref{deform lemma}. Therefore $\mathscr{F}_{d-1}(\mu)_{*}(\Omega_{d-1})\neq0$ implies that $\mathscr{F}_{d}(\mu)_{*}(\Omega_{d})\neq0$ for $d\ge2$. 
\smallskip

\noindent
Next we show that the map $\mathscr{F}_{1}(\mu): \(\mathscr{B}_{1}(M), \mathscr{B}_{0}(M)\) \longrightarrow $ $(\mathscr{W}^{1}, \mathscr{W}^{0})$ satisfies $\mathscr{F}_{1}(\mu)_{*}(\Omega_{1})\neq 0$ in $\mathscr{H}_{n}(\mathscr{W}^{1}, \mathscr{W}^{0})$ assuming that $\mathcal{J}_{g,k}$ has no critical points in $\mathcal{M}_{k}^{+}$. This will complete the proof of Proposition \ref{topo prop} by induction.
\smallskip

\noindent
Note, $\mathscr{W}^{0}=\emptyset$ since $\inf\limits_{\mathcal{M}_{k}^{+}}\mathcal{J}_{g}(u)=Y_{k}(\S^{n})$ and the infimum is not achieved,  in our case. From Proposition \ref{deform lemma} we obtain that for some $\delta$ small, the sublevel $\mathscr{W}^{1}$ deformation retracts to $\mathscr{O}_{1,\delta}$, with $\mathscr{V}(1,\tilde{\delta}\,)\subset\mathscr{O}_{1, \delta}\subset\mathscr{V}(1,\delta)$ for some $\tilde{\delta}$ small. 
For $\mu$ and $\hat{\varepsilon}$ sufficiently small we have  $\mathscr{F}_{1}(\mu): M\to \mathscr{V}(1,\hat{\varepsilon})$, and is such that $\mathfrak{s}\circ\mathscr{F}_{1}(\mu)=\mathds{1}_{M}$. Here $\mathfrak{s}$ is the selection map from Proposition \ref{selection map}. So then 
\begin{eqnarray}
\begin{tikzcd}[column sep=large, row sep=large]
\arrow[d,"\frak{i}"']\mathscr{V}(1,\hat{\varepsilon})&\arrow[l,"\mathscr{F}_{1}(\mu)"']M\arrow[d,"\text{Id}_{M}"]\arrow[r,"\mathscr{F}_{1}(\mu)"]&\mathscr{W}^{1}\arrow[r,"\frak{r}"]&\mathscr{O}_{1, \delta}\arrow[ld,"\frak{i}"]\\
\mathscr{V}(1,\delta)\arrow[r,"\frak{s}_{1}"]&M&\arrow[l,"\frak{s}_{1}"']\mathscr{V}(1,\delta)
\end{tikzcd}
\end{eqnarray}
Here $\frak{i}$ is the inclusion map and $\frak{r}$ is the retraction given by Proposition \ref{deform lemma}. This tells us $\frak{s}_{1}\circ \frak{i}\circ \frak{r}\circ \mathscr{F}_{1}(\mu)=\text{Id}_{M}$. Passing to homologies we get that $(\frak{s}_{1})_{*}\circ \frak{i}_{*}\circ\frak{r}_{*}\circ\mathscr{F}_{1}(\mu)_{*}(\Omega_{1})\neq 0$ in $\mathscr{H}_{n}(\mathscr{W}^{1})=\mathscr{H}_{n}(\mathscr{W}^{1},\mathscr{W}^{0})$.
\noindent
Note that this proposition implies for all $d\in\N$, up to taking $\varepsilon$ and $\mu$  small, the map  $\mathscr{F}_{d}(\mu): \(\mathscr{B}_{d}(M), \mathscr{B}_{d-1}(M)\) \longrightarrow  (\mathscr{W}^{d}, \mathscr{W}^{d-1})$ is not homotopic to a constant map. 
\end{proof}
\medskip

\noindent
{\bf{Proof of Theorem \ref{thm:main}}}: We proceed by contradiction and, therefore, assume that eq. \eqref{eq:one} has no positive smooth solutions. So the functional $\mathcal{J}_{g,k}$ has no critical points in $\mathcal{M}_{k}^{+}$ by Proposition \ref{crit_pts}. 

Corollary \ref{energy loss}  gives that there exists $d_{\star}\in\N$ large such that $\mathscr{F}_{d}(\mu)\(\mathscr{B}_{d}(M)\)\subset\mathscr{W}^{d-1}$ for all $d\ge d_{\star}$, implying $\mathscr{F}_{d}(\mu)$ is homotopic to a constant map in $(\mathscr{W}^{d}, \mathscr{W}^{d-1})$. However, this contradicts the non-triviality in Proposition \ref{topo prop} for $d\ge d_{*}$, thereby completing the proof. \qed
\medskip

\appendix
\section{An expansion of the GJMS operator}\label{A2}

We provide the expansion of $P_{g}$ in conformal normal coordinates obtained in \cite{MazPrem} (Proposition B.2), following the earlier work \cite{MazVetois}. This is based on Juhl’s formulae for GJMS operators \cite{JuhlGJMS}. 

We use the notations from Section \ref{sec1} and recall that $g_{\xi} := \Lambda_\xi^{\frac{4}{n-2k}} g$ is the conformal metric satisfying \eqref{conf.1}, \eqref{conf.2} and  \eqref{conf.3}. For a smooth radial function  $v(x)=v(r)$ we have  as $r=|x|\to0$
\begin{align*}
&P_{(\exp^{g_{\xi}}_\xi)^{*}g_{\xi}}\,v (x)=\,\Delta_{0}^{k}\,v(x)+ \big(C_{1,2k-2}\Scal_{g_{\xi};ab}(\xi)+\psi^{(3)}_{2k-2}\big)\Delta_{0}^{k-1}v(x)-\bigg(C_{2,2k-2}\,\times \notag\\
&\,\sum\limits_{i,j}\Weyl_{g_{\xi},iabj}(\xi)\Weyl_{g_{\xi},icdj}(\xi)\frac{x^{a}x^{b}x^{c}x^{d}}{r^{2}}+\frac{\psi^{(5)}_{2k-2}}{r^{2}}\bigg)\partial_{r}^{2}\Delta_{0}^{k-2}v(x)+\bigg(C_{1,2k-3}\,\times  \notag\\
&\,\sum\limits_{i,j}\Weyl_{g_{\xi},iabj}(\xi)\Weyl_{g_{\xi},icdj}(\xi)\frac{x^{a}x^{b}x^{c}x^{d}}{r^{3}}-C_{2,2k-3}\Scal_{g_{\xi};ab}(\xi)\frac{x^{a}x^{b}}{r}+\frac{\psi^{(5)}_{2k-3}}{r^{3}}\Bigg)\times \notag\\
&\,\partial_{r}\Delta_{0}^{k-2}v(x)-\Bigg( C_{3,2k-3}\sum\limits_{i,j}\Weyl_{g_{\xi},iabj}(\xi)\Weyl_{g_{\xi},icdj}(\xi)\frac{x^{a}x^{b}x^{c}x^{d}}{r^{3}}+C_{4,2k-3}\,\times \notag\\
&\,\Scal_{g_{\xi};ab}(\xi)\frac{x^{a}x^{b}}{r}\Bigg)\partial^{3}_{r}\Delta_{0}^{k-3}v(x)-\Big(C_{1,2k-4}|\Weyl_{g_{\xi}}|^{2}+\psi_{2k-4}^{(1)}\Big)\Delta^{k-2}_{0}v(x)-\Big(C_{2,2k-4}\,\times  \notag\\
&\,|\Weyl_{g_{\xi}}|^{2}+C_{3,2k-4}\Delta_{g_{\xi}}\Ricci_{g_{\xi},ab}(\xi)\frac{x^{a}x^{b}}{r^{2}}+C_{4,2k-4}\Scal_{g_{\xi};ab}(\xi)\frac{x^{a}x^{b}}{r^{2}}-C_{5,2k-4}\,\times  \notag\\
&\,\sum\limits_{i,j}\Weyl_{g_{\xi},abj}(\xi)\Weyl_{g_{\xi},icdj}(\xi)\frac{x^{a}x^{b}x^{c}x^{d}}{r^{4}}+\frac{\psi_{2k-4}^{(3)}}{r^{2}}\Bigg)\partial^{2}_{r}\Delta_{0}^{k-3}v(x)\notag-\bigg(C_{6,2k-4}\Scal_{g_{\xi};ab}(\xi)  \notag\\
&\,\times\frac{x^{a}x^{b}}{r^{2}}+C_{7,2k-4}\sum\limits_{i,j}\Weyl_{g_{\xi},iabj}(\xi)\Weyl_{g_{\xi},icdj}(\xi)\frac{x^{a}x^{b}x^{c}x^{d}}{r^{4}} +\frac{\psi_{2k-4}^{(5)}}{r^{2}}\bigg)\partial^{4}_{r}\Delta_{0}^{k-4}v(x)\,+\notag\\
\end{align*}
\begin{align}\label{expansion.Pg1}
&\,\bigg(C_{1,2k-5}|\Weyl_{g_{\xi}}|^{2}+C_{2,2k-5}\Delta_{g_{\xi}}\Ricci_{g_{\xi},ab}(\xi)\frac{x^{a}x^{b}}{r^{2}}+C_{3,2k-5}\Scal_{g_{\xi};ab}(\xi)\frac{x^{a}x^{b}}{r^{2}}+\frac{\psi_{2k-5}^{(3)}}{r^{2}}\bigg)\notag\\
&\,\times\frac{\partial_{r}\Delta_{0}^{k-3}v(x)}{r}+\bigg(C_{2,2k-5}|\Weyl_{g_{\xi}}|^{2}-C_{3,2k-5}\Delta_{g}\Ricci_{g_{\xi,ab}}(\xi)\frac{x^{a}x^{b}}{r^{2}}+C_{4,2k-5}\Scal_{g_{\xi};ab}(\xi) \notag\\
&\,\times\frac{x^{a}x^{b}}{r^{2}}+C_{5,2k-5}\sum\limits_{i,j}\Weyl_{g_{\xi},iabj}(\xi)\Weyl_{g_{\xi},icdj}(\xi)\frac{x^{a}x^{b}x^{c}x^{d}}{r^{4}}+\frac{\psi_{2k-5}^{(5)}}{r^{2}}\bigg)\frac{\partial^{3}_{r}\Delta_{0}^{k-4}v(x)}{r}\,+\notag\\
&\,\bigg(C_{1,2k-6}|\Weyl_{g_{\xi}}|^{2}+C_{2,2k-6}\Delta_{g}\Ricci_{g_{\xi,ab}}(\xi)\frac{x^{a}x^{b}}{r^{2}}-C_{3,2k-6}\Scal_{g_{\xi};ab}(\xi)\frac{x^{a}x^{b}}{r^{2}}-C_{4,2k-6} \notag\\
&\,\times\sum\limits_{i,j}\Weyl_{g_{\xi},iabj}(\xi)\Weyl_{g_{\xi},icdj}(\xi)\frac{x^{a}x^{b}x^{c}x^{d}}{r^{4}}+\frac{\psi_{2k-6}^{(3)}}{r^{2}}\bigg)\frac{\partial^{2}_{r}\Delta_{0}^{k-4}v(x)}{r^{2}}+\bigg(C_{1,2k-7}\,\times \notag\\
&\,|\Weyl_{g_{\xi}}|^{2}-C_{2,2k-7}\Delta_{g}\Ricci_{g_{\xi,ab}}(\xi)\frac{x^{a}x^{b}}{r^{2}}+C_{3,2k-7}\Scal_{g_{\xi};ab}(\xi)\frac{x^{a}x^{b}}{r^{2}}\;+ \notag\\
&\,C_{4,2k-7}\sum\limits_{i,j}\Weyl_{g_{\xi},iabj}(\xi)\Weyl_{g_{\xi},icdj}(\xi)\frac{x^{a}x^{b}x^{c}x^{d}}{r^{4}}+\frac{\psi_{2k-7}^{(5)}}{r^{2}}\bigg)\frac{\partial_{r}\Delta_{0}^{k-4}v(x)}{r^{3}}\;+\notag\\
&\,\bigO \big( \sum_{j=0}^{3} r^{4-j}|\nabla^{2k-2-j} v(x)|\big) + \bigO \big( \sum_{j=0}^{2k-6}|\nabla^{j} v(x)|\big)\;+ \notag\\
&\,\sum_{p=0}^{2k-6} \sum_{q=0}^{\big[\frac{2k-5}{2}\big]} \frac{\psi_q^{(2k-5-2q)}(x)}{r^{2k-5-2q + p}} \partial_r^{2k-5-p} v(x)+ \bigO \big( r^N |\nabla^{2k-1} v(x)|\big).
\end{align}
The coefficients $C_{i,2k-j}$ are positive are rational functions of $n, k$ and $\psi_\ell^{(i)}$ are homogeneous polynomials of degree $i \ge 1$ in $\R^n$ with coefficients depending only on $g_\xi$. For further details and a proof, see Appendix B in \cite{MazPrem}.
\bigskip

\subsection*{Acknowledgements:} 
S. Mazumdar gratefully acknowledges the support from the MATRICS grant MTR/2022/000447 of the Science and Engineering Research Board (currently ANRF) of India. Part of this work was carried out during S.Mazumdar's visit to Howard University, and S. Mazumdar is grateful for the support and hospitality provided. S. Mazumdar would also like to thank Fr\'ed\'eric Robert and K. Sandeep for introducing him to the work of Bahri and the Bahri–Coron paper several years ago, and for many valuable discussions. 
\smallskip

\noindent
C.B. Ndiaye is partially supported by NSF grant DMS–2000164 and AMS 2025–2026 Claytor-Gilmer Fellowship.

\bigskip


\medskip

\begin{thebibliography}{99}

\bib{AndKoRatWei}{article}{
     title={Quantitative stability of the total $Q$-curvature near minimizing metrics}, 
     author={João Henrique Andrade},
     author={ König, Tobias},
     author={Ratzkin, Jesse},
     author={Wei, Juncheng},
     date={2024},
     note={arXiv:2407.06934},
}

\bib{Aubin}{article}{
   author={Aubin, Thierry},
   title={\'Equations diff\'erentielles non lin\'eaires et probl\`eme de
   Yamabe concernant la courbure scalaire},
   journal={J. Math. Pures Appl. (9)},
   volume={55},
   date={1976},
   number={3},
   pages={269--296},
   issn={0021-7824},
}

\bib{ALL}{article}{
   author={Avalos, Rodrigo},
   author={Laurain, Paul},
   author={Lira, Jorge H.},
   title={A positive energy theorem for fourth-order gravity},
   journal={Calc. Var. Partial Differential Equations},
   volume={61},
   date={2022},
   number={2},
   pages={Paper No. 48, 33},
   issn={0944-2669},
}

\bib{Bahri}{book}{
    author={Bahri, Abbas},
   title={Critical points at infinity in some variational problems},
   series={Pitman Research Notes in Mathematics Series},
   volume={182},
   publisher={Longman Scientific \& Technical, Harlow; copublished in the
   United States with John Wiley \& Sons, Inc., New York},
   date={1989},
   pages={vi+I15+307},
   isbn={0-582-02164-2},
}


\bib{BahriBrezis}{article}{
   author={Bahri, Abbas},
   author={Brezis, Ha\"im},
   title={Non-linear elliptic equations on Riemannian manifolds with the
   Sobolev critical exponent},
   conference={
      title={Topics in geometry},
   },
   book={
      series={Progr. Nonlinear Differential Equations Appl.},
      volume={20},
      publisher={Birkh\"auser Boston, Boston, MA},
   },
   isbn={0-8176-3828-8},
   date={1996},
   pages={1--100},
}

\bib{BahriCoron}{article}{
   author={Bahri, Abbas},
   author={Coron, Jean-Michel},
   title={On a nonlinear elliptic equation involving the critical Sobolev
   exponent: the effect of the topology of the domain},
   journal={Comm. Pure Appl. Math.},
   volume={41},
   date={1988},
   number={3},
   pages={253--294},
   issn={0010-3640},
}

\bib{BFR}{article}{
   author={Baird, Paul},
   author={Fardoun, Ali},
   author={Regbaoui, Rachid},
   title={Prescribed $Q$-curvature on manifolds of even dimension},
   journal={J. Geom. Phys.},
   volume={59},
   date={2009},
   number={2},
   pages={221--233},
   issn={0393-0440},
}


\bib{Branson1}{article}{
   author={Branson, Thomas P.},
   title={Differential operators canonically associated to a conformal
   structure},
   journal={Math. Scand.},
   volume={57},
   date={1985},
   number={2},
   pages={293--345},
   issn={0025-5521},
   review={\MR{0832360}},
}

\bib{Branson2}{article}{
   author={Branson, Thomas P.},
   title={Group representations arising from Lorentz conformal geometry},
   journal={J. Funct. Anal.},
   volume={74},
   date={1987},
   number={2},
   pages={199--291},
   issn={0022-1236},
}

\bib{Branson3}{article}{
   author={Branson, Thomas P.},
   title={Sharp inequalities, the functional determinant, and the
   complementary series},
   journal={Trans. Amer. Math. Soc.},
   volume={347},
   date={1995},
   number={10},
   pages={3671--3742},
   issn={0002-9947},
   review={\MR{1316845}},
}

\bib{BransonOrs}{article}{
   author={Branson, Thomas P.},
   author={\O rsted, Bent},
   title={Explicit functional determinants in four dimensions},
   journal={Proc. Amer. Math. Soc.},
   volume={113},
   date={1991},
   number={3},
   pages={669--682},
   issn={0002-9939},
}

\bib{BreQ}{article}{
   author={Brendle, Simon},
   title={Global existence and convergence for a higher order flow in
   conformal geometry},
   journal={Ann. of Math. (2)},
   volume={158},
   date={2003},
   number={1},
   pages={323--343},
   issn={0003-486X},
}

\bib{Bre1}{article}{
   author={Brendle, Simon},
   title={Convergence of the Yamabe flow for arbitrary initial energy},
   journal={J. Differential Geom.},
   volume={69},
   date={2005},
   number={2},
   pages={217--278},
   issn={0022-040X},
}

\bib{BrezisFlow}{article}{
   author={Brezis, Ha\"im},
   title={On a characterization of flow-invariant sets},
   journal={Comm. Pure Appl. Math.},
   volume={23},
   date={1970},
   pages={261--263},
   issn={0010-3640},
}

\bib{BrezisCoron}{article}{
   author={Brezis, Ha\"im},
   author={Coron, Jean-Michel},
   title={Convergence of solutions of $H$-systems or how to blow bubbles},
   journal={Arch. Rational Mech. Anal.},
   volume={89},
   date={1985},
   number={1},
   pages={21--56},
   issn={0003-9527},
}

\bib{Cao}{article}{
   author={Cao, Jian Guo},
   title={The existence of generalized isothermal coordinates for
   higher-dimensional Riemannian manifolds},
   journal={Trans. Amer. Math. Soc.},
   volume={324},
   date={1991},
   number={2},
   pages={901--920},
   issn={0002-9947},
}

\bib{CarPrem}{article}{
     author={Carletti, Lorenzo},
     author={Premoselli, Bruno},
     title={A priori bounds for energy-bounded solutions of critical polyharmonic equations on open sets}, 
     year={2026},
     note={In preparation},
}

\bib{CaseGov}{article}{
   author={Case, Jeffrey S.},
   author={Gover, A. Rod},
   title={The GJMS operators in geometry, analysis and physics},
   journal={J. Lond. Math. Soc. (2)},
   volume={113},
   date={2026},
   number={1},
   pages={Paper No. e70375, 21},
   issn={0024-6107},
}

\bib{CaseMal}{article}{
   author={Case, Jeffrey S.},
   author={Malchiodi, Andrea},
   title={A factorization of the GJMS operators of special Einstein products
   and applications},
   journal={J. Lond. Math. Soc. (2)},
   volume={110},
   date={2024},
   number={5},
   pages={Paper No. e70023, 17},
   issn={0024-6107},
}

\bib{KC Chang}{book}{
   author={Chang, Kung-ching},
   title={Infinite-dimensional Morse theory and multiple solution problems},
   series={Progress in Nonlinear Differential Equations and their
   Applications},
   volume={6},
   publisher={Birkh\"auser Boston, Inc., Boston, MA},
   date={1993},
   pages={x+312},
   isbn={0-8176-3451-7},
}

\bib{ChenHou}{article}{
   author={Chen, Xuezhang},
   author={Hou, Fei},
   title={Remarks on GJMS operator of order six},
   journal={Pacific J. Math.},
   volume={289},
   date={2017},
   number={1},
   pages={35--70},
}

\bib{Deimling}{book}{
   author={Deimling, Klaus},
   title={Ordinary differential equations in Banach spaces},
   series={Lecture Notes in Mathematics},
   volume={Vol. 596},
   publisher={Springer-Verlag, Berlin-New York},
   date={1977},
   pages={vi+137},
   review={\MR{0463601}},
}

\bib{DjaMal}{article}{
   author={Djadli, Zindine},
   author={Malchiodi, Andrea},
   title={Existence of conformal metrics with constant $Q$-curvature},
   journal={Ann. of Math. (2)},
   volume={168},
   date={2008},
   number={3},
   pages={813--858},
   issn={0003-486X},
}

\bib{Dold}{book}{
   author={Dold, Albrecht},
   title={Lectures on algebraic topology},
   series={Classics in Mathematics},
   note={Reprint of the 1972 edition},
   publisher={Springer-Verlag, Berlin},
   date={1995},
   pages={xii+377},
   isbn={3-540-58660-1},
}

\bib{FefGra1}{article}{
   author={Fefferman, Charles},
   author={Graham, C. Robin},
   title={Conformal invariants},
   note={The mathematical heritage of \'Elie Cartan (Lyon, 1984)},
   journal={Ast\'erisque},
   date={1985},
   pages={95--116},
   issn={0303-1179},
}

\bib{FefGra2}{book}{
   author={Fefferman, Charles},
   author={Graham, C. Robin},
   title={The ambient metric},
   series={Annals of Mathematics Studies},
   volume={178},
   publisher={Princeton University Press, Princeton, NJ},
   date={2012},
   pages={x+113},
   isbn={978-0-691-15313-1},
}

\bib{FefGra3}{article}{
   author={Fefferman, Charles},
   author={Graham, C. Robin},
   title={Juhl's formulae for GJMS operators and $Q$-curvatures},
   journal={J. Amer. Math. Soc.},
   volume={26},
   date={2013},
   number={4},
   pages={1191--1207},
   issn={0894-0347},
}

\bib{Gamara}{article}{
   author={Gamara, Najoua},
   title={The CR Yamabe conjecture---the case $n=1$},
   journal={J. Eur. Math. Soc. (JEMS)},
   volume={3},
   date={2001},
   number={2},
   pages={105--137},
   issn={1435-9855},
}

\bib{GamYac}{article}{
   author={Gamara, Najoua},
   author={Yacoub, Ridha},
   title={CR Yamabe conjecture---the conformally flat case},
   journal={Pacific J. Math.},
   volume={201},
   date={2001},
   number={1},
   pages={121--175},
   issn={0030-8730},
}



\bib{GongKimWei}{article}{
      author={Gong, Liuwei},
      author={Kim, Seunghyeok},
      author={Wei, Juncheng},
      title={Compactness and non-compactness theorems of the fourth- and sixth-order constant $Q$-curvature problems}, 
      date={2025},
      note={arXiv: 2502.14237},
}


\bib{GongLeeWei}{article}{
      title={Global Convergence of the Gursky-Malchiodi $Q$-curvature Flow}, 
      author={Gong, Liuwei},
      author={Lee, Sanghoon},
      author={Wei, Juncheng},
      date={2026},
      note={arXiv: 2602.04267},
}

\bib{Gov}{article}{
   author={Gover, A. Rod},
   title={Laplacian operators and $Q$-curvature on conformally Einstein
   manifolds},
   journal={Math. Ann.},
   volume={336},
   date={2006},
   number={2},
   pages={311--334},
   issn={0025-5831},
}

\bib{GovPete}{article}{
   author={Gover, A. Rod},
   author={Peterson, Lawrence J.},
   title={Conformally invariant powers of the Laplacian, $Q$-curvature, and
   tractor calculus},
   journal={Comm. Math. Phys.},
   volume={235},
   date={2003},
   number={2},
   pages={339--378},
   issn={0010-3616},
}

\bib{GJMS}{article}{
   author={Graham, C. Robin},
   author={Jenne, Ralph},
   author={Mason, Lionel J.},
   author={Sparling, George A. J.},
   title={Conformally invariant powers of the Laplacian. I. Existence},
   journal={J. London Math. Soc. (2)},
   volume={46},
   date={1992},
   number={3},
   pages={557--565},
   issn={0024-6107},
}

\bib{GZ}{article}{
   author={Graham, C. Robin},
   author={Zworski, Maciej},
   title={Scattering matrix in conformal geometry},
   journal={Invent. Math.},
   volume={152},
   date={2003},
   number={1},
   pages={89--118},
   issn={0020-9910},
}

\bib{Gun}{article}{
   author={G\"unther, Matthias},
   title={Conformal normal coordinates},
   journal={Ann. Global Anal. Geom.},
   volume={11},
   date={1993},
   number={2},
   pages={173--184},
   issn={0232-704X},
}

\bib{GurHangLin}{article}{
   author={Gursky, Matthew J.},
   author={Hang, Fengbo},
   author={Lin, Yueh-Ju},
   title={Riemannian manifolds with positive Yamabe invariant and Paneitz
   operator},
   journal={Int. Math. Res. Not. IMRN},
   date={2016},
   number={5},
   pages={1348--1367},
   issn={1073-7928},
}

\bib{GurMal}{article}{
   author={Gursky, Matthew J.},
   author={Malchiodi, Andrea},
   title={A strong maximum principle for the Paneitz operator and a
   non-local flow for the $Q$-curvature},
   journal={J. Eur. Math. Soc. (JEMS)},
   volume={17},
   date={2015},
   number={9},
   pages={2137--2173},
   issn={1435-9855},
}


\bib{HangYang1}{article}{
   author={Hang, Fengbo},
   author={Yang, Paul C.},
   title={Sign of Green's function of Paneitz operators and the $Q$
   curvature},
   journal={Int. Math. Res. Not. IMRN},
   date={2015},
   number={19},
   pages={9775--9791},
   issn={1073-7928},
}

\bib{HangYang2}{article}{
   author={Hang, Fengbo},
   author={Yang, Paul C.},
   title={$Q$-curvature on a class of manifolds with dimension at least 5},
   journal={Comm. Pure Appl. Math.},
   volume={69},
   date={2016},
   number={8},
   pages={1452--1491},
   issn={0010-3640},
}

\bib{Hebey1}{book}{
   author={Hebey, Emmanuel},
   title={Compactness and stability for nonlinear elliptic equations},
   series={Zurich Lectures in Advanced Mathematics},
   publisher={European Mathematical Society (EMS), Z\"urich},
   date={2014},
   pages={x+291},
   isbn={978-3-03719-134-7},
}

\bib{HebeyRob}{article}{
   author={Hebey, Emmanuel},
   author={Robert, Fr\'ed\'eric},
   title={Coercivity and Struwe's compactness for Paneitz type operators
   with constant coefficients},
   journal={Calc. Var. Partial Differential Equations},
   volume={13},
   date={2001},
   number={4},
   pages={491--517},
   issn={0944-2669},
}


\bib{HumPetPre}{article}{
      title={Extremising eigenvalues of the GJMS operators in a fixed conformal class}, 
      author={Humbert, Emmanuel},
      author={Petrides, Romain},
      author={Premoselli, Bruno},
      year={2025},
      note={arXiv: 2505.08280},
}

\bib{HumRau}{article}{
   author={Humbert, Emmanuel},
   author={Raulot, Simon},
   title={Positive mass theorem for the Paneitz-Branson operator},
   journal={Calc. Var. Partial Differential Equations},
   volume={36},
   date={2009},
   number={4},
   pages={525--531},
   issn={0944-2669},
}

\bib{JuhlBook}{book}{
   author={Juhl, Andreas},
   title={Families of conformally covariant differential operators,
   $Q$-curvature and holography},
   series={Progress in Mathematics},
   volume={275},
   publisher={Birkh\"auser Verlag, Basel},
   date={2009},
   pages={xiv+488},
   isbn={978-3-7643-9899-6},
   review={\MR{2521913}},
}

\bib{JuhlGJMS}{article}{
   author={Juhl, Andreas},
   title={Explicit formulas for GJMS-operators and $Q$-curvatures},
   journal={Geom. Funct. Anal.},
   volume={23},
   date={2013},
   number={4},
   pages={1278--1370},
   issn={1016-443X},
   review={\MR{3077914}},
}

\bib{Kallel et el}{article}{
   author={Kallel, Sadok},
   author={Karoui, Rym},
   title={Symmetric joins and weighted barycenters},
   journal={Adv. Nonlinear Stud.},
   volume={11},
   date={2011},
   number={1},
   pages={117--143},
   issn={1536-1365},
}

\bib{LeePar}{article}{
   author={Lee, John M.},
   author={Parker, Thomas H.},
   title={The Yamabe problem},
   journal={Bull. Amer. Math. Soc. (N.S.)},
   volume={17},
   date={1987},
   number={1},
   pages={37--91},
   issn={0273-0979},
}




\bib{LLL}{article}{
   author={Li, Jiayu},
   author={Li, Yuxiang},
   author={Liu, Pan},
   title={The $Q$-curvature on a 4-dimensional Riemannian manifold $(M,g)$
   with $\int_MQdV_g=8\pi^2$},
   journal={Adv. Math.},
   volume={231},
   date={2012},
   number={3-4},
   pages={2194--2223},
   issn={0001-8708},
}

\bib{Lions}{article}{
   author={Lions, Pierre-Louis},
   title={The concentration-compactness principle in the calculus of
   variations. The limit case. II},
   journal={Rev. Mat. Iberoamericana},
   volume={1},
   date={1985},
   number={2},
   pages={45--121},
   issn={0213-2230},
}

\bib{Martin}{article}{
   author={Martin, Robert H., Jr.},
   title={Differential equations on closed subsets of a Banach space},
   journal={Trans. Amer. Math. Soc.},
   volume={179},
   date={1973},
   pages={399--414},
   issn={0002-9947},
}

\bib{MayNdi1}{article}{
   author={Mayer, Martin},
   author={Ndiaye, Cheikh Birahim},
   title={Barycenter technique and the Riemann mapping problem of
   Cherrier-Escobar},
   journal={J. Differential Geom.},
   volume={107},
   date={2017},
   number={3},
   pages={519--560},
   issn={0022-040X},
}

\bib{MayNdi2}{article}{
   author={Mayer, Martin},
   author={Ndiaye, Cheikh Birahim},
   title={Fractional Yamabe problem on locally flat conformal infinities of
   Poincar\'{e}-Einstein manifolds},
   journal={Int. Math. Res. Not. IMRN},
   date={2024},
   number={3},
   pages={2561--2621},
}


\bib{Maz1}{article}{
   author={Mazumdar, Saikat},
   title={GJMS-type operators on a compact Riemannian manifold: best constants and Coron-type solutions},
   journal={J. Differential Equations},
   volume={261},
   date={2016},
   number={9},
   pages={4997--5034},
}

\bib{Maz2}{article}{
   author={Mazumdar, Saikat},
   title={Struwe's decomposition for a polyharmonic operator on a compact
   Riemannian manifold with or without boundary},
   journal={Commun. Pure Appl. Anal.},
   volume={16},
   date={2017},
   number={1},
   pages={311--330},
   issn={1534-0392},
}

\bib{MazPrem}{article}{
      title={Compactness of conformal metrics with constant $Q$-curvature of higher order}, 
      author={Mazumdar, Saikat},
      author={Premoselli, Bruno},
      year={2025},
      note={arXiv: 2510.00888},
}

\bib{MazVetois}{article}{
   author={Mazumdar, Saikat},
   author={V\'etois, J\'er\^ome},
   title={Existence results for the higher-order $Q$-curvature equation},
   journal={Calc. Var. Partial Differential Equations},
   volume={63},
   date={2024},
   number={6},
   pages={Paper No. 151, 29},
}

\bib{Mic}{article}{
   author={Michel, Benoît},
   title={Masse des op\'erateurs GJMS},
   date={2010}, 
   note={arXiv: 1012.4414},
}

\bib{Ndiaye1}{article}{
   author={Ndiaye, Cheikh Birahim},
   title={Constant $Q$-curvature metrics in arbitrary dimension},
   journal={J. Funct. Anal.},
   volume={251},
   date={2007},
   number={1},
   pages={1--58},
   issn={0022-1236},
}

\bib{Ndiaye2}{article}{
   author={Ndiaye, Cheikh Birahim},
   title={Algebraic topological methods for the supercritical $Q$-curvature
   problem},
   journal={Adv. Math.},
   volume={277},
   date={2015},
   pages={56--99},
   issn={0001-8708},,
}

\bib{Ndiaye3}{article}{
   author={Ndiaye, Cheikh Birahim},
   title={Topological methods for the resonant $Q$-curvature problem in
   arbitrary even dimension},
   journal={J. Geom. Phys.},
   volume={140},
   date={2019},
   pages={178--213},
   issn={0393-0440},
}


\bib{NSS}{article}{
   author={Ndiaye, Cheikh Birahim},
   author={Sire, Yannick},
   author={Sun, Liming},
   title={Uniformization theorems: between Yamabe and Paneitz},
   journal={Pacific J. Math.},
   volume={314},
   date={2021},
   number={1},
   pages={115--159},
   issn={0030-8730},
}



%
%

\bib{Paneitz}{article}{
      author={Paneitz, Stephen~M.},
       title={A quartic conformally covariant differential operator for
  arbitrary pseudo-{R}iemannian manifolds (summary)},
        date={2008},
     journal={SIGMA Symmetry Integrability Geom. Methods Appl.},
      volume={4},
       pages={Paper 036, 3},
         url={https://doi.org/10.3842/SIGMA.2008.036},
      review={\MR{2393291}},
}

\bib{Prem13}{article}{
      author={Premoselli, Bruno},
       title={A priori estimates for finite-energy sign-changing blowing-up
  solutions of critical elliptic equations},
        date={2024},
        ISSN={1073-7928,1687-0247},
     journal={Int. Math. Res. Not. IMRN},
      number={6},
       pages={5212\ndash 5273},
}

\bib{QingRas}{article}{
   author={Qing, Jie},
   author={Raske, David},
   title={On positive solutions to semilinear conformally invariant
   equations on locally conformally flat manifolds},
   journal={Int. Math. Res. Not.},
   date={2006},
   pages={Art. ID 94172, 20},
   issn={1073-7928},
}


\bib{RobertGJMS}{article}{
   author={Robert, Fr\'ed\'eric},
   title={Admissible $Q$-curvatures under isometries for the conformal GJMS
   operators},
   conference={
      title={Nonlinear elliptic partial differential equations},
   },
   book={
      series={Contemp. Math.},
      volume={540},
      publisher={Amer. Math. Soc., Providence, RI},
   },
   isbn={978-0-8218-4907-1},
   date={2011},
   pages={241--259},
}


\bib{RobPoly1}{article}{
   author={Robert, Fr\'ed\'eric},
   title={Localization of bubbling for high order nonlinear equations}, 
   year={2024},
   note={arXiv: 2501.00531},
}

\bib{RobPoly2}{article}{
    title={Sharp multiscale control for high order nonlinear equations}, 
    author={Robert, Fr\'ed\'eric},
    year={2025},
    note={arXiv: 2509.07517},
}

\bib{Schoen}{article}{
   author={Schoen, Richard},
   title={Conformal deformation of a Riemannian metric to constant scalar
   curvature},
   journal={J. Differential Geom.},
   volume={20},
   date={1984},
   number={2},
   pages={479--495},
}

\bib{SchoenYau}{article}{
   author={Schoen, Richard},
   author={Yau, Shing Tung},
   title={On the proof of the positive mass conjecture in general
   relativity},
   journal={Comm. Math. Phys.},
   volume={65},
   date={1979},
   number={1},
   pages={45--76},
}

\bib{Struwe}{article}{
   author={Struwe, Michael},
   title={A global compactness result for elliptic boundary value problems
   involving limiting nonlinearities},
   journal={Math. Z.},
   volume={187},
   date={1984},
   number={4},
   pages={511--517},
   issn={0025-5874},
}

\bib{Swan}{article}{
   author={Swanson, Charles A.},
   title={The best Sobolev constant},
   journal={Appl. Anal.},
   volume={47},
   date={1992},
   number={4},
   pages={227--239},
   issn={0003-6811},
}

\bib{WeiXu}{article}{
      author={Wei, Juncheng},
      author={Xu, Xingwang},
       title={Classification of solutions of higher order conformally invariant
  equations},
        date={1999},
        ISSN={0025-5831,1432-1807},
     journal={Math. Ann.},
      volume={313},
      number={2},
       pages={207\ndash 228},
}

\bib{Yamabe}{article}{
   author={Yamabe, Hidehiko},
   title={On a deformation of Riemannian structures on compact manifolds},
   journal={Osaka Math. J.},
   volume={12},
   date={1960},
   pages={21--37},
}
\end{thebibliography}
\end{document}